\documentclass[final,1p,times]{elsarticle}
\usepackage{amsmath}
\usepackage{verbatim}  
\allowdisplaybreaks[4]
\usepackage{amsthm}
\usepackage{amscd}
\RequirePackage{CJKnumb}
\usepackage{amsfonts}
\usepackage{amssymb}
\usepackage{graphicx}
\usepackage{subcaption} 
\usepackage{caption}

\usepackage{booktabs}
\biboptions{sort&compress}
\usepackage{xcolor}
\usepackage[pagebackref, colorlinks, citecolor=green, linkcolor=red, urlcolor=blue]{hyperref}
\usepackage{pgfplots}
\pgfplotsset{compat=newest} 
\numberwithin{equation}{section}
\newtheorem{theorem}{Theorem}[section]
\usepackage{mathrsfs}
\usepackage{titletoc}
\usepackage{float}
\usepackage{longtable}
\usepackage{adjustbox}
\newcommand{\be}{\begin{equation}}
	\newcommand{\ee}{\end{equation}}	
\theoremstyle{remark} 

\journal{***}
\begin{document}
	\begin{frontmatter}
		\title{Stability and Hopf bifurcation analysis of an HIV infection model with latent reservoir, immune impairment and delayed CTL immune response}	
		\author{Songbo Hou \corref{cor1}}
		\ead{housb@cau.edu.cn}
		\address{Department of Applied Mathematics, College of Science, China Agricultural University,  Beijing, 100083, P.R. China}
		\author{Xinxin Tian}
		\ead{txx@cau.edu.cn}
		\address{Department of Applied Mathematics, College of Science, China Agricultural University,  Beijing, 100083, P.R. China}
		
		\cortext[cor1]{Corresponding author: Songbo Hou}
		\begin{abstract}
			In this paper, we develop a dynamic model of HIV infection that incorporates latent hosts, cytotoxic T lymphocyte (CTL) immunity, saturated incidence rates, and two transmission mechanisms: virus-to-cell and cell-to-cell transmission. The model has three kinds of delays: intracellular delay, replication of viruses delay, immune response delay. Initially, the model's solutions are confirmed to be both nonnegative and bounded for nonnegative initial values. Subsequently, two biologically critical parameters were identified: the virus reproduction number $\mathcal{R}_0$ and the immune reproduction number $\mathcal{R}_1$. Thereafter, by invoking LaSalle's principle of invariance alongside Lyapunov functionals, we establish stability criteria for each equilibrium. The results indicate that the stability of the endemic equilibrium may be altered by a positive immune delay, whereas intracellular and viral replication delays do not affect the equilibria. By considering the delay in the immune response as a bifurcation-inducing threshold, we derive the exact conditions necessary for these stability transitions. Further analysis shows that increasing the immune delay destabilizes the endemic equilibrium, inducing a Hopf bifurcation. Additionally, using the center manifold theorem and normal  form theory, we explored the direction and stability of Hopf bifurcations in detail. To corroborate these theoretical results, numerical simulations are systematically conducted.
			
		\end{abstract}	
	\end{frontmatter}
	
	\section{Introduction}
As a pathogenic retrovirus, HIV compromises the human immune system by depleting CD4+ T lymphocytes, a cell population vital to sustaining immunological equilibrium. HIV invades the immune system by integrating its DNA into host cells, gradually impairing their function, making it difficult for the body to defend against pathogens, and ultimately potentially leading to AIDS, a severe global immunodeficiency disease that poses a significant threat to public health. Therefore, studying the pathogenesis, transmission, and prevention and control strategies of AIDS is extremely important.
	
During the process of HIV infection, the virus exhibits uneven distribution within the body and often selectively hides in specific cells or tissues, forming "latent reservoirs" to evade immune system surveillance \cite{WOS:000268071100035}. Although antiretroviral therapies such as HAART have significantly reduced HIV-related complications and mortality \cite{WOS:000086586600009,WOS:000072669800001}, the virus can still persist in a latent state, and viral loads can easily rebound after treatment interruption \cite{WOS:000079421400007,WOS:000167121800002}. After the widespread application of ART, low viral levels and latently infected cells are still seen in patients, potentially due to the ongoing release of virus from the activation of latent cells \cite{WOS:000084149700059}. Therefore, to gain a deeper understanding of HIV infection dynamics, recent research has begun to incorporate the mechanisms of latent infection cell activation into mathematical model analysis \cite{WOS:000395717900003,WOS:000074861500009,WOS:000177334900047,WOS:000262804700008,WOS:000272033100013,WOS:000539042000018,Hmarrass2025}. Furthermore, while previous studies have considered HIV transmission and some aspects of the immune system's response, they have neglected the immune response of CTLs. CTLs are activated during viral infection, recognise and eliminate infected cells, reduce viral load, and provide a protective mechanism \cite{WOS:A1996UD59700039,WOS:000266271700006,WOS:000278768800055,BELL}. The production of CTLs is linked to the quantity of previously active, infected cells, and involves a time delay. In recent years, HIV infection models that incorporate CTL immune responses and time delay factors have become a research focus. Within the framework of delay differential equations, these investigations examine how time delays influence model dynamics \cite{WOS:000352560500008,WOS:001347333700002,WOS:000513551300002,WOS:000441251900006,WOS:000755088100017,WOS:000244572500009,WOS:000641256000022}.
	
	At the same time, previous studies have simplified the potential damage that antigens can cause to the immune system. In reality, various pathogens can suppress immune responses. Studies by Iwami et al. \cite{WOS:000266713900008} have found that HIV infection can lead to damage to CTL cells. To address HIV-induced immune dysregulation, researchers have undertaken extensive studies \cite{WOS:000781901400009,WOS:000181073000075,WOS:000397808600004,WOS:A1990ER29800032,WOS:000292176200016}, shedding light on the complex interplay between viral mechanisms and host immunity. Wang et al. \cite{WOS:000292176200016} introduced a mathematical framework for viral pathogenesis, focusing on how immune dysregulation influences disease progression. Xu et al. \cite{WOS:000656821600024} considered a mathematical model comprising uninfected cells $x$, latently infected cells $p$, infected cells $y$, and CTL immune response cells $z$. The model is structured as follows:
	\begin{equation}\label{1.1}
		\begin{aligned}\left\{\begin{aligned}
				\frac{dx(t)}{dt}=&\lambda-\beta_{2}x(t)y(t)-\mu_1x(t),\\
				\frac{dp(t)}{dt}=&\rho \beta_{2}x(t)y(t)-\alpha p(t)-\mu_{2}p(t),\\
				\frac{dy(t)}{dt}=&(1-\rho)\beta_{2}x(t)y(t)+\alpha p(t)-\mu_{3}y(t)-ay(t)z(t),\\
				\frac{dz(t)}{dt}=&cy(t-\tau)-\mu_{5}z(t)-\eta y(t)z(t).
			\end{aligned}\right.\end{aligned}
	\end{equation}
Within the framework of the model (\ref{1.1}), the rate of generation of uninfected cells is characterized by $\lambda$, while $\mu_1$ signifies the rate at which these cells are eliminated. $\beta_2$ is the coefficient for the infection rate, indicating the rate that infected cells infect uninfected cells. Defined by $\rho \in (0,1)$, the terms $\rho$ and $1-\rho$ modulate the distribution of infections between inactive and active forms. Furthermore, $\alpha$ governs the transition from latent to actively infected cells, while $\mu_2$ determines the mortality of cells during the latent phase. Infected cells decay at a rate proportional to $\mu_3 y(t)$ and are eliminated through CTL-mediated immune activity proportionally to $a y(t)z(t)$. Additionally, $c$ quantifies the proliferation rate of CTL cells, $\eta$ captures the rate of immune impairment, CTL cell degradation is described by $\mu_5$, and $\tau$ denotes a temporal delay associated with the CTL immune response.

In the study of HIV infection dynamics, the bilinear activation hypothesis of CTL immune response adopted in system (\ref{1.1}) exhibits certain biological limitations. Under actual immunological conditions, when the concentrations of infected cells or CTL cells exceed a threshold, the immune response demonstrates saturation characteristics where the reaction rate no longer increases linearly with cell numbers. This phenomenon has been well documented in multiple studies \cite{WOS:000336253100008,WOS:000693357100001,Mo2022}. Although various functional forms (such as Hill functions or Michaelis-Menten kinetics) can be employed to characterize CTL responses, the simple saturation form $\frac{cyz}{(h+z)}$ has been widely adopted due to its parameter economy and biological plausibility\cite{Do,WOS:A1996UD59700039}. Notably, systematic investigations by Wodarz and Nowak \cite{Do} demonstrated that this formulation can accurately predict viral rebound dynamics following treatment interruption. Furthermore, existing models that consider only a single CTL immune time delay may produce deviations in predicting immune responses and viral load variations due to the neglect of other temporal factors. Notably, research on HIV transmission mechanisms \cite{WOS:000487331700033,WOS:000081497400034,WOS:000405724700011,WOS:000412253200041} has clearly revealed the critical importance of dual infection pathways: both direct cell-to-cell transmission and cell-free viral particle transmission collectively drive the infection process, yet system (\ref{1.1}) fails to comprehensively incorporate the synergistic effects of these two transmission pathways.

To address the aforementioned issues, this study will replace the bilinear incidence rate in system (\ref{1.1}) with a saturated incidence rate function. Simultaneously, we will integrate dual transmission pathways and introduce triple time delays encompassing intracellular delay, viral replication delay, and CTL immune response delay. This extended framework provides enhanced capability to elucidate the cooperative effects between multiple time delays and nonlinear immune responses on HIV infection dynamics, with the specific modeling procedures detailed as follows:
\begin{equation}\label{1.2}
		\begin{aligned}\left\{\begin{aligned}
				\frac{dx(t)}{dt}=&\lambda-x(t)[\beta_{1}v(t)+\beta_{2}y(t)]-\mu_1x(t),\\
				\frac{dp(t)}{dt}=&\rho e^{-m_1\tau_1}x(t-\tau_1)[\beta_{1}v(t-\tau_1)+\beta_{2}y(t-\tau_1)]-\alpha p(t)-\mu_{2}p(t),\\
				\frac{dy(t)}{dt}=&(1-\rho)e^{-m_2\tau_{2}}x(t-\tau_2)[\beta_{1}v(t-\tau_{2})+\beta_{2}y(t-\tau_{2})]+\alpha p(t)-\mu_{3}y(t)-ay(t)z(t),\\
				\frac{dv(t)}{dt}=&ky(t)-\mu_{4}v(t),\\
				\frac{dz(t)}{dt}=&\frac{cy(t-\tau_3)z(t-\tau_3)}{h+z(t-\tau_3)}-\mu_{5}z(t)-\eta y(t)z(t),
			\end{aligned}\right.\end{aligned}
	\end{equation}
where $v(t)$ denotes the amount of free virus particles existing at time $t$, the model assumes that uninfected cells get infected by free virus particles at a rate of $\beta_1 x(t)v(t)$ and by infected cells at a rate of $\beta_2 x(t)y(t)$. Additionally, the model incorporates a delay $\tau_1$, which represents the duration between viral entrance into an uninfected cell and the subsequent establishment of latent infection within the cell. Consequently, the probability of cell survival over the interval $[t-\tau_1, t]$ is given by $e^{-m_1\tau_1}$. Furthermore, $\tau_2$ denotes the delay in virus replication, and $e^{-m_2\tau_2}$ represents the survival probability of the virus over the delay interval $[t - \tau_2, t]$. Infected cells generate free virus particles at a rate of $ky(t)$, which then decay at a rate of $\mu_4 v(t)$. The term $\frac{c y(t)z(t)}{h+z(t)}$ captures the phenomenon where CTL immune responses reach saturation, driven by the presence of infected cells, and  $\tau_3$ represents the time delay in the CTL immune response. We assume that $c > \eta h$ and that all parameters in Eq.(\ref{1.2}) are positive. To better describe this phenomenon, we have a schematic diagram of model (\ref{1.2}) in Figure 1 as follows.

\begin{figure}[htbp]
	\centering
	\includegraphics[width=0.8\textwidth]{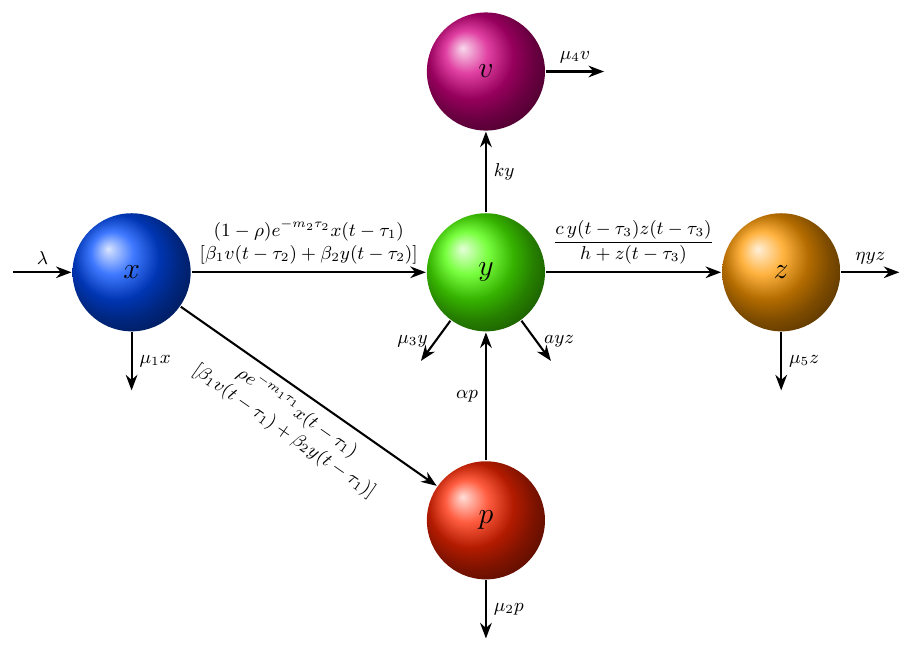}
	\caption{The pictorial representation of proposed model}
	\label{fig:image1}
\end{figure}

	\vskip 0.2cm The following initial conditions are taken into consideration:
	\\
	\begin{equation}\label{1.3}\begin{aligned}
			&x(\theta)=\pi_1(\theta),\qquad p(\theta)=\pi_2(\theta),\qquad y(\theta)=\pi_3(\theta),\\
			&v(\theta)=\pi_4(\theta),\qquad z(\theta)=\pi_5(\theta),\\
			&\pi_{j}(\theta)\geq0,\qquad\theta\in[-\tau,0],\\
			&\pi_{j}\in C\Big([-\tau,0],\mathbb{R}_{\geq0}\Big),\quad j=1,\ldots,5,
		\end{aligned}	
	\end{equation}
where $\tau$=$\max\left\{\tau_1, \tau_2,\tau_3\right\}$ and $C$ is the Banach space of continuous functions mapping the interval $[-\tau,0] $ into $\mathbb{R}_{\geq0}$ with norm $\left\|\pi_j\right\|=\sup_{-\tau\leq\theta\le 0} \left |\pi_j(\theta)\right|$. According to \cite{hale1993introduction}, system (\ref{1.2}) possesses a unique solution for $t > 0$.

The structure of this paper is outlined below: In the forthcoming section, we examine the key properties of model (\ref{1.2}), addressing its solution's boundedness and non-negativity. Additionally, we will determine the equilibrium points and calculate the fundamental reproduction number. The third section offers an in-depth exploration of the local and global asymptotic stability of the uninfected equilibrium $E_0$, the non-immune equilibrium $E_1$, and the infected equilibrium with CTL immune response $E_2$. Section 4 explores the sufficient conditions for the occurrence of an $E_2$ Hopf bifurcation. Through numerical simulations in Section 5, we dissect the system's dynamical behavior. Finally, we will present a discussion and conclusion.
	\section{Preliminaries}
	In this part, we investigate the qualitative behavior of Eq.(\ref{1.2}), including the proposed model's boundedness and positivity. Following that, we utilize the methodologies presented in references \cite{WOS:000441251900006} and \cite{WOS:000179220600004} to determine the viral reproductive number $\mathcal{R}_0$ with the CTL immune response reproductive number $\mathcal{R}_1$. Furthermore, with these values as our foundation, we analyze the equilibrium states of the model.
	\subsection{Positivity}
	\vskip 0.2cm
	\newtheorem{thm}{\bf Theorem}[section]
	\begin{thm}\label{thm1}
		All solutions of system (\ref{1.2}) with positive initial conditions always stay positive.
	\end{thm} 
	\begin{proof} 
		For all $t \geq 0$, define $m(t) =\min\left\{x(t), p(t), y(t), v(t), z(t)\right\}$. Given that the starting values are positive, it follows that $m(0)>0$. To establish that $ m(t)$ is positive for all $t \geq 0$, suppose, contrary to our claim, the system does not maintain positivity. Consequently, there must exist a point in time, denoted as $t_1 > 0$, where $m(t)$ is positive for all $0\le t<t_1$, and $m(t)=0$ at $t = t_1$. By analyzing the behavior of
		$m(t_1)$, we can derive the following five cases:
		\vskip 0.2cm (1) If $m(t_1) = x(t_1) = 0$, we deduce from the initial equation of system (\ref{1.2})
		$$
		\begin{aligned}
			\frac{d x(t)}{dt}&=\lambda-x(t)[\beta_1 v(t)+\beta_2 y(t)]-\mu_1 x(t)\\
			&\geq -\mu_1 x(t)-\max\left\{\beta_1 v(t)\right\}x(t)-\max\left\{\beta_2 y(t)\right\}x(t)\\
			&=-b_1 x(t),
		\end{aligned}$$
		for $t\in[0,t_1]$, where $b_1=\mu_1+\max_{t\in[0,t_1]} \left\{\beta_1 v(t)\right\}+\max_{t\in[0,t_1]} \left\{\beta_2 y(t)\right\}$. Consequently, $x(t_1)\geq x(0)e^{-b_1 t_1}>0$, which contradicts the fact that $x(t_1)=0$.
		\vskip 0.2cm(2) If $m(t_1) = p(t_1) = 0$, considering the second equation in system (\ref{1.2}), we derive that
		$$
		\begin{aligned}
			\frac{d p(t)}{dt}&=\rho e^{-m_1\tau_1}x(t-\tau_1)[\beta_1 v(t-\tau_1)+\beta_2 y(t-\tau_1)]-\alpha p(t)-\mu_2 p(t)\\
			&\geq -\alpha p(t)-\mu_2 p(t)\\
			&=-b_2 p(t),
		\end{aligned}$$
		for $t\in[0,t_1]$, where $b_2=\alpha+\mu_2$. Consequently, $p(t_1)\geq p(0)e^{-b_2 t_1}>0$, which contradicts the fact that $p(t_1)=0$.
		\vskip 0.2cm(3) If $m(t_1) = y(t_1) = 0$, upon inspecting the third equation of system (\ref{1.2}), we find that
		$$
		\begin{aligned}
			\frac{dy(t)}{dt}&=(1-\rho)e^{-m_2 \tau_2}x(t-\tau_2)[\beta_1 v(t-\tau_2)+\beta_2 y(t-\tau_2)]+\alpha p(t)-\mu_3 y(t)-ay(t)z(t)\\
			&\geq -\mu_3 y(t)-\max\left\{az(t)\right\}y(t)\\
			&=-b_3 y(t),
		\end{aligned}
		$$
		for $t\in[0,t_1]$, where $b_3=\mu_3+\max_{t\in[0,t_1]}\left\{az(t)\right\} $. Consequently, $y(t_1)\geq y(0)e^{-b_3 t_1}>0$, which contradicts the fact that $y(t_1)=0$.
		\vskip 0.2cm(4) If $m(t_1)=v(t_1)=0$, the fourth governing equation in system (\ref{1.2}) leads to
		$$
		\begin{aligned}
			\frac{dv(t)}{dt}&=ky(t)-\mu_4 v(t)\\
			&\geq -\mu_4 v(t)\\
			&=-b_4 v(t),
		\end{aligned}
		$$
		for $t\in[0,t_1]$, where $b_4=\mu_4$. Consequently, $v(t_1)\geq v(0)e^{-b_4 t_1}>0$, which contradicts the fact that $v(t_1)=0$.
		\vskip 0.2cm(5)  If $m(t_1)=z(t_1)=0$, from the fifth governing equation in system (\ref{1.2}) we can get
		$$
		\begin{aligned}
			\frac{dz(t)}{dt}&=\frac{cy(t-\tau_3)z(t-\tau_3)}{h+z(t-\tau_3)}-\mu_5 z(t)-\eta y(t)z(t)\\
			&\geq -\mu_5 z(t)-\max_{t\in[0,t_1]}\left\{\eta y(t)\right\} z(t)\\
			&=-b_5 z(t),
		\end{aligned}
		$$
		for $t\in[0,t_1]$, where $b_5=\mu_5+\max_{t\in[0,t_1]}\left\{\eta y(t)\right\}$. Consequently, $z(t_1)\geq z(0)e^{-b_5 t_1}>0$, which is in direct conflict with the condition that $z(t_1)=0$.
		
		In summary, when $t \geq 0$, $x(t) > 0$, $p(t) > 0$, $y(t) > 0$, $v(t) > 0$, and $z(t) > 0$. This completes the proof.
	\end{proof}

	Next, we address the case of non-negative initial values. Following the approach in \cite{cao2024}, we define the following vector field:
	$$	U = (x, p, y, v, z),$$
	and
	$$\dot{U} = \left( \frac{dx}{dt}, \frac{dp}{dt}, \frac{dy}{dt}, \frac{dv}{dt}, \frac{dz}{dt} \right).$$
	
	Consider the function
	$$
	g(t, U_t) = \left(g_1(t, U_t), g_2(t, U_t), g_3(t, U_t), g_4(t, U_t), g_5(t, U_t)\right),$$
	where $ U_t(\theta) = U(t + \theta) $ for $ \theta \in [-\tau, 0]$, and $ g_j(t, U_t)$ represents the right-hand side of the $ j $-th equation in system (\ref{1.2}). Then we can rewrite system (\ref{1.2}) as 
	$\dot{U}=	g(t, U_t) .$
	It is clear that $g(t, U_t)$ satisfies the criteria given in Theorem 2.3 of \cite{cao2024}. Therefore, we establish the following theorem.
	
	\begin{thm}
		All solutions of system (\ref{1.2}) with non-negative initial conditions remain non-negative for all $t \geq 0 $.
	\end{thm}
	\subsection{Boundedness}
	\vskip 0.2cm
	\begin{thm}\label{thm1}
		The solutions of system (\ref{1.2}) are uniformly bounded for $t\ge 0$ under initial condition (\ref{1.3}).
	\end{thm} 
	\begin{proof} Under the starting parameters specified in (\ref{1.3}), we denote by $(x(t), p(t), y(t), v(t), z(t))$ any positive solution to the system (\ref{1.2}). Define
		$$G_1(t)=x(t)+\frac{1}{\rho} e^{m_1\tau_1} p(t+\tau_1).$$
	Given the initial condition (\ref{1.3}), differentiating $G_1(t)$ along the trajectories of the positive solutions in system (\ref{1.2}) allows us to deduce
		$$
		\begin{aligned}
			G_1'(t)&=x'(t)+\frac{1}{\rho} e^{m_1\tau_1} p'(t+\tau_1)\\
			&=\lambda-\mu_1 x(t)-\frac{1}{\rho} e^{m_1\tau_1}(\alpha+\mu_2)p(t+\tau_1)\\
			&\le\lambda-\sigma_1 G_1(t),
		\end{aligned}
		$$
		yielding that $\lim_{t\to\infty}G_1(t)\le\frac{\lambda}{\sigma_1}$, where $\sigma_1=\min\left\{\mu_1,\alpha+\mu_2\right\}$. Consequently, given a sufficiently small $\delta_1 > 0$, there exists a $T_1 > 0$ such that if $t > T_1$, the following holds
		$$G_1(t)\le \frac{\lambda}{\sigma_1}+\delta_1.$$
		Since $G_1(t)=x(t)+\frac{1}{\rho} e^{m_1\tau_1} p(t+\tau_1)\le\frac{\lambda}{\sigma_1}+\delta_1$, it follows that $p(t)\le\rho e^{-m_1\tau_1}(\frac{\lambda}{\sigma_1}+\delta_1)$.\\
		Define$$G_2(t)=x(t)+\frac{1}{1-\rho}e^{m_2\tau_2}y(t+\tau_2).$$
Upon differentiating $G_2(t)$ along the positive trajectories of the system (\ref{1.2}), considering the initial condition (\ref{1.3}), we obtain the derivative 
		$$
		\begin{aligned}
			G_2'(t)&=x'(t)+\frac{1}{1-\rho}e^{m_2\tau_2}y'(t+\tau_2)\\
			&=\lambda-\mu_1 x(t)+\frac{\alpha}{1-\rho}e^{m_2\tau_2}p(t+\tau_2)-\frac{\mu_3}{1-\rho}e^{m_2\tau_2}y(t+\tau_2)-\frac{a}{1-\rho}e^{m_2\tau_2}y(t+\tau_2)z(t+\tau_2)\\
			&\le\lambda+\frac{\alpha\rho}{1-\rho}e^{m_2\tau_2-m_1\tau_1}(\frac{\lambda}{\sigma_1}+\delta_1)-\mu_1 x(t)-\frac{\mu_3}{1-\rho}e^{m_2\tau_2}y(t+\tau_2)\\
			&\le M-\sigma_2 G_2(t),
		\end{aligned}
		$$
		yielding that $\lim_{t\to\infty}G_2(t)\le\frac{M}{\sigma_2}$, where $M=\lambda+\frac{\alpha\rho}{1-\rho}e^{m_2\tau_2-m_1\tau_1}(\frac{\lambda}{\sigma_1}+\delta_1)$, $\sigma_2=\min\left\{\mu_1,\mu_3\right\}$. Consequently, for a sufficiently small $\delta_2 > 0$, there exists a $T_2 > 0$ that means if $t > T_2$, we have
		$$G_2(t)\le \frac{M}{\sigma_2}+\delta_2.$$
		Let $T = \max\left\{T_1, T_2\right\}$. Therefore, when $t > T$, it holds that $$G_1(t) \leq \frac{\lambda}{\sigma_1} + \delta_1,$$ $$G_2(t) \leq \frac{M}{\sigma_2} + \delta_2.$$
		Since $G_2(t) \leq \frac{M}{\sigma_2} + \delta_2$, it follows that $y(t)\le(1-\rho)e^{-m_2\tau_2}(\frac{M}{\sigma_2}+\delta_2).$
		\\
In addition, according to the system's fourth and fifth equations (\ref{1.2}), if $t>T$,
		$$v'(t)\le k(1-\rho)e^{-m_2\tau_2}(\frac{M}{\sigma_2}+\delta_2)-\mu_4 v(t),$$
		$$z'(t)\le c(1-\rho)e^{-m_2\tau_2}(\frac{M}{\sigma_2}+\delta_2)-\mu_5 z(t).$$
		Since $\delta_1 > 0$, $\delta_2 > 0$ can be arbitrarily small positive numbers, we therefore deduce that
		$$\lim_{t\to\infty}\sup v(t)\le\frac{k(1-\rho)M}{\mu_4 \sigma_2} e^{-m_2\tau_2}\le \frac{k(1-\rho)\lambda}{\mu_4 \sigma_2}e^{-m_2\tau_2}+\frac{k\alpha\rho\lambda}{\mu\sigma_2\sigma_1}e^{-m_1\tau_1},$$
		$$\lim_{t\to\infty}\sup z(t)\le\frac{c(1-\rho)M}{\mu_5\sigma_2}e^{-m_2\tau_2}\le\frac{c(1-\rho)\lambda}{\mu_5\sigma_2}e^{-m_2\tau_2}+\frac{c\alpha\rho\lambda}{\mu_5\sigma_2\sigma_1}e^{-m_1\tau_1}.$$
		Hence, for all $t \ge 0$, $x(t)$, $p(t)$, $y(t)$, $v(t)$, and $z(t)$ are ultimately bounded, and the following set $\Omega$ is positively invariant for the system (\ref{1.2}): $\Omega=\Bigg\{(x,p,y,v,z):x\le$
		$\min\Big\{\frac{\lambda}{\sigma_1},\frac{\lambda}{\sigma_2}+\frac{\alpha\rho\lambda}{(1-\rho)\sigma_1\sigma_2}e^{m_2\tau_2-m_1\tau_1}\Big\}, p\le\rho\frac{\lambda}{\sigma_1}e^{-m_1\tau_1}, y\le\frac{(1-\rho)\lambda}{\sigma_2}e^{-m_2\tau_2}+\frac{\alpha\rho\lambda}{\sigma_2\sigma_1}e^{-m_1\tau_1}, v\le\frac{k(1-\rho)\lambda}{\mu_4 \sigma_2}e^{-m_2\tau_2}+\frac{k\alpha\rho\lambda}{\mu\sigma_2\sigma_1}e^{-m_1\tau_1},$ 
		$z\le\frac{c(1-\rho)\lambda}{\mu_5\sigma_2}e^{-m_2\tau_2}+\frac{c\alpha\rho\lambda}{\mu_5\sigma_2\sigma_1}e^{-m_1\tau_1}\Bigg\}.$ 
	\end{proof}
	\subsection{Viable equilibria and reproductive ratios}
	
	In the following, we compute the steady states and use the next-generation approach \cite{WOS:000441251900006,WOS:000179220600004} to derive the basic reproduction number from system (\ref{1.2}). First, we specify the matrices $\mathbb{F}$ and $\mathbb{V}$ as:
	$$\left.\mathbb{F}=\left(\begin{matrix}0&\rho\beta_2 e^{-m_1\tau_1} x_0&\rho\beta_1 e^{-m_1\tau_1}x_0\\0&(1-\rho)\beta_2 e^{-m_2\tau_2}x_0&(1-\rho)\beta_1 e^{-m_2\tau_2}x_0\\0&0&0\end{matrix}\right.\right),\qquad\left.\mathbb{V}=\left(\begin{matrix}\alpha+\mu_2&0&0\\-\alpha&u_{3}&0\\0&-k&\mu_4\end{matrix}\right.\right),$$
	where $x_0=\frac{\lambda}{\mu_1}$. Then
	$$\left.\mathbb{FV}^{-1}=\left(\begin{matrix}\psi_1&\psi_2&\psi_3\\\psi_4&\psi_5&\psi_6\\0&0&0\end{matrix}\right.\right),$$
	where
	$$
	\begin{aligned}
		&\psi_1=\frac{\alpha\rho\beta_2 e^{-m_1\tau_1}}{\mu_3(\alpha+\mu_2)}x_0+\frac{\alpha k\rho\beta_1 e^{-m_1\tau_1}}{\mu_3\mu_4(\alpha+\mu_2)}x_0,\\
		&\psi_2=\frac{\rho\beta_2 e^{-m_1\tau_1}}{\mu_3}x_0+\frac{ k\rho\beta_1 e^{-m_1\tau_1}}{\mu_3\mu_4}x_0,\\
		&\psi_3=\frac{\rho\beta_1 e^{-m_1\tau_1}}{\mu_4}x_0,\\
		&\psi_4=\frac{\alpha(1-\rho)\beta_2 e^{-m_2\tau_2}}{\mu_3(\alpha+\mu_2)}x_0+\frac{\alpha k(1-\rho)\beta_1 e^{-m_2\tau_2}}{\mu_3\mu_4(\alpha+\mu_2)}x_0,\\
		&\psi_5=\frac{(1-\rho)\beta_2 e^{-m_2\tau_2}}{\mu_3(\alpha+\mu_2)}x_0+\frac{ k(1-\rho)\beta_1 e^{-m_2\tau_2}}{\mu_3\mu_4(\alpha+\mu_2)}x_0,\\
		&\psi_6=\frac{(1-\rho)\beta_1 e^{-m_2\tau_2}}{\mu_4}x_0.
	\end{aligned}
	$$
	The spectral radius of $\mathbb{FV}^{-1}$ can be used to determine the basic reproduction number $\mathcal{R}_0$:
	
	$$\mathcal{R}_0=\left(\frac{k\beta_1 x_0}{\mu_3\mu_4}+\frac{\beta_2 x_0}{\mu_3}\right)\gamma,$$
	where
	\begin{equation}\label{2.1}
		\gamma=\left(\frac{\alpha\rho}{\alpha+\mu_2}e^{-m_1\tau_1}+(1-\rho)e^{-m_2\tau_2}\right).
	\end{equation}
	The parameter $\mathcal{R}_0$ can be expressed as $\mathcal{R}_0 = \mathcal{R}_{01} + \mathcal{R}_{02}$, where
	$$\mathcal{R}_{01}=\frac{k\beta_1 x_0\gamma}{\mu_3\mu_4},\qquad\mathcal{R}_{02}=\frac{\beta_2 x_0\gamma}{\mu_3}.$$
	The model has three equilibrium:
	
	(1) The non-infected equilibrium $$E_0=(x_0,0,0,0,0)=\left(\frac{\lambda}{\mu_1},0,0,0,0\right).$$
	
	(2) The infected equilibrium with an immunity-inactivated $E_1=(x_1,p_1,y_1,v_1,0)$, where
	$$
	\begin{aligned}
		&x_1=\frac{x_0}{\mathcal{R}_0},\qquad p_1=\frac{\rho\lambda e^{-m_1\tau_1}}{(\alpha+\mu_2)\mathcal{R}_0}(\mathcal{R}_0 -1),\\
		&y_1=\frac{\mu_1\mu_4}{k\beta_1+\mu_4\beta_2}(\mathcal{R}_0-1),\qquad v_1=\frac{k}{\mu_4}y_1,\qquad z_1=0.
	\end{aligned}
	$$
	
	We can obtain the CTL immune reproduction number $\mathcal{R}_1$ using a similar approach as for $\mathcal{R}_0$ \cite{WOS:000179220600004}, where
	$$\mathcal{R}_1=\frac{cy_1}{h(\mu_5+\eta y_1)}.$$
	
	(3) The infectted equilibrium with active immunity, denoted as $E_2=(x_2,p_2,y_2,v_2,z_2)$, which satisfying the following system:
	\begin{equation}\label{2.2}
		\begin{aligned}
			&\lambda-x(\beta_{1}v+\beta_{2}y)-\mu_1x=0,\\
			&\rho e^{-m_1\tau_1}x(\beta_{1}v+\beta_{2}y)-\alpha p-\mu_{2}p=0,\\
			&(1-\rho)e^{-m_2 \tau_{2}}x(\beta_{1}v+\beta_{2}y)+\alpha p-\mu_{3}y-ayz=0,\\
			&ky-\mu_{4}v=0,\\
			&\frac{c yz}{h+z}-\mu_{5}z-\eta yz=0.
		\end{aligned}
	\end{equation}
	We conclude that 
	\begin{equation}\label{2.3}
		f_1(y)\triangleq z=\frac{cy}{\mu_5+\eta y}-h,	
	\end{equation}
	\begin{equation}\label{2.4}
		f_2(z)\triangleq y=\frac{\gamma\lambda(\beta_1 k+\beta_2\mu_4)-\mu_1\mu_4(\mu_3+az)}{(\mu_3+az)(\beta_1 k+\beta_2\mu_4)}.	
	\end{equation}
	
From Eq.(\ref{2.3}), it is evident that $z$ grows monotonically with $y$. Our calculations reveal that $z = 0$ at $y = \frac{\mu_5 h}{c-\eta h}$ and $z$ reaches $-h$ when $y$ is zero. Moreover, as $y \rightarrow \infty$, $\frac{cy}{h(\mu_5+\eta y)}$ converges to $\frac{c}{h\eta}$. Hence, as $y \rightarrow \infty$, the function $f_1(y)$ approaches a horizontal asymptote at $z = \frac{c-\eta h}{\eta}$.

In Eq.(\ref{2.4}), $y$ decreases monotonically with $z$. From this relationship, it follows that $y$ equals $y_1$ at $z = 0$ and drops to zero when $z$ reaches $\frac{\mu_3(\mathcal{R}_0-1)}{a}$. When $\mathcal{R}_1 > 1$, we have $y_1>\frac{\mu_5\eta}{c-\eta h}$. Additionally, through calculations, we find that as $z \rightarrow \infty$, $f_2(z)$ approaches a horizontal asymptote at $y = -\frac{\mu_1 \mu_4}{\beta_1 k+\beta_2\mu_4}$. Therefore, the curves of two functions defined in Eqs.(\ref{2.3}) and (\ref{2.4}) intersect at exactly one point $(z_2, y_2)$ within the interval $z\in\left(0,\frac{c-\eta h}{h}\right)$ (see Figure 2).
	\begin{figure}[htbp] 
		\centering
		\includegraphics[width=11cm,height=8cm]{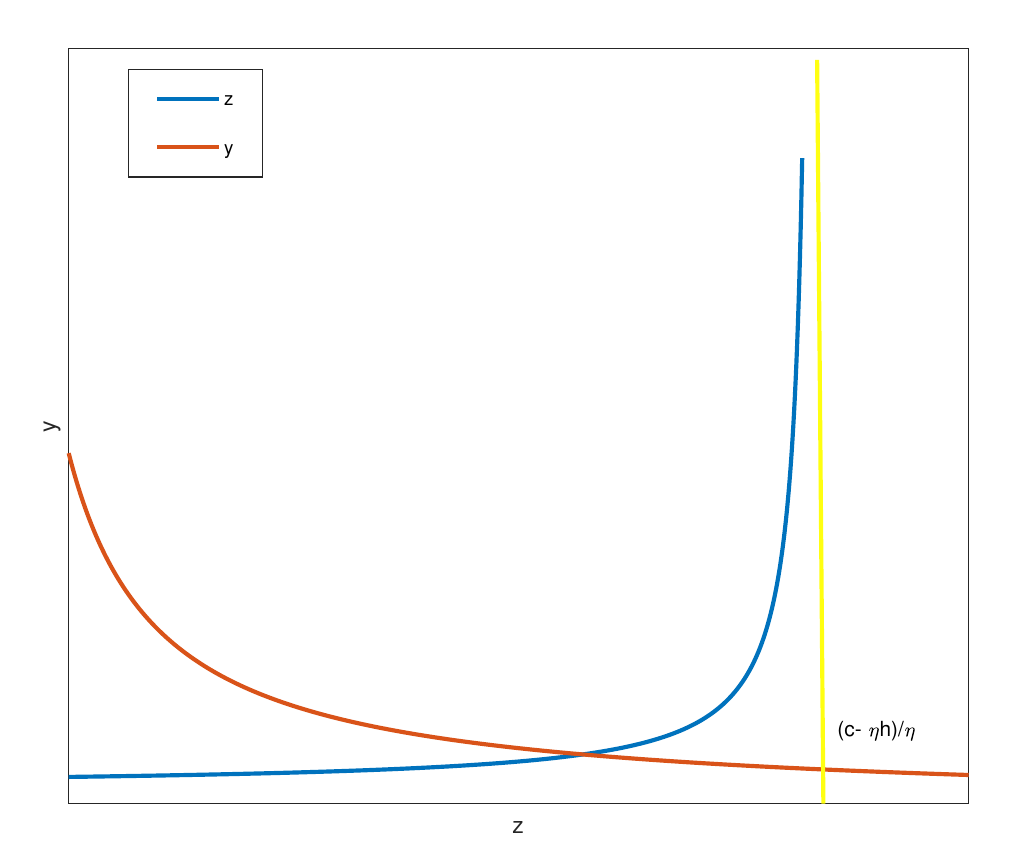}
		\caption{The curves of functions $z$ and $y$}
		\label{fig:image1}
	\end{figure}
	\\
	First, we define
	$$
	\begin{aligned}
		&b_1=ac(\beta_1 k+\beta_2 \mu_4)+\mu_3\eta(\beta_1 k+\beta_2 \mu_4)-ah\eta(\beta_1 k+\beta_2 \mu_4),\\
		&b_2=ac\mu_1\mu_4+\mu_3\mu_5(\beta_1 k+\beta_2 \mu_4)+\mu_1\mu_3\mu_4\eta-\gamma\eta\lambda(\beta_1 k+\beta_2 \mu_4)-ah\mu_5(\beta_1 k+\beta_2 \mu_4)-ah\mu_1\mu_4\eta,\\
		&b_3=-\gamma\mu_5\lambda(\beta_1 k+\beta_2 \mu_4)+\mu_1\mu_3\mu_4\mu_5-ah\mu_1\mu_4\mu_5.
	\end{aligned}
	$$
	Next, using Eqs.(\ref{2.3}) and (\ref{2.4}), we can derive $$y_2=\frac{-b_2+\sqrt{b_2^2-4b_1 b_3}}{b_1}>0.$$
	Similarly, we can obtain the expressions for $z_2$ and $v_2$ 
	$$z_2=\frac{cy}{\mu_5+\eta y}-h>0,\qquad v_2 =\frac{k}{\mu_4}y_2>0.$$ 
	Additionally, by considering the system (\ref{2.2}) evaluated at the point $E_2$, we have
	$$x_2=\frac{\lambda}{\mu_1+\beta_1 v_2+\beta_2 y_2}>0,$$
	$$p_2=\frac{\rho e^{-m_1\tau_1}}{\alpha+\mu_2}\left(\frac{\lambda \beta_1 k y_2+\lambda\beta_2\mu_4 y_2}{\mu_1\mu_4+\beta_1 k y_2+\beta_2 \mu_4 y_2}\right)>0.$$

	\section{Examination of equilibrium stability}
This part of our study initially focuses on determining the prerequisites for the local stability of three equilibria: the non-infected equilibrium $E_0$, the immune-inactivated equilibrium $E_1$, and the endemic equilibrium $E_2$ (including CTL responses). Next, we will evaluate the global stability properties of all attainable equilibria in the system (\ref{1.2}) through a rigorous analytical approach, employing Lyapunov functionals to verify asymptotic convergence across the entire domain.
	\subsection{Local asymptotic stability}
	\vskip 0.2cm
	\begin{thm}\label{thm1}
		If $\mathcal{R}_0 < 1$, the equilibrium $E_0$ is locally asymptotically stable for any positive values of $\tau_1$, $\tau_2$, and $\tau_3$. Conversely, $E_0$ is unstable when $\mathcal{R}_0 > 1$.
	\end{thm} 
	\begin{proof} We take into model (\ref{1.2}) and can get the characteristic equation
		\begin{equation}\label{3.1}(s+\mu_1)(s+\mu_5)g_0(s)=0,\end{equation}
		where
		$$
		\begin{aligned}
			g_0(s)=&(s+\alpha+\mu_2)(s+\mu_3)(s+\mu_4)-\alpha\rho e^{-(m_1+s)\tau_1}x_0(\beta_2 s+\mu_4\beta_2+k \beta_1)\\
			&-(1-\rho)e^{-(m_2+s)\tau_2}x_0[(s+\alpha+\mu_2)(\beta_2 s+\mu_4\beta_2+k\beta_1)].
		\end{aligned}
		$$
	Clearly, $s_0^*=-\mu_1$ and $s_0^{**}=-\mu_5$ are negative solutions for Eq.(\ref{3.1}), while the other roots are ascertained by the subsequent equation
		\begin{equation}\label{3.2}g_0(s)=0.\end{equation}
		We can confidently assert that Eq.(\ref{3.2}) possesses roots with solely negative real parts. Alternatively, if this is not the case, the equation must have at least one root $s_0 = \text{Re }s_0 + i\text{Im }s_0$ where $\text{Re }s_0 > 0$. Let $s_0$ be a root of the equation, and both sides of Eq.(\ref{3.2}) are divided by $(s_0+\alpha+\mu_2)(s_0+\mu_3)(s_0+\mu_4)$. It follows that
		$$
		\begin{aligned}
			1&=\left|\frac{(1-\rho)e^{-(m_2+s_0)\tau_2}x_0(\beta_2 s_0+\mu_4\beta_2+k\beta_1)}{(s_0+\mu_3)(s_0+\mu_4)}+\frac{\alpha\rho e^{-(m_1+s_0)\tau_1}x_0(\beta_2 s_0+\mu_4\beta_2+k\beta_1)}{(s_0+\mu_3)(s_0+\mu_4)(s_0+\alpha+\mu_2)}\right|\\
			&<\left|\frac{(1-\rho)e^{-(m_2+s_0)\tau_2}x_0(\beta_2 s_0+\mu_4\beta_2+k\beta_1)}{(s_0+\mu_3)(s_0+\mu_4)}+\frac{\alpha\rho e^{-(m_1+s_0)\tau_1}x_0(\beta_2 s_0+\mu_4\beta_2+k\beta_1)}{(\alpha+\mu_2)(s_0+\mu_3)(s_0+\mu_4)}\right|\\
			&<\frac{\gamma x_0[(s_0+\mu_4)\beta_2+k\beta_1)]}{(s_0+\mu_3)(s_0+\mu_4)}=\frac{\gamma x_0\beta_2}{s_0+\mu_3}+\frac{\gamma k x_0\beta_1}{(s_0+\mu_2)(s_0+\mu_4)}\\
			&<\frac{\gamma(\mu_4\beta_2+k\beta_1)x_0}{\mu_3\mu_4}=\mathcal{R}_0\\
			&<1,
		\end{aligned}
		$$
		which contradicts $\mathcal{R}_0<1$. Thus, every solution to Eq.(\ref{3.1}) exhibits negative real components, demonstrating that \( E_0 \) converges steadily in a local sense. In contrast, once the threshold \( \mathcal{R}_0 > 1 \) is surpassed, we deduce
		$$
		\begin{aligned}
			g_0(0)=&\mu_3\mu_4(\alpha+\mu_2)-(\alpha+\mu_2)x_0(1-\rho)e^{-m_2\tau_2}(\mu_4\beta_2+k\beta_1)-x_0\alpha\rho e ^{-m_1\tau_1}(\mu_4\beta_2+k\beta_1)\\
			=&\mu_3\mu_4(\alpha+\mu_2)(1-\mathcal{R}_0)<0.
		\end{aligned}
		$$
		Moreover, $g_0(s)\to+\infty$ as $s\to +\infty $. Hence, Eq.(\ref{3.2}) has a positive
		root $\omega_0\in(0,+\infty)$ such that $g_0(\omega_0) = 0$ if $\mathcal{R}_0 > 1$. That is to say, $E_0$ becomes unstable when $\mathcal{R}_0>1$.
	\end{proof}
	\begin{thm}\label{thm1}
	If $\mathcal{R}_1 < 1 < \mathcal{R}_0$, the equilibrium $E_1$ is locally asymptotically stable for any positive values of $\tau_1$, $\tau_2$, and $\tau_3$. However, $E_1$ becomes unstable when $\mathcal{R}_1 > 1$.
	\end{thm} 
	\begin{proof} We take $E_1$ into model (\ref{1.2}) and can get the characteristic equation
		\begin{equation}\label{3.3}
			g_1(s)=g_{11}(s)g_{12}(s)=0,	
		\end{equation}
		where
		$$
		\begin{aligned}
			g_{11}(s)=&s+\mu_5+\eta y_1-\frac{e^{-s\tau_3}cy_1}{h},\\
			g_{12}(s)=&(s+\beta_1 v_1+\beta_2 y_1+\mu_1)(s+\alpha+\mu_2)(s+\mu_4)(s+\mu_3)\\
			&-\alpha\rho e^{-(m_1+s)\tau_1}(s+\mu_1)[(s+\mu_4)\beta_2 x_1 +k\beta_1 x_1]\\
			&-(1-\rho)e^{-(m_2+s)\tau_2}(s+\mu_1)(s+\alpha+\mu_2)[(s+\mu_4)\beta_2 x_1+k\beta_1x_1].
		\end{aligned}
		$$
		Consider the equation
		$$g_{11}(s)=s+\mu_5+\eta y_1-\frac{e^{-s\tau_3}cy_1}{h}=0.$$
		When $\mathcal{R}_1<1$, assuming that $g_{11}(s)=0$ has a positive real root $s_1^ * = Re s_1^* + iIm s_1^*$ with $Res_1^*\ge0$, we can get $s_1^*=\frac{e^{-s\tau_3}cy_1}{h}-\mu_5-\eta y_1$, and we have $1<\left|\frac{s_1^*}{\mu_5+\eta y_1}+1\right|=\left|\frac{e^{-s\tau_3}cy_1}{h(\mu_5+\eta y_1)}\right|\le\mathcal{R}_1$, this contradicts $\mathcal{R}_1<1$. Consequently, the stability of $E_1$ is ascertained by the roots of the following equation
		\begin{equation}\label{3.4}
			g_{12}(s)=0.\end{equation}
		Meanwhile, we can assert that all roots of Eq.(\ref{3.4}) possess negative real portions. Otherwise, it possesses at least one root
		$s_1=Res_1+iIms_1$ with $Res_1\ge 0$. When $s_1$ satisfies $g_{12}(s_1)=0$, we divide both sides of Eq.(\ref{3.4}) by $(s_1+\alpha+\mu_2)(s_1+\beta_1 v_1+\beta_2 y_1+\mu_1)$, resulting in
		$$
		\begin{aligned}
			\left|(s_1+\mu_4)(s_1+\mu_3)\right|=&\left|\frac{\alpha\rho e^{-(m_1+s_1)\tau_1}(s_1+\mu_1)}{(s_1+\alpha+\mu_2)(s_1+\beta_1 v_1+\beta_2 y_1+\mu_1)}[(s_1+\mu_4)\beta_2 x_1 +k\beta_1 x_1]\right.\\
			&+\left.\frac{(1-\rho)e^{-(m_2+s_1)\tau_2}(s_1+\mu_1)}{(s_1+\beta_1 v_1+\beta_2 y_1+\mu_1)}[(s_1+\mu_4)\beta_2 x_1+k\beta_1x_1]\right|\\
			<&\gamma\left|(\beta_2 x_1 \mu_4+k\beta_1 x_1)\right|<\mu_3\mu_4.
		\end{aligned}
		$$
		Nevertheless, it is obvious that $\left|(s_1+\mu_4)(s_1+\mu_3)\right|\ge\mu_3\mu_4$, which gives rise to a contradiction. Hence, all roots of Eq.(\ref{3.3}) possess negative real portions, signifying the local asymptotic stability of $E_1$. Next, if $\mathcal{R}_1>1$, we observe that
		$$g_{11}(0)=\mu_5+\eta y_1-\frac{cy_1}{h}<0,\qquad \lim_{s\to+\infty}g_{11}(s)=+\infty.$$
		There is at least one $w_1\in(0,+\infty)$ such that $g_{11}(w_1)=0$ if $\mathcal{R}_1>1$, which implies
		$E_1$ is unstable. This concludes the verification of the argument.
	\end{proof}
	\begin{thm}\label{thm1}
	If $\mathcal{R}_1 > 1$, the equilibrium $E_2$ becomes locally asymptotically stable for any positive values of $\tau_1$ and $\tau_2$, provided $\tau_3 = 0$.
	\end{thm} 
	\begin{proof}We take $E_2$ into model (\ref{1.2}) and can get the characteristic equation
		\begin{equation}\label{3.5}
			g_2(s)=0,	
		\end{equation}
		where
		$$
	\begin{aligned}
		g_2(s)&=(s+\beta_1 v_2 +\beta_2 y_2+\mu_1)(s+\alpha+\mu_2)(s+\mu_4)(s+\mu_3+a z_2)\left(s-\frac{cy_2h}{(h+z_2)^2}+\mu_5+\eta y_2\right)\\
		&+(s+\beta_1 v_2 +\beta_2 y_2+\mu_1)(s+\alpha+\mu_2)(s+\mu_4)ay_2\left(\frac{cz_2}{h+z_2}-\eta z_2\right)\\
	\end{aligned}
	$$
		$$
		\begin{aligned}
			&-\alpha\rho e^{-(m_1+s)\tau_1}(s+\mu_1)\left(s-\frac{cy_2h}{(h+z_2)^2}+\mu_5+\eta y_2\right)[k\beta_1 x_2+(s+\mu_4)\beta_2 x_2]\\
			&-1-\rho)e^{-(m_2+s)\tau_2}(s+\mu_1)(s+\alpha+\mu_2)\left(s-\frac{cy_2h}{(h+z_2)^2}+\mu_5+\eta y_2\right)[k\beta_1 x_2+(s+\mu_4)\beta_2 x_2].
		\end{aligned}
		$$
		Based on Eq.(\ref{3.5}), all roots are confirmed to have negative real parts. Otherwise, there must be at least one root $s_2 = \text{Re}\,s_2 + i\,\text{Im}\,s_2$ with $\text{Re}\,s_2 \geq 0$. For $s_2$ such that $g_2(s_2) = 0$, we divide both sides of Eq.(\ref{3.5}) by $(s_2 + \beta_1 v_2 + \beta_2 y_2 + \mu_1)(s_2 + \alpha + \mu_2)$, yielding:
		\begin{equation}\label{3.6}
			\begin{aligned}
				&(s_2+\mu_4)(s_2+\mu_3+a z_2)\left(s_2-\frac{cy_2h}{(h+z_2)^2}+\mu_5+\eta y_2\right)+(s_2+\mu_4)ay_2\left(\frac{cz_2}{h+z_2}-\eta z_2\right)\\
				&= \frac{\alpha\rho e^{-(m_1+s_2)\tau_1}(s_2+\mu_1)}{(s_2+\beta_1 v_2 +\beta_2 y_2+\mu_1)(s_2+\alpha+\mu_2)}\left(s_2-\frac{cy_2h}{(h+z_2)^2}+\mu_5+\eta y_2\right)[k\beta_1 x_2+(s_2+\mu_4)\beta_2 x_2]\\
				&+\frac{(1-\rho)e^{-(m_2+s_2)\tau_2}(s_2+\mu_1)}{(s_2+\beta_1 v_2 +\beta_2 y_2+\mu_1)}\left(s_2-\frac{cy_2h}{(h+z_2)^2}+\mu_5+\eta y_2\right)[k\beta_1 x_2+(s_2+\mu_4)\beta_2 x_2].
		\end{aligned}\end{equation}
		Furthermore, we have 
		$$
		\begin{aligned}
			& \left|\frac{\alpha\rho e^{-(m_1+s_2)\tau_1}(s_2+\mu_1)}{(s_2+\beta_1 v_2 +\beta_2 y_2+\mu_1)(s_2+\alpha+\mu_2)}\left(s_2-\frac{cy_2h}{(h+z_2)^2}+\mu_5+\eta y_2\right)[k\beta_1 x_2+(s_2+\mu_4)\beta_2 x_2]\right.\\
			+ &\left.\frac{(1-\rho)e^{-(m_2+s)\tau_2}(s_2+\mu_1)}{(s_2+\beta_1 v_2 +\beta_2 y_2+\mu_1)}\left(s_2-\frac{cy_2h}{(h+z_2)^2}+\mu_5+\eta y_2\right)[k\beta_1 x_2+(s_2+\mu_4)\beta_2 x_2]\right|\\
			& <\gamma\left|s_2-\frac{cy_2h}{(h+z_2)^2}+\mu_5+\eta y_2\right|[k\beta_1 x_2+(s_2+\mu_4)\beta_2 x_2].
		\end{aligned}
		$$
	Simultaneously, it is derived from system (\ref{2.2}) that
		\begin{equation}\label{3.7}
			\gamma x_2(\beta_1 v_2+\beta_2 y_2)=(\mu_3+az_2)y_2,\qquad ky_2=\mu_4 v_2.
		\end{equation}
		Substituting Eq.(\ref{3.7}) into Eq.(\ref{3.6}), we have
		$$
		\begin{aligned}
			&\left|(s_2+\mu_4)(s_2+\mu_3+a z_2)\left(s_2-\frac{cy_2h}{(h+z_2)^2}+\mu_5+\eta y_2\right)+(s_2+\mu_4)ay_2\left(\frac{cz_2}{h+z_2}-\eta z_2\right)\right|\\
			&>\gamma\left|s_2-\frac{cy_2h}{(h+z_2)^2}+\mu_5+\eta y_2\right|[k\beta_1 x_2+(s_2+\mu_4)\beta_2 x_2].
		\end{aligned}
		$$
	A contradiction is obtained. Hence, under the condition $\mathcal{R}_1 > 1$, each root of Eq.(\ref{3.3}) possesses a negative real part, ensuring the local asymptotic stability of $E_2$.
	\end{proof}
	\subsection{Global asymptotic stability}
	\vskip 0.2cm
	Initially, we define a function $ N(\theta) = \theta - 1 - \ln \theta $ and use the notation $ (x, p, y, v, z) $ to represent $ (x(t), p(t), y(t), v(t), z(t)) $.
	\begin{thm}\label{thm1}
	For all positive values of $\tau_1$, $\tau_2$, and $\tau_3$, when $\mathcal{R}_0<1$, the non-infected equilibrium $E_0 = (x_0, 0, 0, 0, 0)$ of system (\ref{1.2}) attains global stability in an asymptotic sense.
	\end{thm} 
	\begin{proof} Define $M_0(x,p,y,v,z)$ as follows:
		$$
		\begin{aligned}
			M_0=&\gamma x_0 N\left(\frac{x}{x_0}\right)+\frac{\alpha}{\alpha+\mu_2}p+y+\frac{\mu_3}{k}(1-\mathcal{R}_{02})v+\frac{a h}{c}z\\
			&+\frac{\alpha\rho}{\alpha+\mu_2}e^{-m_1\tau_1}\int_{0}^{\tau_1} (\beta_1 x(t-\theta)v(t-\theta)+\beta_2 x(t-\theta)y(t-\theta))d\theta\\
			&+(1-\rho)e^{-m_2\tau_2}\int_{0}^{\tau_2} (\beta_1 x(t-\theta)v(t-\theta)+\beta_2 x(t-\theta)y(t-\theta))d\theta\\
			&+ah\int_{0}^{\tau_3} \frac{y(t-\theta)z(t-\theta)}{h+z(t-\theta)}d\theta,
		\end{aligned}
		$$
		where $\gamma$ is given by Eq.(\ref{2.1}). It is observed that the function $M_0(x, p, y, v, z)$ is positive for all $x, p, y, v, z > 0$, and $M_0(x_0, 0, 0, 0, 0) = 0$. The derivative $\frac{dM_0}{dt}$ is then calculated along the solutions of system (\ref{1.2}) as follows:
		\begin{equation}\label{3.8}
			\begin{aligned}
				\frac{dM_0}{dt}=&\gamma\left(1-\frac{x_0}{x}\right)(\lambda-\mu_1 x-\beta_1 xv-\beta_2 xy)\\
				&+\frac{\alpha}{\alpha+\mu_2}\left\{\rho e^{-m_1\tau_1}x(t-\tau_1)[\beta_1 v(t-\tau_1)+\beta_2 y(t-\tau_1)]-(\alpha+\mu_2)p\right\}\\
				&+\left\{(1-\rho)e^{-m_2\tau_2} x(t-\tau_2)[\beta_1 v(t-\tau_2)+\beta_2 y(t-\tau_2)]+\alpha p-\mu_3 y-a yz\right\}\\
				&+\frac{\alpha\rho}{\alpha+\mu_2}e^{-m_1\tau_1}\left\{\beta_1 xv+\beta_2 xy-\beta_1 x(t-\tau_1)v(t-\tau_1)-\beta_2 x(t-\tau_1)y(t-\tau_1)\right\}\\
				&+(1-\rho)e^{-m_2\tau_2}\left\{\beta_1 xv+\beta_2 xy-\beta_1 x(t-\tau_2)v(t-\tau_2)-\beta_2 x(t-\tau_2)y(t-\tau_2)\right\}\\
				&+\frac{a h}{c}\left(\frac{cyz}{h+z}-\mu_5 z-\eta yz\right)+\mu_3(1-\mathcal{R}_{02})y-\frac{\mu_3\mu_4}{k}(1-\mathcal{R}_{02})v.
		\end{aligned}\end{equation}
		Eq.(\ref{3.8}) can be simplified as follows:
		$$
		\begin{aligned}
			\frac{dM_0}{dt}=&-\gamma\frac{\mu_1(x-x_0)^2}{x}+\gamma(\beta_1 x_0 v+\beta_2 x_0 y)-\mu_3 y-a yz\\
			&+\frac{a h}{c}\left(\frac{cyz}{h+z}-\mu_5 z-\eta yz\right)+\mu_3(1-\mathcal{R}_{02})y-\frac{\mu_3\mu_4}{k}(1-\mathcal{R}_{02})v
		\end{aligned}
		$$
		$$
		\begin{aligned}
			=&-\gamma\frac{\mu_1(x-x_0)^2}{x}+\left(\gamma\beta_1 x_0-\frac{\mu_3\mu_4}{k}(1-\mathcal{R}_{02})\right)v\\
			&+(\gamma\beta_2 x_0 -\mu_3 \mathcal{R}_{02})y+a yz\left(\frac{h}{h+z}-1-\frac{h\eta}{c}\right)-\frac{ah}{c}\mu_5 z.
		\end{aligned}
		$$
		We have
		$$
		\begin{aligned}
			&\gamma\beta_2 x_0-\mu_3\mathcal{R}_{02}=0,\\
			&\gamma\beta_1 x_0-\frac{\mu_3\mu_4}{k}(1-\mathcal{R}_{02})=\frac{\mu_3\mu_4}{k}(\mathcal{R}_0 -1).
		\end{aligned}
		$$
		Therefore, we obtain
		$$\frac{dM_0}{dt}=-\gamma\frac{\mu_1(x-x_0)^2}{x}+\frac{\mu_3\mu_4}{k}(\mathcal{R}_0 -1)v+ayz\left(\frac{h}{h+z}-1\right)-\frac{ah}{c}(\eta yz+\mu_5 z).$$
		Note that
		$$\frac{h}{h+z(t)}-1\le0,\qquad z(t)\ge 0.$$
		Thus, $\frac{dM_0}{dt}\le 0$ when $\mathcal{R}_0\le 1$ for all $x,p,y,v,z>0.$ Moreover, $\frac{dM_0}{dt}=0$ if and only if $v(t) = 0$, $z(t) = 0$, and $x(t) =x_0$. Let $V_0=\left\{(x,p,y,v,z):\frac{dM_0}{dt}=0\right\}$ and $V'_
		0$ be the largest invariant subset
		of $V_0$. System (\ref{1.2}) solutions trend to $V'_0$.  $v(t) = 0$ is the result for every element in $V'_0$.
		Thus Eq.(\ref{1.2}) yields
		$$v'(t)=0=ky(t).$$
		Given that $y(t) = 0$, from Eq.(\ref{1.2}) it follows that
		$$p'(t)=0=-(\alpha+\mu_2)p(t).$$
		Consequently, $p(t) = 0$, implying that $V'_0$ reduces to the singleton set $\left\{E_0\right\}$. Applying LaSalle’s invariance principle, we determine that $E_0$ is globally asymptotically stable when $\mathcal{R}_0 \leq 1$.
	\end{proof}
	\begin{thm}\label{thm1}
		For the system (\ref{1.2}), given that $\mathcal{R}_1 \leq 1 < \mathcal{R}_0$, it follows that $E_1$ is globally asymptotically stable for all $\tau_1>0$, $\tau_2>0$, $\tau_3>0$.
	\end{thm} 
	\begin{proof} Let $M_1(x,p,y,v,z)$ be given as follows:
		$$M_1=\gamma x_1N\left(\frac{x}{x_1}\right)+\frac{\alpha}{\alpha+\mu_2}p_1 N\left(\frac{p}{p_1}\right)+y_1N\left(\frac{y}{y_1}\right)+\frac{\gamma\beta_1 x_1 v_1}{ky_1} v_1N\left(\frac{v}{v_1}\right)+\frac{a}{\mu_5}y_1 z $$
		$$
		\begin{aligned}
			&+\frac{\alpha\rho}{\alpha+\mu_2}e^{-m_1\tau_1}\beta_1 x_1 v_1\int_{0}^{\tau_1} N\left(\frac{x(t-\theta)v(t-\theta)}{x_1 v_1}\right)d\theta\\
			&+(1-\rho)e^{-m_2\tau_2}\beta_1 x_1 v_1\int_{0}^{\tau_2}N\left(\frac{x(t-\theta)v(t-\theta)}{x_1 v_1}\right)d\theta\\
			&+\frac{\alpha\rho}{\alpha+\mu_2}e^{-m_1\tau_1}\beta_2 x_1 y_1 \int_{0}^{\tau_1} N\left(\frac{x(t-\theta)y(t-\theta)}{x_1 y_1}\right)d\theta\\
			&+(1-\rho)e^{-m_2\tau_2}\beta_2 x_1 y_1\int_{0}^{\tau_2} N\left(\frac{x(t-\theta)y(t-\theta)}{x_1 y_1}\right)d\theta\\
			&+\frac{ac}{\mu_5}y_1\int_{0}^{\tau_3}\frac{y(t-\theta)z(t-\theta)}{h+z(t-\theta)}d\theta.
		\end{aligned}
		$$
		Given that $M_1(x,p,y,v,z) > 0$ for all positive values of $x, p, y, v, z$, and $M_1(x_1,p_1,y_1,v_1,0) = 0$, we proceed to compute the derivative $\frac{dM_1}{dt}$.
		\begin{equation}\label{3.9}
		\begin{aligned}
			\frac{dM_1}{dt}=&\gamma\left(1-\frac{x_1}{x}\right)(\lambda-\beta_1 xv-\beta_2 xy-\mu_1 x)\\
			&+\frac{\alpha}{\alpha+\mu_2}\left(1-\frac{p_1}{p}\right)\left\{\rho e^{-m_1\tau_1}x(t-\tau_1)[\beta_1 v(t-\tau_1)+\beta_2 y(t-\tau_1)]-\alpha p-\mu_2 p\right\}\\
			&+\left(1-\frac{y_1}{y}\right)\left\{(1-\rho)e^{-m_2\tau_2}x(t-\tau_2)[\beta_1 v(t-\tau_2)+\beta_2 y(t-\tau_2)]+\alpha p-\mu_3 y-ayz\right\}\\
			&+\frac{\gamma\beta_1 x_1 v_1}{ky_1}\left(1-\frac{v_1}{v}\right)(ky-\mu_4v)+\frac{ay_1}{\mu_5}\left(\frac{cyz}{h+z}-\mu_5 z-\eta yz\right)\\
			&+\frac{\alpha\rho}{\alpha+\mu_2}e^{-m_1\tau_1}\beta_1 x_1 v_1\left[\frac{xv}{x_1 v_1}-\frac{x(t-\tau_1)v(t-\tau_1)}{x_1 v_1}-\ln\left(\frac{x(t-\tau_1)v(t-\tau_1)}{x_1 v_1}\right)\right]\\
			&+(1-\rho)e^{-m_2\tau_2}\beta_1 x_1 v_1\left[\frac{xv}{x_1 v_1}-\frac{x(t-\tau_2)v(t-\tau_2)}{x_1 v_1}-\ln\left(\frac{x(t-\tau_2)v(t-\tau_2)}{x_1 v_1}\right)\right]\\
				\end{aligned} 
		\end{equation}
	$$
			\begin{aligned}
				&+\frac{\alpha\rho}{\alpha+\mu_2}e^{-m_1\tau_1}\beta_2 x_1 y_1\left[\frac{xy}{x_1 y_1}-\frac{x(t-\tau_1)y(t-\tau_1)}{x_1 y_1}-\ln\left(\frac{x(t-\tau_1)y(t-\tau_1)}{x_1 y_1}\right)\right]\\
				&+(1-\rho)e^{-m_2\tau_2}\beta_2 x_1 y_1\left[\frac{xy}{x_1 y_1}-\frac{x(t-\tau_2)y(t-\tau_2)}{x_1 y_1}-\ln\left(\frac{x(t-\tau_2)y(t-\tau_2)}{x_1 y_1}\right)\right].
		\end{aligned}$$
		Since $E_1$ is the equilibrium, we have 
		$$\begin{aligned}
			&\lambda=\beta_1 x_1 v_1 +\beta_2 x_1 y_1+\mu_1 x_1,\\
			&\frac{\alpha\rho}{\alpha+\mu_2}e^{-m_1\tau_1}(\beta_1 x_1 v_1 +\beta_2 x_1 y_1 )=\alpha p_1,\\
			&\frac{\alpha\rho}{\alpha+\mu_2}e^{-m_1\tau_1}(\beta_1 x_1 v_1 +\beta_2 x_1 y_1 )+(1-\rho)e^{-m_2\tau_2}(\beta_1 x_1 v_1 +\beta_2 x_1 y_1 )=\mu_3 y_1,\\
			&ky_1=\mu_4 v_1.
		\end{aligned}
		$$
		Then we simplify Eq.(\ref{3.9}) as follows
		$$
		\begin{aligned}
			\frac{dM_1}{dt}=&\gamma\left(1-\frac{x_1}{x}\right)(\mu_1 x_1-\mu_1 x)+\gamma\left(1-\frac{x_1}{x}\right)(\beta_1 x_1 v_1 +\beta_2 x_1 y_1 )\\
			&-\frac{\alpha\rho}{\alpha+\mu_2}e^{-m_1\tau_1}\left(\beta_1 x_1 v_1\frac{x(t-\tau_1)v(t-\tau_1)p_1}{x_1 v_1 p}+\beta_2 x_1 y_1\frac{x(t-\tau_1)y(t-\tau_1)p_1}{x_1 y_1 p}\right)\\
			&+\frac{\alpha\rho}{\alpha+\mu_2}e^{-m_1\tau_1}(\beta_1 x_1 v_1 +\beta_2 x_1 y_1 )-\alpha p\frac{y_1}{y}\\
			&-(1-\rho)e^{-m_2\tau_2}\left(\beta_1 x_1 v_1\frac{x(t-\tau_2)v(t-\tau_2)y_1}{x_1 v_1 y}+\beta_2 x_1 y_1\frac{x(t-\tau_2)y(t-\tau_2)}{x_1 y}\right)\\
			&+\left(\frac{\alpha\rho}{\alpha+\mu_2}e^{-m_1\tau_1}+(1-\rho)e^{-m_2\tau_2}\right)(\beta_1 x_1 v_1 +\beta_2 x_1 y_1 )	\\
			&+\gamma\beta_1 x_1 v_1\left(1-\frac{y v_1}{y_1 v}\right)+\gamma\beta_1 x_1 v_1+\frac{\gamma\beta_1 x_1 v_1 y}{y_1}-\mu_3 y\\
			&-\frac{\alpha\rho}{\alpha+\mu_2}e^{-m_1\tau_1}(\beta_1 x_1 v_1 +\beta_2 x_1 y_1 )\frac{p y_1}{p_1 y}\\
			&+\frac{ayz}{h+z}\left(\frac{cy_1-\eta y_1 h-\mu_5 h}{\mu_5}\right)-\frac{a yz}{(h+z)\mu_5}(\eta y_1 z+\mu_5 z)\\
			&+\frac{\alpha\rho}{\alpha+\mu_2}e^{-m_1\tau_1}\beta_1 x_1 v_1\ln\left(\frac{x(t-\tau_1)v(t-\tau_1)}{xv}\right)\\
			&+(1-\rho)e^{-m_2\tau_2}\beta_1 x_1 v_1\ln\left(\frac{x(t-\tau_2)v(t-\tau_2)}{xv}\right)\\
			&+\frac{\alpha\rho}{\alpha+\mu_2}e^{-m_1\tau_1}\beta_2 x_1 y_1\ln\left(\frac{x(t-\tau_1)y(t-\tau_1)}{xy}\right)\\
			&+(1-\rho)e^{-m_2\tau_2}\beta_2 x_1 y_1\ln\left(\frac{x(t-\tau_2)y(t-\tau_2)}{xy}\right).
		\end{aligned}
		$$
	
	Examine the following equalities using $(i = 1)$:
	$$\ln\left(\frac{x(t-\tau_1)v(t-\tau_1)}{xv}\right)=\ln\left(\frac{x(t-\tau_1)v(t-\tau_1)p_1}{x_i v_i p}\right)+\ln\left(\frac{py_i}{p_i y}\right)+\ln\left(\frac{yv_i}{y_i v}\right)+\ln\left(\frac{x_i}{x}\right),$$
		\begin{equation}\label{3.10}
			\begin{aligned}
				&\ln\left(\frac{x(t-\tau_2)v(t-\tau_2)}{xv}\right)=\ln\left(\frac{x(t-\tau_2)v(t-\tau_2)y_i}{x_i v_i y}\right)+\ln\left(\frac{yv_i}{y_i v}\right)+\ln\left(\frac{x_i}{x}\right),\\
				&\ln\left(\frac{x(t-\tau_1)y(t-\tau_1)}{xy}\right)=\ln\left(\frac{x(t-\tau_1)y(t-\tau_1)p_i}{x_i y_i p}\right)+\ln\left(\frac{py_i}{p_i y}\right)+\ln\left(\frac{x_i}{x}\right),\\
				&\ln\left(\frac{x(t-\tau_2)y(t-\tau_2)}{xy}\right)=\ln\left(\frac{x(t-\tau_2)y(t-\tau_2)}{x_i y}\right)+\ln\left(\frac{x_i}{x}\right),
		\end{aligned}\end{equation}
		we obtain
		$$
		\begin{aligned}
			\frac{dM_1}{dt}=&-\gamma\frac{\mu_1(x-x_1)^2}{x}-\gamma(\beta_1 x_1 v_1+\beta_2 x_1 y_1)N\left(\frac{x_1}{x}\right)\\
			&-\frac{\alpha\rho}{\alpha+\mu_2}e^{-m_1\tau_1}\beta_1 x_1 v_1N\left(\frac{x(t-\tau_1)v(t-\tau_1)p_1}{x_1 v_1 p}\right)\\
			&-(1-\rho)e^{-m_2\tau_2}\beta_1 x_1 v_1 N\left(\frac{x(t-\tau_2)v(t-\tau_2)y_1}{x_1 v_1 y}\right)\\
			&-\frac{\alpha\rho}{\alpha+\mu_2}e^{-m_1\tau_1}\beta_2 x_1 y_1 N\left(\frac{x(t-\tau_1)y(t-\tau_1)p_1}{x_1 y_1 p}\right)\\
			&-(1-\rho)e^{-m_2\tau_2}\beta_2 x_1 y_1 N\left(\frac{x(t-\tau_2)y(t-\tau_2)}{x_1y}\right)\\
			&-\frac{\alpha\rho}{\alpha+\mu_2}e^{-m_1\tau_1}(\beta_1 x_1 v_1 +\beta_2 x_1 y_1 )N\left(\frac{p y_1}{p_1 y}\right)\\
			&-\gamma\beta x_1 v_1 N\left(\frac{y v_1}{y_1 v}\right)-\frac{a yz}{(h+z)\mu_5}(\eta y_1 z+\mu_5 z)\\
			&+\frac{ayz}{h+z}\left(\frac{cy_1-\eta y_1 h-\mu_5 h}{\mu_5}\right).
		\end{aligned}
		$$
		Due to $\mathcal{R}_1=\frac{cy_1}{h(\mu_5+\eta y_1)}\le 1$, $cy_1-\eta y_1 h-\mu_5 h\le0$. Thus, $\frac{dM_1}{dt}\le 0$ and $\frac{dM_1}{dt}=0$ occur at $E_1$. Let $V'_1$ be the largest invariant subset of the set $V_1 =\left\{(x,p,y,v,z):\frac{dM_1}{dt}=0\right\}$. Therefore, the trajectories of system (\ref{1.2}) converge to $V'_1$. Clearly, $V_1$ consists solely of $\left\{E_1\right\}$. By applying LaSalle’s invariance principle, we establish that $E_1$ is globally asymptotically stable under the conditions $\mathcal{R}_1 \le 1$ and $\mathcal{R}_0 > 1$.
	\end{proof}
	\begin{thm}\label{thm1}
Given system (\ref{1.2}), when $\mathcal{R}_1>1$, $E_2$ becomes globally asymptotically stable for any positive $\tau_1$ and $\tau_2$, with $\tau_3=0.$
	\end{thm} 
	\begin{proof} Consider $M_2(x,p,y,v,z)$:
		$$
		\begin{aligned}
			M_2=&\gamma x_2N\left(\frac{x}{x_2}\right)+\frac{\alpha}{\alpha+\mu_2}p_2 N\left(\frac{p}{p_2}\right)+y_2 N\left(\frac{y}{y_2}\right)+\frac{\gamma\beta_1 x_2 v_2}{ky_2} v_2N\left(\frac{v}{v_2}\right)+\frac{az_2}{\mu_5}y_2 N\left(\frac{z}{z_2}\right)\\
			&+\frac{\alpha\rho}{\alpha+\mu_2}e^{-m_1\tau_1}\beta_1 x_2 v_2\int_{0}^{\tau_1} N\left(\frac{x(t-\theta)v(t-\theta)}{x_2 v_2}\right)d\theta\\
			&+(1-\rho)e^{-m_2\tau_2}\beta_1 x_2 v_2\int_{0}^{\tau_2}N\left(\frac{x(t-\theta)v(t-\theta)}{x_2 v_2}\right)d\theta\\
			&+\frac{\alpha\rho}{\alpha+\mu_2}e^{-m_1\tau_1}\beta_2 x_2 y_2 \int_{0}^{\tau_1} N\left(\frac{x(t-\theta)y(t-\theta)}{x_2 y_2}\right)d\theta\\
		\end{aligned}$$
		$$
		\begin{aligned}
			&+(1-\rho)e^{-m_2\tau_2}\beta_2 x_2 y_2\int_{0}^{\tau_2} N\left(\frac{x(t-\theta)y(t-\theta)}{x_2 y_2}\right)d\theta.
		\end{aligned}$$
		
		We observe that $M_2(x, p, y, v, z) > 0$ for all $x, p, y, v, z > 0$, and $M_2(x, p, y, v, z)$ attains its global minimum at $E_2$. To proceed, we derive $\frac{dM_2}{dt}$ as follows:
		\begin{equation}\label{3.11}
			\begin{aligned}
				\frac{dM_2}{dt}=&\gamma\left(1-\frac{x_2}{x}\right)(\lambda-\beta_1 xv-\beta_2 xy-\mu_1 x)\\
				&+\frac{\alpha}{\alpha+\mu_2}\left(1-\frac{p_2}{p}\right)\left\{\rho e^{-m_1\tau_1}x(t-\tau_1)[\beta_1 v(t-\tau_1)+\beta_2 y(t-\tau_1)]-\alpha p-\mu_2 p\right\}\\
				&+\left(1-\frac{y_2}{y}\right)\left\{(1-\rho)e^{-m_2\tau_2}x(t-\tau_2)[\beta_1 v(t-\tau_2)+\beta_2 y(t-\tau_2)]+\alpha p-\mu_3 y-ayz\right\}\\
				&+\frac{\gamma\beta_1 x_2 v_2}{ky_1}\left(1-\frac{v_2}{v}\right)(ky-\mu_4v)+\frac{ay_2}{\mu_5}\left(1-\frac{z_2}{z}\right)\left(\frac{cyz}{h+z}-\mu_5 z-\eta yz\right)\\
				&+\frac{\alpha\rho}{\alpha+\mu_2}e^{-m_1\tau_1}\beta_1 x_2 v_2\left[\frac{xv}{x_2 v_2}-\frac{x(t-\tau_1)v(t-\tau_1)}{x_2 v_2}-\ln\left(\frac{x(t-\tau_1)v(t-\tau_1)}{x_2 v_2}\right)\right]\\
				&+(1-\rho)e^{-m_2\tau_2}\beta_1 x_2 v_2\left[\frac{xv}{x_2 v_2}-\frac{x(t-\tau_2)v(t-\tau_2)}{x_2 v_2}-\ln\left(\frac{x(t-\tau_2)v(t-\tau_2)}{x_2 v_2}\right)\right]\\
				&+\frac{\alpha\rho}{\alpha+\mu_2}e^{-m_1\tau_1}\beta_2 x_2 y_2\left[\frac{xy}{x_2 y_2}-\frac{x(t-\tau_1)y(t-\tau_1)}{x_2 y_2}-\ln\left(\frac{x(t-\tau_1)y(t-\tau_1)}{x_2 y_2}\right)\right]\\
				&+(1-\rho)e^{-m_2\tau_2}\beta_2 x_2 y_2\left[\frac{xy}{x_2 y_2}-\frac{x(t-\tau_2)y(t-\tau_2)}{x_2 y_2}-\ln\left(\frac{x(t-\tau_2)y(t-\tau_2)}{x_2 y_2}\right)\right].\\
		\end{aligned}\end{equation}
		Note that 
		$$\begin{aligned}
			&\lambda=\beta_1 x_2 v_2 +\beta_2 x_2 y_2+\mu_1 x_2,\\
			&\frac{\alpha\rho}{\alpha+\mu_2}e^{-m_1\tau_1}(\beta_1 x_2 v_2 +\beta_2 x_2 y_2 )=\alpha p_2,\\
			&\frac{\alpha\rho}{\alpha+\mu_2}e^{-m_1\tau_1}(\beta_1 x_2 v_2 +\beta_2 x_2 y_2 )+(1-\rho)e^{-m_2\tau_2}(\beta_1 x_2 v_1 +\beta_2 x_2 y_2 )=\mu_3 y_2+ay_2 z_2,\\
			&ky_2=\mu_4 v_2,\\
			&\mu_5=\frac{cy_2}{h+z_2}-\eta y_2.
		\end{aligned}
		$$
		Simplifying Eq.(\ref{3.11}), we get
		$$
		\begin{aligned}
			\frac{dM_2}{dt}=&\gamma\left(1-\frac{x_2}{x}\right)(\mu_1 x_2-\mu_1 x)+\gamma\left(1-\frac{x_2}{x}\right)(\beta_1 x_2 v_2 +\beta_2 x_2 y_2 )\\
			&-\frac{\alpha\rho}{\alpha+\mu_2}e^{-m_1\tau_1}\left(\beta_1 x_2 v_2\frac{x(t-\tau_1)v(t-\tau_1)p_2}{x_2 v_2 p}+\beta_2 x_2 y_2\frac{x(t-\tau_1)y(t-\tau_1)p_2}{x_2 y_2 p}\right)\\
			&+\frac{\alpha\rho}{\alpha+\mu_2}e^{-m_1\tau_1}(\beta_1 x_2 v_2 +\beta_2 x_2 y_2 )\\
			&-(1-\rho)e^{-m_2\tau_2}\left(\beta_1 x_2 v_2\frac{x(t-\tau_2)v(t-\tau_2)y_2}{x_2 v_2 y}+\beta_2 x_2 y_2\frac{x(t-\tau_2)y(t-\tau_2)}{x_2 y}\right)\\
			&+\left(\frac{\alpha\rho}{\alpha+\mu_2}e^{-m_1\tau_1}+(1-\rho)e^{-m_2\tau_2}\right)(\beta_1 x_2 v_2 +\beta_2 x_2 y_2 )\\
		\end{aligned}$$
		$$
		\begin{aligned}
			&+\gamma\beta_1 x_2 v_2\left(1-\frac{y v_2}{y_2 v}\right)-\frac{\alpha\rho}{\alpha+\mu_2}e^{-m_1\tau_1}(\beta_1 x_2 v_2 +\beta_2 x_2 y_2 )\frac{p y_2}{p_2 y}\\
			&+\frac{ay_2 cy}{\mu_5(h+z)}(z-z_2)+\frac{ay_2 \eta y}{\mu_5}(z_2-z)+ay(z_2-z)\\
			&+\frac{\alpha\rho}{\alpha+\mu_2}e^{-m_1\tau_1}\beta_1 x_2 v_2\ln\left(\frac{x(t-\tau_1)v(t-\tau_1)}{xv}\right)\\
			&+(1-\rho)e^{-m_2\tau_2}\beta_1 x_2 v_2\ln\left(\frac{x(t-\tau_2)v(t-\tau_2)}{xv}\right)\\
			&+\frac{\alpha\rho}{\alpha+\mu_2}e^{-m_1\tau_1}\beta_2 x_2 y_2\ln\left(\frac{x(t-\tau_1)y(t-\tau_1)}{xy}\right)\\
			&+(1-\rho)e^{-m_2\tau_2}\beta_2 x_2 y_2\ln\left(\frac{x(t-\tau_2)y(t-\tau_2)}{xy}\right).
		\end{aligned}$$
		\\
		Applying Eq.(\ref{3.10}) when $(i = 2)$, we obtain
		$$
		\begin{aligned}
			\frac{dM_2}{dt}=&-\gamma\frac{\mu_1(x-x_2)^2}{x}-\gamma(\beta_1 x_2 v_2+\beta_2 x_2 y_2)N\left(\frac{x_2}{x}\right)\\
			&-\frac{\alpha\rho}{\alpha+\mu_2}e^{-m_1\tau_1}(\beta_1 x_2 v_2 +\beta_2 x_2 y_2 )N\left(\frac{p y_2}{p_2 y}\right)\\
			&-\gamma\beta x_2 v_2 N\left(\frac{y v_2}{y_2 v}\right)-\frac{ay}{\mu_5(h+z)}(\mu_5+\eta y_2)(z-z_2)^2\\
			&-\frac{\alpha\rho}{\alpha+\mu_2}e^{-m_1\tau_1}\beta_1 x_2 v_2N\left(\frac{x(t-\tau_1)v(t-\tau_1)p_2}{x_1 v_1 p}\right)\\
			&-(1-\rho)e^{-m_2\tau_2}\beta_1 x_2 v_2 N\left(\frac{x(t-\tau_2)v(t-\tau_2)y_2}{x_2 v_2 y}\right)\\
			&-\frac{\alpha\rho}{\alpha+\mu_2}e^{-m_1\tau_1}\beta_2 x_2 y_2 N\left(\frac{x(t-\tau_1)y(t-\tau_1)p_2}{x_2 y_2 p}\right)\\
			&-(1-\rho)e^{-m_2\tau_2}\beta_2 x_2 y_2 N\left(\frac{x(t-\tau_2)y(t-\tau_2)}{x_2y}\right).
		\end{aligned}
		$$
		
		Since $\mathcal{R}_1 > 1$, then $x_2,p_2,y_2,v_2,z_2 > 0$. We obtain $\frac{dM_2}{dt}\le0$, and then the solutions of
		system (\ref{1.2}) tend to $V'_2$, the largest invariant subset of $V_2 = \left\{(x,p,y,v,z): \frac{dM_2}{dt}=0\right\}$. Clearly, $\frac{dM_2}{dt} = 0$ when $x(t)=x_2,p(t)=p_2,y(t)=y_2,v(t)=v_2,z(t)=z_2$. The global asymptotic stability of $E_2$ is conducted from LaSalle’s invariance principle.
	\end{proof}
	\section{ Hopf bifurcation analysis of $E_2$}
    \subsection{Conditions for stability and Hopf bifurcation}
	Theorem 3.6 establishes that the equilibrium $E_2$ is globally asymptotically stable under the conditions of positive $\tau_1$ and $\tau_2$, with $\tau_3=0$. In contrast, when $\tau_3 > 0$, a Hopf bifurcation emerges at $E_2$. Given the complexity of calculations when $\tau_1 > 0$ and $\tau_2 > 0$, we simplify our analysis by setting $\tau_1 = 0$ and $\tau_2 = 0$ while keeping $\tau_3 > 0$ in the subsequent discussions.
	
	The characteristic equation of the system (\ref{1.2}) at $E_2$
	is
	\begin{equation}\label{4.1}\xi^5+I_1\xi^4+I_2\xi^3+I_3\xi^2+I_4\xi+I_5+(T_1\xi^4+T_2\xi^3+T_3\xi^2+T_4\xi+T_5)e^{-\xi\tau_3}=0,\end{equation}
	where
	
	\begin{align*}
		I_1=&\beta_1 v_2+\beta_2y_2+\mu_1+\alpha+\mu_2+\mu_4+\mu_3+az_2+\mu_5+\eta y_2+\rho\beta_2 x_2-\beta_2 x_2,\\
		I_2=&(\beta_1 v_2+\beta_2 y_2+\mu_1)(\alpha+\mu_2+\mu_4+\mu_3+az_2+\mu_5+\eta y_2)+(\alpha+\mu_2)(\mu_4+\mu_3+a z_2+\mu_5+\eta y_2)\\
		&+\mu_4(\mu_3+a z_2+\mu_5+\eta y_2)+(\mu_3+az_2)(\mu_5+\eta y_2)-ay_2\eta z_2-\alpha\rho\beta_2 x_2-k(1-\rho)\beta_1 x_2\\
		&-(1-\rho)\beta_2 x_2 (\mu_1+\alpha+\mu_2+\mu_5+\eta y_2+ \mu_4),\\
		I_3=&(\beta_1 v_2+\beta_2 y_2+\mu_1)(\alpha+\mu_2)(\mu_4+\mu_3+az_2+\mu_5+\eta y_2)+\mu_4(\alpha+\mu_2)(\mu_3+az_2+\mu_5+\eta y_2)\\
		&+(\beta_1 v_2+\beta_2 y_2+\mu_1)(\mu_3+az_2)(\mu_5+\eta y_2)+\mu_4(\beta_1 v_2 +\beta_2 y_2+\mu_1)(\mu_3+az_2+\mu_5+\eta y_2)\\
		&+(\alpha+\mu_2)(\mu_3+az_2)(\mu_5+\eta y_2)-ay_2\eta z_2(\beta_1 v_2+\beta_2 y_2+\mu_1+\alpha+\mu_2+\mu_4)-k\alpha\rho\beta_1 x_2\\
		&+\mu_4(\mu_3+az_2)(\mu_5+\eta y_2)-\alpha\rho\beta_2 x_2(\mu_1+\mu_5+\eta y_2+\mu_4)-k(1-\rho)\beta_1 x_2(\mu_1+\alpha+\mu_2+\mu_5+\eta y_2)\\
		&-(1-\rho)\beta_2 x_2[\mu_1(\alpha+\mu_2+\mu_5+\eta y_2 +\mu_4)+(\alpha+\mu_2)(\mu_5+\eta y_2+\mu_4)+(\mu_5+\eta y_2)\mu_4],\\
		I_4=&\mu_4(\beta_1 v_2+\beta_2 y_2+\mu_1)(\alpha+\mu_2)(\mu_3+az_2+\mu_5+\eta y_2)+\mu_4(\alpha+\mu_2)(\mu_3+az_2)(\mu_5+\eta y_2)\\
		&+(\beta_1 v_2+\beta_2 y_2+\mu_1)(\alpha+\mu_2)(\mu_3+az_2)(\mu_5+\eta y_2)+\mu_4(\beta_1 v_2+\beta_2 y_2+\mu_1)(\mu_3+az_2)(\mu_5+\eta y_2)\\
		&-ay_2\eta z_2[(\beta_1 v_2 +\beta_2 y_2+\mu_1)(\alpha+\mu_2+\mu_4)+\mu_4(\alpha+\mu_2)]-k\alpha\rho\beta_1 x_2(\mu_1+\mu_5+\eta y_2)\\
		&-(1-\rho)\beta_2 x_2[\mu_1(\alpha+\mu_2)(\mu_5+\eta y_2+\mu_4)+\mu_1\mu_4(\mu_5+\eta y_2)+\mu_4(\alpha+\mu_2)(\mu_5+\eta y_2)]\\
		&-k(1-\rho)\beta_1 x_2[\mu_1(\alpha+\mu_2+\mu_5+\eta y_2)+(\alpha+\mu_2)(\mu_5+\eta y_2)]\\
		&-\alpha\rho\beta_2 x_2[\mu_1(\mu_5+\eta y_2+\mu_4)+\mu_4(\mu_5+\eta y_2)],\\
		I_5=&\mu_4(\beta_1 v_2+\beta_2 y_2+\mu_1)(\alpha+\mu_2)(\mu_3+az_2)(\mu_5+\eta y_2)-ay_2\eta z_2\mu_4(\beta_1 v_2+\beta_2 y_2+\mu_1)(\alpha+\mu_2)\\
		&-\alpha\rho\beta_2 x_2\mu_1\mu_4(\mu_5+\eta y_2)-k\alpha\rho\beta_1 x_2 \mu_1(\mu_5+\eta y_2)-(1-\rho)\beta_2 x_2 \mu_1\mu_4(\alpha+\mu_2)(\mu_5+\eta y_2)\\
		&-k(1-\rho)\beta_1 x_2 \mu_1 (\alpha+\mu_2)(\mu_5+\eta y_2),\\
		T_1=&-\frac{chy_2}{(h+z_2)^2},\\
		T_2=&\frac{acy_2z_2}{h+z_2}+\frac{chy_2}{(h+z_2)^2}[(1-\rho)\beta_2 x_2-(\beta_1 v_2+\beta_2 y_2+\mu_1+\alpha+\mu_2+\mu_4+\mu_3+az_2)],\\
		T_3=&\frac{acy_2z_2}{h+z_2}(\beta_1 v_2+\beta_2 y_2+\mu_1+\alpha+\mu_2+\mu_4)+\frac{chy_2}{(h+z_2)^2}[\alpha\rho\beta_2 x_2+(1-\rho)k\beta_1 x_2]\\
		&-\frac{chy_2}{(h+z_2)^2}[(\beta_1 v_2+\beta_2 y_2+\mu_1)(\alpha+\mu_2+\mu_4+\mu_3+az_2)+(\alpha+\mu_2)(\mu_4+\mu_3+az_2)]\\
		&+\frac{chy_2}{(h+z_2)^2}[(1-\rho)\beta_2 x_2(\mu_1+\alpha+\mu_2+\mu_4)-\mu_4(\mu_3+az_2)],\\
		T_4=&\frac{acy_2z_2}{h+z_2}[(\beta_1 v_2+\beta_2 y_2+\mu_1)(\alpha+\mu_2+\mu_4)+(\alpha+\mu_2)\mu_4]+\frac{chy_2}{(h+z_2)^2}(1-\rho)k\beta_1 x_2(\mu_1+\alpha+\mu_2)\\
		&+\frac{chy_2}{(h+z_2)^2}[\alpha\rho\beta_2 x_2(\mu_1+\mu_4)+\alpha\rho k\beta_1 x_2+(1-\rho)\beta_2 x_2(\mu_1\alpha+\mu_1\mu_2+\mu_1\mu_4+\alpha\mu_4+\mu_2\mu_4)]\\
		&-\frac{chy_2}{(h+z_2)^2}[(\beta_1 v_2+\beta_2 y_2+\mu_1)(\alpha+\mu_2)(\mu_3+\mu_4+az_2)+\mu_4(\beta_1 v_2+\beta_2 y_2+\mu_1)(\mu_3+az_2)]\\
		&-\frac{chy_2}{(h+z_2)^2}\mu_4(\alpha+\mu_2)(\mu_3+az_2),\\
		T_5=&\frac{acy_2z_2}{h+z_2}\mu_4(\beta_1 v_2+\beta_2 y_2+\mu_1)(\alpha+\mu_2)+\frac{chy_2}{(h+z_2)^2}(1-\rho)(\alpha+\mu_2)(\beta_2 x_2 \mu_1\mu_4+k\beta_1\mu_1 x_2)\\
		&+\frac{chy_2}{(h+z_2)^2}[\alpha\rho\mu_1 x_2(\beta_2\mu_4+k\beta_1)-\mu_4(\beta_1 v_2+\beta_2 y_2+\mu_1)(\alpha+\mu_2)(\mu_3+az_2)].
	\end{align*}
	
	Letting $\xi = iw(w > 0)$. Substituting $\xi = iw$ into Eq.(\ref{4.1}) and delineating the real and imaginary components, we derive
	\begin{equation}\label{4.2}
		\left\{
		\begin{aligned}
			&-I_1 w^4+I_3 w^2-I_5=(T_1 w^4-T_3 w^2+T_5)\cos w\tau_3+(-T_2 w^3+T_4 w)\sin w\tau_3, \\
			&w^5-I_2 w^3+I_4 w=(T_1 w^4-T_3 w^2+T_5)\sin w\tau_3-(-T_2 w^3+T_4 w)\cos w\tau_3.
		\end{aligned}
		\right.\end{equation}
The two equations of (\ref{4.2}) are squared and added to derive that
	\begin{equation}\label{4.3}
		w^{10}+A_1 w^8+A_2 w^6+A_3 w^4+A_4 w^2+A_5=0,	
	\end{equation}
	where
	$$\begin{aligned}
		A_1=&I_1^2-2I_2-T_1^2,\\
		A_2=&I_1^2+2I_4-2I_1I_3-T_2^2+2T_1T_3,\\
		A_3=&I_3^2+2I_1I_5-2I_2I_4-2T_1T_5+2T_2T_4-T_3^2,\\
		A_4=&I_4^2-2I_3I_5+2T_3T_5-T_4^2,\\
		A_5=&I_5^2-T_5^2.
	\end{aligned}
	$$
	Let $h = w^2$. Eq.(\ref{4.3}) becomes
	\begin{equation}\label{4.4}
		h^5+A_1 h^4+A_2 h^3+A_3 h^2+A_4 h+A_5=0.\end{equation}
	Let
	\begin{equation}\label{4.5}
		F(h)=h^5+A_1 h^4+A_2 h^3+A_3 h^2+A_4 h+A_5.
	\end{equation}
	The derivative of $F(h)$ is
	\begin{equation}\label{4.6}
		F'(h)=5h^4+4A_1 h^3+3A_2 h^2+2A_3 h+A_4.	
	\end{equation}
	Then let $F'(h)=0$, we can get
	\begin{equation}\label{4.7}
		5h^4+4A_1 h^3+3A_2 h^2+2A_3 h+A_4=0,	
	\end{equation}
	Denote
	$$\begin{aligned}
		&\hat{A}_1=-\frac{6A_1^{2}}{25}+\frac{3A_2}{5},\qquad\hat{A}_2=\frac{8A_1^{3}}{125}-\frac{6A_1 A_2}{25}+\frac{2A_3}{5},\\
		&\hat{A}_3=-\frac{3A_1^{4}}{625}+\frac{3A_1^{2}A_2}{125}-\frac{2A_1A_3}{25}+\frac{A_4}{5},\qquad\Delta_0=\hat{A}_1^{2}-4\hat{A}_3,\\
		&B_1 =-\frac{1}{3}\hat{A}_1^{2}-4\hat{A}_3,\qquad{B}_{2}=-\frac{2}{27}\hat{A}_1^{3}+\frac{8}{3}\hat{A}_1\hat{A}_3-\hat{A}_2^{2},  \\
		&\Delta_{1} =\frac{1}{27}B_1^{3}+\frac{1}{4}B_2^{2},\qquad B_3=\sqrt[3]{-\frac{B_2}{2}+\sqrt{\Delta_{1}}}+\sqrt[3]{-\frac{B_2}{2}-\sqrt{\Delta_{1}}}+\frac{\hat{A}_1}{3},  \\
		&\Delta_2 =-B_3-\hat{A}_1+\frac{2\hat{A}_2}{\sqrt{B_3-\hat{A}_1}},\qquad\Delta_{3}=-B_3-\hat{A}_1-\frac{2\hat{A}_2}{\sqrt{B_3-\hat{A}_1}}.
	\end{aligned}$$	
	Therefore, by applying Lemmas 2.1 to 2.4 from the reference \cite{WOS:000263372000045}, we confirm the presence of positive solutions for equation (\ref{4.4}) and derive the solutions to equation (\ref{4.7}), which we label as $ h_i $ for $ i $ ranging from 1 to 4.
	\newtheorem{lemma}{Lemma}[section]
	
Considering Eq.(\ref{4.4}) with $k$ positive roots (where $ k $ is an integer between 1 and 5), we denote these roots as $l_m $ for $ m = 1,2,\ldots, \kappa $. Consequently, there are also $ k $ positive roots for Eq.(\ref{4.3}), expressed as $ w_{m} = \sqrt{l_{m}} $.
	
	Following the Equation (\ref{4.2}), we have
	\begin{equation}\label{4.8}\begin{aligned}
			\cos w\tau_{3}& =\frac{(I_{1}w^{4}-I_{3}w^{2}+I_{5})(-T_{1}w^{4}+T_{3}w^{2}-T_{5})}{(T_{1}w^{4}-T_{3}w^{2}+T_{5})^{2}+(T_{2}w^{3}-T_{4}w)^{2}}-\frac{(w^{3}-I_{2}w^{3}+I_{4}w)(-T_{2}w^{3}+T_{4}w)}{(T_{1}w^{4}-T_{3}w^{2}+T_{5})^{2}+(T_{2}w^{3}-T_{4}w)^{2}} \\
			&\doteq\phi(w).
	\end{aligned}\end{equation}
	
	By solving for $\tau_3$ from Eq.(\ref{4.8}), we arrive at the expression
	
	$$\tau_{n}^{(m)} = \frac{1}{w_{m}}\left[\arccos(\phi(w)) + 2n\pi\right], \quad \text{for} \quad m = 1, 2, \ldots, \kappa \quad \text{and} \quad n = 0, 1, 2, \ldots.$$
In this context, $\pm w_m$ embody the imaginary parts of the solutions to the characteristic equation Eq.(\ref{4.1}). It is clear that, for any $m = 1,2,\ldots, \kappa$, the sequence $\tau_{n}^{(m)}$ increases monotonically as $n$ ranges from 0 to infinity, and furthermore, $\lim_{{n \to +\infty}} \tau_{n}^{(m)} = +\infty$. Therefore, there exist $m_{0}\in\left\{1,2,3,\ldots,\kappa\right\}$ and $n_{0}\in\left\{0,1,2,3,\ldots\right\}$ such that
	$$\tau_{n_0}^{(m_0)}=\min\biggl\{\tau_n^{(m)}:m=1,\:2,\:3,\:\ldots,\:\kappa;n=0,\:1,\:2,\:\ldots\biggr\}.$$
	Define
	\begin{equation}\label{4.9}
		\tau_{0} = \tau_{n_{0}}^{(m_{0})}, \quad w_{0} = w_{m_{0}}, \quad h_{0} = l_{m_{0}}. \end{equation}
	Adopting the methodology outlined in \cite{WOS:000263372000045}, we formulate the hypotheses
	\begin{equation}
		\begin{aligned}
			&(I_1+T_1)(I_2+T_2)-(I_3+T_3)>0,\\
			&I_1+T_1>0,\quad I_2+T_2>0,\quad I_3+T_3>0,\quad I_4+T_4>0,\quad I_5+T_5>0,\\
			&[(I_1+T_1)(I_2+T_2)-(I_3+T_3)](I_3+T_3)-(I_1+T_1)^2 (I_4+T_4)+(I_1+T_1)(I_5+T_5)>0.
		\end{aligned}\tag{H}\label{eq:H}
	\end{equation}
	Consider the root of the characteristic Eq.(\ref{4.1}), denoted as
	\begin{equation}\label{4.10}
		\xi(\tau_{3}) = \zeta(\tau_{3}) + iw(\tau_{3}),
	\end{equation}
	which fulfills the conditions $\zeta(\tau_{0}) = 0$, $w(\tau_{0}) = w_{0}$, and $h_0 = w_0^2$. Based on these conditions, we can derive the subsequent lemma.	
	\begin{lemma} \label{lemma1}
		Assuming $l_m = w_m^2$ and $F'(l_m) \neq 0$, where $F'(h)$ is as defined in Eq.(\ref{4.6}), we obtain the following equality of signs:
	$$sign\left(\left.\frac{d}{d\tau_{3}}Re(\xi(\tau_3))\right|_{\tau_{3}=\tau_{n}^{(m)}}\right) = sign\left(F'(h)\right).$$
	\end{lemma}
	\begin{proof}By differentiating Eq.(\ref{4.1}) with respect to $\tau_3$, we derive 
		$$\begin{aligned}
			&(5\xi^{4}+4I_{1}\xi^{3}+3I_{2}\xi^{2}+2I_{3}\xi+I_{4})\frac{d\xi}{d\tau_{3}}+(4T_{1}\xi^{3}+3T_{2}\xi^{2}+2T_{3}\xi+T_{4})e^{-\xi\tau_{3}}\frac{d\xi}{d\tau_{3}} \\
			&-(T_{1}\xi^{4}+T_{2}\xi^{3}+T_{3}\xi^{2}+T_{4}\xi+T_{5})\biggl(\xi+\tau_{3}\frac{d\xi}{d\tau_{3}}\biggr)e^{-\xi\tau_{3}}=0,
		\end{aligned}$$
		this implies that
		$$\begin{aligned}
			\left(\frac{d\xi}{d\tau_{3}}\right)^{-1} =&\frac{(5\xi^{4}+4I_{1}\xi^{3}+3I_{2}\xi^{2}+2I_{3}\xi+I_{4})e^{\xi\:\tau_{3}}}{\xi(T_{1}\xi^{4}+T_{2}\xi^{3}+T_{3}\xi^{2}+T_{4}\xi+T_{5})}  \\
			&+\frac{(4T_{1}\xi^{3}+3T_{2}\xi^{2}+2T_{3}\xi+T_{4})}{\xi(T_{1}\xi^{4}+T_{2}\xi^{3}+T_{3}\xi^{2}+T_{4}\xi+T_{5})}-\frac{\tau_{3}}{\xi}.
		\end{aligned}$$
		Thus,
		$$\begin{aligned}
			sign\biggl(\frac{dRe(\xi(\tau_{3}))}{d\tau_{3}}\biggr)\biggr|_{\tau_{3}=\tau_{n}^{(m)}} =&sign\Bigg(Re\Bigg(\frac{d\xi}{d\tau_{3}}\Bigg)^{-1}\Bigg)\Bigg|_{\xi=iw_{m}} \\
			=&sign\left(\frac{5w_{m}^{8}+(4I_{1}^{2}-8I_{2})w_{m}^{6}+(6I_{4}+3I_{2}-6I_{1}I_{3})w_{m}^{4}}{(-w_{m}^{4}+I_{2}w_{m}^{2}+I_{4})^{2}w_{m}^{2}+(I_{1}w_{m}^{4}-I_{3}w_{m}^{2}+I_{5})^{2}}\right. \\
			&+\frac{(4I_{1}I_{5}-4I_{2}I_{4}+2I_{3}^{2})w_{m}^{2}+(I_{4}^{2}-2I_{3}I_{5})}{(-w_{m}^{4}+I_{2}w_{m}^{2}+I_{4})^{2}w_{m}^{2}+(I_{1}w_{m}^{4}-I_{3}w_{m}^{2}+I_{5})^{2}} \\
			&+\frac{-4T_{1}^{2}w_{m}^{6}+(-3T_{2}^{2}+6T_{1}T_{3})w_{m}^{4}}{(T_{2}w_{m}^{2}-T_{4})^{2}w_{m}^{2}+(T_{1}w_{m}^{4}-T_{3}w_{m}^{2}+T_{5})^{2}} \\
			&+\frac{(4T_{2}T_{4}-4T_{1}T_{5}-2T_{3}^{2})w_{m}^{2}+(-T_{4}^{2}+2T_{3}T_{5})}{(T_{2}w_{m}^{2}-T_{4})^{2}w_{m}^{2}+(T_{1}w_{m}^{4}-T_{3}w_{m}^{2}+T_{5})^{2}}\Bigg).
		\end{aligned}$$
		From Eq.(\ref{4.3}), we can get
		$$\begin{matrix}(-w_m^4+I_2 w_m^2+I_4)^2 w_m^2+(I_1 w_m^4-I_3 w_m^2+I_5)^2=(T_2 w_m^2-T_4)^2 w_m^2+(T_1 w_m^4-T_3 w_m^2+T_5)^2.\end{matrix}$$
		From Eq.(\ref{4.5}), we have
		$$\begin{aligned}F^{\prime}(l_{m})|_{l_{m}=w_{m}^{2}}&=5w_{m}^{8}+4(I_{1}^{2}-2I_{2}-T_{1}^{2})w_{m}^{6}+3(I_{2}^{2}+2I_{4}-2I_{1}I_{3}-T_{2}^{2}+2T_{1}T_{3})w_{m}^{4}\\&+2(I_{3}^{2}+2I_{1}I_{5}-2I_{2}I_{4}-2T_{1}T_{5}+2T_{2}T_{4}-T_{3}^{2})w_{m}^{2}+(I_{4}^{2}-2I_{3}I_{5}+2T_{3}T_{5}-T_{4}^{2}).\end{aligned}$$
		Hence,
		$$sign\biggl(\frac{dRe(\xi(\tau_3))}{d\tau_3}\biggr)\biggr|_{\tau_3=\tau_n^{(m)}}=sign\biggl(\frac{F^{\prime}(l_m)}{(T_2 w_m^2-T_4)^2 w_m^2+(T_1 w_m^4-T_3 w_m^2+T_5)^2}\biggr),$$
	this implies that  $sign\left(\frac{dRe(\xi(\tau_{3}))}{d\tau_{3}}\right)\Bigg|_{\tau_{3}=\tau_{n}^{(m)}}=sign\big(F^{\prime}(h)\big).$
	\end{proof}
	By applying the Hopf bifurcation theorem and Lemmas 2.1 through 2.5 from \cite{WOS:000263372000045}, We derive the subsequent theorem.
	\begin{theorem} \label{lemma1}
		Define $h_0$, $w_0$, and $\tau_0$ as given in Eq.(\ref{4.9}) and Eq.(\ref{4.10}). Assume \eqref{eq:H} holds.
		
		\begin{enumerate}
			\item[(i)] If any of the following conditions is not satisfied:
			
			\begin{enumerate}
				\item[(a)] $A_5 < 0$;
				
				\item[(b)] $A_5 \geq 0$, $\hat{A}_2 = 0$, $\Delta \geq 0$, and $\hat{A}_1 < 0$ or $\hat{A}_3 \leq 0$, and there exists at least one $h^{*} \in \{h_1, h_2, h_3, h_4\}$ such that $h^{*} > 0$ and $F(h^{*}) \leq 0$;
				
				\item[(c)] $A_5 \geq 0$, $\hat{A}_2 \neq 0$, $B_3 > \hat{A}_1$, and $\Delta_{2} \geq 0$ or $\Delta_{3} \geq 0$, and there exists at least one $h^{*} \in \{h_1, h_2, h_3, h_4\}$ such that $h^{*} > 0$ and $F(h^{*}) \leq 0$;
				
				\item[(d)] $A_5 \geq 0$, $\hat{A}_2 \neq 0$, $B_3 < \hat{A}_1$, $\frac{\hat{A}_2^2}{4(\hat{A}_1 - B_3)^2} + \frac{B_3}{2} = 0$, $\bar{h} > 0$, and $F(\bar{h}) \leq 0$, where $\bar{h} = \frac{\hat{A}_2}{2(\hat{A}_1 - B_3)} - \frac{A_1}{3}$, then $E_2$ is locally asymptotically stable for $\tau_3 \geq 0$.
			\end{enumerate}
			
			\item[(ii)] If any of the conditions (a)--(d) from (i) are satisfied, then the immune equilibrium $E_2$ is locally asymptotically stable for $\tau_{3} \in [0, \tau_0)$.
			
			\item[(iii)] If one of the conditions (a)--(d) is met and $F' (h_{0})$ is not zero, the model (\ref{1.2}) experiences a Hopf bifurcation from the immune balance $E_{2}$ as $\tau_3$ crosses the pivotal value $\tau_0$.
		\end{enumerate}
	\end{theorem}
\subsection{Direction and stability of Hopf bifurcation}	
To further explore the characteristics of the Hopf bifurcation at $E_2$, we employ the center manifold theory and normal form theory \cite{Hassard1981} to calculate the quantities $\Gamma_1$, $\Gamma_2$, and $T$. The procedure is analogous to that in \cite{Xu2022,bp}. Without loss of generality, suppose $\tau_3 = \tau_0 + \varepsilon$. When $\varepsilon = 0$, the Hopf bifurcation occurs at $E_2$. Define $u_1(t) = x(t) - x_2$, $u_2(t) = p(t) - p_2$, $u_3(t) = y(t) - y_2$, $u_4(t) = v(t) - v_2$, and $u_5(t) = z(t) - z_2$. As a result, system (\ref{1.2}) can be reformulated as a functional differential equation, which is written as
\begin{equation}\label{4.11}
\dot{u}(t) = L_{\varepsilon}(u_t) + F(\varepsilon,u_t),
\end{equation}
where $u(t)=(u_1(t),u_2(t),u_3(t),u_4(t),u_5(t))^{T}\in C^1([-\tau_3,0],\mathbb{R}^5)$, $u_t(l)=u(t+l)$ for $l \in[-\tau_3, 0]$. Note that $L_{\varepsilon} : C([-\tau_3,0],\mathbb{R}^5)\to \mathbb{R}^5$ and $F : \mathbb{R} \times C([-\tau_3,0],\mathbb{R}^5) \to \mathbb{R}^5$. For $\phi\in C([-\tau_3,0],\mathbb{R}^5)$, define
$$
L_{\varepsilon}(\phi) = A_{\max}\phi(0) + B_{\max}\phi(-\tau_3),$$
where
	$$A_{\max}=\begin{pmatrix}
		-(\beta_1 v_2+\beta_2 y_2)-\mu_1 & 0& -\beta_2 x_2 &-\beta_1 x_2 & 0 \\
		\rho(\beta_1 v_2 +\beta_2 y_2) & -\alpha-\mu_2 & \rho\beta_2 x_2& \rho\beta_1 x_2& 0 \\
		(1-\rho)(\beta_1 v_2 +\beta_2 y_2) & \alpha& (1-\rho)\beta_2 x_2-\mu_3-az_2 & (1-\rho)\beta_1 x_2 & -ay_2 \\
		0 & 0 & k & -\mu_4 & 0 \\
		0 & 0 & -\eta z_2 & 0&-\mu_5-\eta y_2
	\end{pmatrix},$$
	$$B_{\max}=\begin{pmatrix}
		0 & 0 & 0 & 0 & 0 \\
		0 & 0 & 0 & 0 & 0 \\
		0 & 0 & 0& 0 & 0 \\
		0 & 0 & 0& 0 & 0 \\
		0 & 0 &\dfrac{c z_2}{h+z_2} & 0 & \dfrac{cy_2 h}{(h+z_2)^2}
	\end{pmatrix},$$
and
$$F(\varepsilon,\phi) = \begin{pmatrix}
-\phi_1(0)[\beta_1 \phi_4(0)+\beta_2\phi_3(0)] \\
\rho\phi_1(0)[\beta_1 \phi_4(0)+\beta_2\phi_3(0)] \\
(1-\rho)\phi_1(0)[\beta_1 \phi_4(0)+\beta_2\phi_3(0)]-a\phi_3(0)\phi_5(0)\\
0 \\
\dfrac{c\phi_3(-\tau_3)\phi_5(-\tau_3)}{h+z_2}+\dfrac{cy_2 z_2}{(h+z_2)^3}{\phi_5}^2(-\tau_3)-\eta\phi_3(0)\phi_5(0)
\end{pmatrix}.$$
By the Riesz representation theorem, there exists a function \(\eta(l, \varepsilon)\) with components of bounded variation for $l \in [-\tau_3, 0]$ such that the linear operator $L_{\varepsilon}$ can be expressed as  
$$
L_{\varepsilon} \phi = \int_{-\tau_3}^0 \mathrm{d}\eta(l,\varepsilon) \, \phi(l).
$$  
We choose$$\eta(l,\varepsilon) =A_{\max} \delta(l) + B_{\max} \delta(l +\tau_3),$$where $\delta(l)$ represents the Dirac function, $\delta(l)=\left\{ \begin{array} {c}{0,l\neq0} \\ {1,l=0} \end{array}.\right. $ 

From now on, we allow functions taking values in $\mathbb{C}^5$ rather than in $\mathbb{R}^5$.   For $\phi\in C^1([-\tau_3,0],\mathbb{C}^5)$, we define$$A(\varepsilon)\phi(l) =\begin{cases} 	\dfrac{\mathrm{d}\phi(l)}{\mathrm{d}l}, & l \in [-\tau_3, 0), \\ 	\int_{-\tau_3}^0 \mathrm{d}\eta(l, \varepsilon) \phi(l) = L_{\varepsilon} \phi, & l = 0, \end{cases}$$and$$ H(\varepsilon)\phi(l) =\begin{cases} 	0, & l \in [-\tau_3, 0), \\ 	F(\varepsilon, \phi), & l = 0. \end{cases}$$So we can rewrite system (\ref{4.11}) as follows: \begin{equation}\label{4.12}\dot{u}_t=A(\varepsilon)u_t+H(\varepsilon)u_t. \end{equation}  For $\varphi \in C^1([0,\tau_3],\mathbb{C}^5) $, define$$A^*\varphi(s) = \begin{cases} -\dfrac{\mathrm{d}\varphi(s)}{\mathrm{d}s}, & s\in (0,\tau_3], \\ \int_{-\tau_3}^0\mathrm{d}\eta^T(s, \varepsilon) \varphi(s), & s=0, \end{cases} $$and a bilinear inner product\begin{equation}\label{4.13}	\langle\varphi,\phi\rangle = \overline{\varphi}^{T}(0)\phi(0) - \int_{l=-\tau_3}^0\int_{\xi=0} ^{l}\overline{\varphi}^{T}(\xi-l)[\mathrm{d}\eta(l)] \phi(\xi)\mathrm{d}\xi,\end{equation}where $\eta(l)=\eta(l,0)$, and $A$ and $A^*$ are the adjoint operators. We have shown that $\pm iw_0$ are the eigenvalues of $A(0)$; obviously $\pm iw_0$ are also the eigenvalues of $A^*(0)$. Then compute the eigenvector $\rho(l)=(1,\rho_2,\rho_3,\rho_4,\rho_5)^{T}e^{iw_0 l}$ and $\rho^*(s)=\overline{T}(1,\rho_2^*,\rho_3^*,\rho_4^*,\rho_5^*)^{T}e^{iw_0 s}$, which are the eigenvectors of $A(0)$ associated with $iw_0$ and $A^*(0)$ associated with $-iw_0$, respectively. Then, we have$$(A_{\max}+B_{\max}e^{-iw_0\tau_0}-iw_0)\rho(0)=0,$$that is$$\begin{cases}	-(\mu_1 + \beta_1 v_2 + \beta_2 y_2 + iw_0) - \beta_2 x_2 \rho_3 - \beta_1 x_2 \rho_4 = 0, \\	\rho(\beta_1 v_2 + \beta_2 y_2) - (\alpha + \mu_2 + iw_0) \rho_2 + \rho\beta_2 x_2 \rho_3 + \rho\beta_1 x_2 \rho_4 = 0, \\	(1-\rho)(\beta_1 v_2 + \beta_2 y_2) + \alpha\rho_2 + [(1-\rho)\beta_2 x_2 - \mu_3 - a z_2 - iw_0] \rho_3 + (1-\rho)\beta_1 x_2 \rho_5 - a y_2 \rho_5 = 0, \\	k\rho_3 - (\mu_4 + iw_0)\rho_4 = 0, \\	\left(-\eta z_2 + \frac{c z_2}{h+z_2}e^{-iw_0\tau_0}\right)\rho_3 + \left(\frac{c y_2 h}{(h+z_2)^2} e^{-iw_0\tau_0} - \mu_5 - iw_0 - \eta y_2\right)\rho_5 = 0.\end{cases}$$Then,$$\begin{aligned}	\rho_2&=\frac{\rho(\beta_1 v_2+\beta_2 y_2)} {\alpha+\mu_2+iw_0}-\frac{[\rho\beta_2 x_2(\mu_4+iw_0)+k\rho\beta_1 x_2]\rho_3}{(\alpha+\mu_2+iw_0)(\mu_4+iw_0)},\\	\rho_3&=-\frac{(\mu_1+\beta_1 v_2 +\beta_2 y_2 +iw_0)(\mu_4+iw_0)} {[\beta_2 x_2(\mu_4+iw_0)+\beta_1 x_2 ]},\\ \rho_4&=\frac{k}{\mu_4+iw_0}\rho_3,\\	\rho_5&=-\frac {[-\eta z_2(h+z_2)^2+cz_2(h+z_2) e^{-iw_0\tau_0}] \rho_3}{[cy_2h e^{-iw_0\tau_0}-(h+z_2)^2 (\mu_5+\eta y_2+iw_0)]}.	\end{aligned}$$
 Similarly, calculate $\rho^{*}(s)$, there is $(A_{\mathrm{max}}^{T}+B_{\mathrm{max}}^{T}e^{iw_0\tau_0}+iw_0)\rho^{*}(0)=0,$ and we can obtain 
$$
	\begin{cases}
		 -(\mu_1+\beta_1v_2+\beta_2 y_2-iw_0)+\rho(\beta_1v_2+\beta_2 y_2)\rho_2^*+(1-\rho)(\beta_1v_2+\beta_2 y_2)\rho_3^*=0, \\
		 [-(\alpha+\mu_2)+iw_0]\rho_2^*+\alpha\rho_3^*=0,\\
		-\beta_2 x_2+\rho \beta_2 x_2\rho_2^*+[(1-\rho)\beta_2 x_2-\mu_3-a z_2+iw_0]\rho_3^*+k\rho_4^*+\left(-\eta z_2+\frac{cz_2}{h+z_2}e^{iw_0\tau_0}\right)\rho_5^*=0,\\
	-\beta_1 x_2+\rho \beta_1 x_2\rho_2^*+(1-\rho)\beta_1 x_2 \rho_3^*+(iw_0-\mu_4)\rho_4^*=0,\\
		-ay_2\rho_3^*+\left(\frac{cy_2 h}{(h+z_2)^2}e^{iw_0\tau_0}+iw_0-\mu_5-\eta y_2\right)\rho_5^*=0,
	\end{cases}
$$
then get
$$\begin{aligned}
	\rho_2^*&=-\frac{\alpha}{[iw_0-(\alpha+\mu_2)]}\rho_3^*,\\
	\rho_3^*&=\frac{(\mu_1+\beta_1v_2+\beta_2 y_2-iw_0)[iw_0-(\alpha+\mu_2)]}{(1-\rho)(\beta_1 v_2+\beta_2 y_2)[iw_0-(\alpha+\mu_2)]-\alpha\rho(\beta_1 v_2+\beta_2 y_2)},\\
	\rho_4^*&=\frac{\beta_1 x_2-\rho \beta_1 x_2\rho_2^*-(1-\rho)\beta_1 x_2 \rho_3^*}{(iw_0-\mu_4)},\\ \rho_5^*&=\frac{ay_2(h+z_2)^2\rho_3^*}{cy_2 he^{iw_0\tau_0}+(iw_0-\mu_5-\eta y_2)(h+z_2)^2}.
\end{aligned}$$ 
According to $\langle \rho^*, \rho \rangle = 1$, $\langle \rho^*, \overline{\rho} \rangle = 0$, and Eq.(\ref{4.13}), define the constant $\overline{T}$,
$$\begin{aligned}
	\langle\rho^{*},\rho\rangle & =\overline{T}\left(1+\rho_{2}\overline{\rho_{2}^{*}}+\rho_{3}\overline{\rho_{3}^{*}}+\rho_{4}\overline{\rho_{4}^{*}}+\rho_{5}\overline{\rho_{5}^{*}}\right) \\
	& -\int_{l=-\tau_0}^{0}\int_{\xi=0}^{l}\overline{T}(1,\overline{\rho_{2}^{*}},\overline{\rho_{3}^{*}},\overline{\rho_{4}^{*}},\overline{\rho_{5}^{*}})e^{-iw_{0}(\xi-l)}[d\eta(l)](1,\rho_{2},\rho_{3},\rho_{4},\rho_{5})^{T}e^{iw_{0}\xi}d\xi \\
	& =\overline{T}\left\{1+\sum_{k=2}^{5}\rho_{k}\overline{\rho_{k}^{*}}+\tau_{0}e^{-iw_{0}\tau_{0}}(1,\overline{\rho_{2}^{*}},\overline{\rho_{3}^{*}},\overline{\rho_{4}^{*}},\overline{\rho_{5}^{*}})B_{\max}(1,\rho_{2},\rho_{3},\rho_{4},\rho_{5})^{T}\right\} \\
	& =\overline{T}\left\{1+\sum_{k=2}^5\rho_k\rho_k^*+\tau_0\rho_5^*\left[\frac{cz_2}{h+z_2}e^{-iw_0\tau_0}\rho_3+\frac{chy_2}{(h+z_2)^2}e^{-iw_0\tau_0}\rho_5\right]\right\}=1,
\end{aligned}$$
and $\overline{T}=\left\{1+\sum_{k=2}^5\rho_k\rho_k^*+\tau_0\rho_5^*\left[\frac{cz_2}{h+z_2}e^{-iw_0\tau_0}\rho_3+\frac{chy_2}{(h+z_2)^2}e^{-iw_0\tau_0}\rho_5\right]\right\}^{-1}$.

Subsequently, we adopt the subsequent notation, if $\varepsilon= 0$, we deem $u_t$ the solution to Eq.(\ref{4.11}) and proceed to determine the coordinates here to describe the central manifold $C_0$, let's denote this as follows
\begin{equation}\label{4.14}
	z(t) = \langle \rho^*, u_t \rangle,\qquad W(t,l) =u_t-z\rho-\bar{z}\bar{\rho}= u_t - 2 \operatorname{Re} \{z(t) \rho(l)\}.
\end{equation}
On $C_0$, there exists $W(t,l) = W(z(t), \overline{z}(t), l)$, and
\begin{equation}\label{4.15}
	W(z(t), \overline{z}(t), l) = W_{20}(l)\frac{z^2}{2} + W_{11}(l)z\overline{z} + W_{02}(l)\frac{\overline{z}^2}{2} +\cdots,
\end{equation}
since $z(t)$ and $\overline{z}(t)$ are local coordinates for center manifold $C_0$ in the directions of $\rho^*$ and $\overline{\rho^*}$, then we can get
$$\begin{aligned} &\dot{z}(t) = \langle \rho^{*}, \dot{u}_{t} \rangle = \langle \rho^{*}, A(0)u_{t} + H(0)u_{t} \rangle, \\ &\qquad = i w_{0}z(t) + \overline{\rho^{*}}^{T}(0) \cdot F(0, W(z(t), \overline{z}(t), 0) + 2 \text{Re}\{z(t)\rho(0)\}), \\ &\qquad \triangleq i w_{0}z(t) + \overline{\rho^{*}}^{T}(0) \cdot f_{0}(z(t), \overline{z}(t)), \end{aligned}$$
where $f_{0}(z,\overline{z})=F(0,W(z,\overline{z},\theta)+z(t)q(\theta)+\overline{z}(t)\overline{q}(\theta)).$ Note that
$$
f_{0}=f_{20}\frac{z^{2}}{2}+f_{11}z\overline{z}+f_{02}\frac{\overline{z}^{2}}{2}+f_{21}\frac{z^{2}\overline{z}}{2}+\cdots,$$
we can rewrite as
$$
\dot{z}(t) = i w_{0}z(t) + g(z(t), \overline{z}(t)),$$
where
\begin{equation}\label{4.16}
	g(z(t), \overline{z}(t)) = \overline{\rho^{*}}^{T}(0) \cdot f_{0}(z(t),\overline{z}(t)) = g_{20}\frac{z^{2}}{2} + g_{11}z\overline{z} + g_{02}\frac{\overline{z}^{2}}{2} + g_{21}\frac{z^{2}\overline{z}}{2} + \cdots.
\end{equation}
From Eqs.(\ref{4.14}) and (\ref{4.15}), it can be obtained that:
$$\begin{aligned}
	&u_t=W(t,l)+2\mathrm{Re}\{z(t)\rho(\iota)\} \\
 &=W_{20}(l)\frac{z^{2}}{2}+W_{11}(l)z\overline{z}+W_{02}(l)\frac{\overline{z}^{2}}{2}+\rho^{T}e^{iw_{0}l}z(t)+\overline{\rho}^{T}e^{-iw_{0}l}\overline{z}(t)+\cdots, 
 \end{aligned}
 $$
 then, from the equation of $F(\varepsilon, \phi)$, we have
 \begin{equation}\label{4.17}
 	\begin{aligned}
 		&g(z(t),\overline{z}(t)) = \overline{\rho^{*}}^{T}(0) \cdot f_{0}(z(t),\overline{z}(t)) \\
 		&= (1,\overline{\rho_2^*},\overline{\rho_3^*},\overline{\rho_4^*},\overline{\rho_5^*})\overline{T}
 		\begin{pmatrix}
 			-\phi_{1}(0)(\beta_1\phi_{4}(0) + \beta_2\phi_{3}(0)) \\
 			\rho\phi_{1}(0)(\beta_1\phi_{4}(0) + \beta_2\phi_{3}(0)) \\
 			(1-\rho)\phi_{1}(0)(\beta_1\phi_{4}(0) + \beta_2\phi_{3}(0)) - a\phi_{3}(0)\phi_{5}(0) \\
 			0 \\
 			\dfrac{c\phi_{3}(-\tau_3)\phi_{5}(-\tau_3)}{h+z_2}+\dfrac{cy_2 z_2}{(h+z_2)^3}{\phi_{5}}^2(-\tau_3)-\eta\phi_{3}(0)\phi_{5}(0)
 		\end{pmatrix},
 	\end{aligned}
 \end{equation}
where
 $$\begin{aligned}
 	& \phi_{1}(0)=z+\overline{z}+W_{20}^{(1)}(0)\frac{z^{2}}{2}+W_{11}^{(1)}(0)z\overline{z}+W_{02}^{(1)}(0)\frac{\overline{z}^{2}}{2}+\cdots, \\
 	& \phi_{3}(0)=\rho_{3}z+\overline{\rho}_{3}\overline{z}+W_{20}^{(3)}(0)\frac{z^{2}}{2}+W_{11}^{(3)}(0)z\overline{z}+W_{02}^{(3)}(0)\frac{\overline{z}^{2}}{2}+\cdots, \\
 	& \phi_{4}(0)=\rho_{4}z+\overline{\rho}_{4}\overline{z}+W_{20}^{(4)}(0)\frac{z^{2}}{2}+W_{11}^{(4)}(0)z\overline{z}+W_{02}^{(4)}(0)\frac{\overline{z}^{2}}{2}+\cdots,\\
 	& \phi_{5}(0)=\rho_{5}z+\overline{\rho}_{5}\overline{z}+W_{20}^{(5)}(0)\frac{z^{2}}{2}+W_{11}^{(5)}(0)z\overline{z}+W_{02}^{(5)}(0)\frac{\overline{z}^{2}}{2}+\cdots,\\
 	& \phi_{3}(-\tau_3)=\rho_{3}z+\overline{\rho}_{3}\overline{z}+W_{20}^{(3)}(-\tau_3)\frac{z^{2}}{2}+W_{11}^{(3)}(-\tau_3)z\overline{z}+W_{02}^{(3)}(-\tau_3)\frac{\overline{z}^{2}}{2}+\cdots, \\
 	& \phi_{5}(-\tau_3)=\rho_{5}z+\overline{\rho}_{5}\overline{z}+W_{20}^{(5)}(-\tau_3)\frac{z^{2}}{2}+W_{11}^{(5)}(-\tau_3)z\overline{z}+W_{02}^{(5)}(-\tau_3)\frac{\overline{z}^{2}}{2}+\cdots,
 \end{aligned}$$
 based on Eqs.(\ref{4.16}) and (\ref{4.17}),
 $$
 \begin{aligned}
 	g(z(t),\overline{z}(t))=&\overline{T}\biggr\{\phi_{1}(0)[\rho\overline{\rho_2^*}+(1-\rho)\overline{\rho_2^*}-1][\beta_1\phi_{4}(0)+\beta_2\phi_{3}(0)]-(a\overline{\rho_3^*}+\eta\overline{\rho_5^*})\phi_{3}(0)\phi_{5}(0)\\&+\frac{c\overline{\rho_5^*}}{h+z_2}\phi_{3}(-\tau_3)\phi_{5}(-\tau_3)+\frac{cy_2 z_2 \overline{\rho_5^*}}{(h+z_2)^3}\phi_5^2(-\tau_3)\biggr\}\\
 \end{aligned}
 $$	
$$
\begin{aligned}
	=&\overline{T}\biggr\{\biggr([-\beta_1+\overline{\rho_2^*}\rho\beta_1+\overline{\rho_3^*}(1-\rho)\beta_1]\rho_4+(-a\overline{\rho_3^*}-\eta\overline{\rho_5^*})\rho_3\rho_5\\
	&+[-\beta_2+\overline{\rho_2^*}\rho\beta_2+\overline{\rho_3^*}(1-\rho)\beta_2]\rho_3+\frac{c\overline{\rho_5^*}}{h+z_2}\rho_3\rho_5+\frac{cy_2 z_2\overline{\rho_5^*}}{(h+z_2)^3}\rho_5^2\biggr)z^{2}\\
	&+\biggr([-\beta_1+\overline{\rho_2^*}\rho\beta_1+\overline{\rho_3^*}(1-\rho)\beta_1](\overline{\rho}_4+\rho_4)+2\frac{cy_2 z_2\overline{\rho_5^*}}{(h+z_2)^3}\rho_5\overline{\rho}_5\\
	&+(\frac{c\overline{\rho_5^*}}{h+z_2}-a\overline{\rho_3^*}-\eta\overline{\rho_5^*})(\overline{\rho}_3\rho_5+\rho_3\overline{\rho}_5)+[-\beta_2+\overline{\rho_2^*}\rho\beta_2+\overline{\rho_3^*}(1-\rho)\beta_2](\rho_3+\overline{\rho}_3)\biggr)z\overline{z}\\
	&+\biggr(-a\overline{\rho_3^*}-\eta\overline{\rho_5^*})\overline{\rho}_3\overline{\rho}_5+\frac{c\overline{\rho_5^*}}{h+z_2}\overline{\rho}_3\overline{\rho}_5+\frac{cy_2 z_2\overline{\rho_5^*}}{(h+z_2)^3}{\overline{\rho}_5}^2\\
	&+[-\beta_1+\overline{\rho_2^*}\rho\beta_1+\overline{\rho_3^*}(1-\rho)\beta_1]\overline{\rho}_4+[-\beta_2+\overline{\rho_2^*}\rho\beta_2+\overline{\rho_3^*}(1-\rho)\beta_2]\overline{\rho}_3\biggr){\overline{z}^2}\\
	&+(-a\overline{\rho_3^*}-\eta\overline{\rho_5^*})(W_{11}^{(5)}(0)\rho_3+W_{11}^{(3)}(0)\rho_5+\frac{\overline{\rho}_{3}}{2}W_{20}^{(5)}(0)+\frac{\overline{\rho}_{5}}{2}W_{20}^{(3)}(0))z^2\overline{z}\\
	&+[-\beta_1+\overline{\rho_2^*}\rho\beta_1+\overline{\rho_3^*}(1-\rho)\beta_1](W_{11}^{(1)}(0)\rho_4+W_{11}^{(4)}(0)+\frac{\overline{\rho}_{4}}{2}W_{20}^{(1)}(0)+\frac{1}{2}W_{20}^{(4)}(0))z^2\overline{z}\\
	&+[-\beta_2+\overline{\rho_2^*}\rho\beta_2+\overline{\rho_3^*}(1-\rho)\beta_2](W_{11}^{(1)}(0)\rho_3+W_{11}^{(3)}(0)+\frac{\overline{\rho}_{3}}{2}W_{20}^{(1)}(0)+\frac{1}{2}W_{20}^{(3)}(0))z^2\overline{z}\\
	&+\frac{c\overline{\rho_5^*}}{h+z_2}(W_{11}^{(5)}(-\tau_3)\rho_3+W_{11}^{(3)}(-\tau_3)\rho_5+\frac{\overline{\rho}_{3}}{2}W_{20}^{(5)}(-\tau_3)+\frac{\overline{\rho}_{5}}{2}W_{20}^{(3)}(-\tau_3))z^2\overline{z}\\
	&+\frac{cy_2 z_2\overline{\rho_5^*}}{(h+z_2)^3}(2W_{11}^{(5)}(-\tau_3)\rho_5+W_{20}^{(5)}(-\tau_3)\overline{\rho}_5)z^2\overline{z}+\cdots\biggr\},
\end{aligned}
$$
where
$$\begin{aligned}
	g_{20}=&2\overline{T}\biggr([-\beta_1+\overline{\rho_2^*}\rho\beta_1+\overline{\rho_3^*}(1-\rho)\beta_1]\rho_4+(-a\overline{\rho_3^*}-\eta\overline{\rho_5^*})\rho_3\rho_5+\frac{c\overline{\rho_5^*}}{h+z_2}\rho_3\rho_5+\frac{cy_2 z_2\overline{\rho_5^*}}{(h+z_2)^3}\rho_5^2\\
	&+[-\beta_2+\overline{\rho_2^*}\rho\beta_2+\overline{\rho_3^*}(1-\rho)\beta_2]\rho_3\biggr),\\
	g_{11}=&\overline{T}\biggr([-\beta_1+\overline{\rho_2^*}\rho\beta_1+\overline{\rho_3^*}(1-\rho)\beta_1](\overline{\rho}_4+\rho_4)+[-\beta_2+\overline{\rho_2^*}\rho\beta_2+\overline{\rho_3^*}(1-\rho)\beta_2](\rho_3+\overline{\rho}_3)\\
	&+(\frac{c\overline{\rho_5^*}}{h+z_2}-a\overline{\rho_3^*}-\eta\overline{\rho_5^*})(\overline{\rho}_3\rho_5+\rho_3\overline{\rho}_5)+2\frac{cy_2 z_2\overline{\rho_5^*}}{(h+z_2)^3}\rho_5\overline{\rho}_5\biggr),\\
	g_{02}=&2\overline{T}\biggr([-\beta_1+\overline{\rho_2^*}\rho\beta_1+\overline{\rho_3^*}(1-\rho)\beta_1]\overline{\rho}_4+[-\beta_2+\overline{\rho_2^*}\rho\beta_2+\overline{\rho_3^*}(1-\rho)\beta_2]\overline{\rho}_3-(a\overline{\rho_3^*}+\eta\overline{\rho_5^*})\overline{\rho}_3\overline{\rho}_5\\
	&+\frac{c\overline{\rho_5^*}}{h+z_2}\overline{\rho}_3\overline{\rho}_5+\frac{cy_2 z_2\overline{\rho_5^*}}{(h+z_2)^3}{\overline{\rho}_5}^2\biggr),\\	
	g_{21}=&\frac{y_2 z_2\overline{\rho_5^*}c\overline{T}}{(h+z_2)^3}(4W_{11}^{(5)}(-\tau_3)\rho_5+2W_{20}^{(5)}(-\tau_3)\overline{\rho}_5)\\
	&+\frac{\overline{\rho_5^*}c\overline{T}}{h+z_2}(2W_{11}^{(5)}(-\tau_3)\rho_3+2W_{11}^{(3)}(-\tau_3)\rho_5+W_{20}^{(5)}(-\tau_3){\overline{\rho}_{3}}+W_{20}^{(3)}(-\tau_3){\overline{\rho}_{5}})\\
	&+\overline{T}(-a\overline{\rho_3^*}-\eta\overline{\rho_5^*})(2W_{11}^{(5)}(0)\rho_3+2W_{11}^{(3)}(0)\rho_5+W_{20}^{(5)}(0)\overline{\rho}_{3}+W_{20}^{(3)}(0)\overline{\rho}_{5})\\
	\end{aligned}
$$
$$\begin{aligned}
&+\overline{T}[-\beta_1+\overline{\rho_2^*}\rho\beta_1+\overline{\rho_3^*}(1-\rho)\beta_1](2W_{11}^{(1)}(0)\rho_4+2W_{11}^{(4)}(0)+W_{20}^{(1)}(0)\overline{\rho}_{4}+W_{20}^{(4)}(0))\\
&+\overline{T}[-\beta_2+\overline{\rho_2^*}\rho\beta_2+\overline{\rho_3^*}(1-\rho)\beta_2](2W_{11}^{(1)}(0)\rho_3+2W_{11}^{(3)}(0)+W_{20}^{(1)}(0)\overline{\rho}_{3}+W_{20}^{(3)}(0)).
\end{aligned}
$$
To calculate the value of $g_{21}$, we need to calculate the value of $W_{20}(l)$ and $W_{11}(l)$. From (\ref{4.12}) and (\ref{4.14}), we have
\begin{equation}\label{4.18}
\dot{W} = \dot{u}_t - \dot{z}\rho -\dot{\bar{z}} \bar{\rho} =
\begin{cases}
	A(0) W - 2 \operatorname{Re} \{\overline{\rho^{*}}(0) \cdot F_{0}(z, \overline{z})\rho(l)\}, & l\in [-\tau_3,0), \\
	A(0) W -2 \operatorname{Re} \{\overline{\rho^{*}}(0) \cdot F_{0}(z, \overline{z})\rho(l)\} + f_0, & l= 0.
\end{cases}
\end{equation}
From (\ref{4.15}), we obtain
\begin{equation}\label{4.19}
	\begin{aligned}
		\dot{W}= \dot{W}_{z}\dot{z}+\dot{W}_{\overline{z}}\dot{\overline{z}}=&(W_{20}(l)z+W_{11}(l)\overline{z}+\cdots)(iw_{0}z(t)+g(z,\overline{z})) \\
		&+(W_{11}(l)z+W_{02}(l)\overline{z}+\cdots)(-iw_{0}\overline{z}(t)+\overline{g}(z,\overline{z})). 
\end{aligned}\end{equation}
Substituting (\ref{4.15}) and (\ref{4.19}) into (\ref{4.18}), and comparing the coefficients of $z^2$ and $z \overline{z}$, we have
\begin{equation}\label{4.20}
	(2i \omega_0 I - A(0)) W_{20}(l) =
\begin{cases}
	- g_{20} \rho(l) - \overline{g}_{02} \overline{\rho}(l), & l\in [-\tau_3,0), \\
	- g_{20} \rho(l) - \overline{g}_{02} \overline{\rho}(l) + f_{20}, & l = 0,
\end{cases}
\end{equation}
and
\begin{equation}\label{4.21}
- A(0) W_{11}(l) =
\begin{cases}
	- g_{11} \rho(l) - \overline{g}_{11}\overline{\rho}(l), & l\in [-\tau_3,0), \\
	- g_{11} \rho(l) - \overline{g}_{11} \overline{\rho}(l) + f_{11}, &l= 0.
\end{cases}
\end{equation}

From the definition of $A(0)$ when $l\in [-\tau_3,0)$, (\ref{4.20}) and (\ref{4.21}), we have
$$\dot{W}_{20} = 2 i w_0 W_{20}(l) + g_{20} \rho(l) + \overline{g}_{02} \overline{\rho}(l),
$$
and
$$\dot{W}_{11} = g_{11} \rho(l) z + \overline{g}_{11} \overline{\rho}(l).
$$
So we can obtain
\begin{equation}\label{4.22}
	W_{20}(l)=\frac{ig_{20}\rho(0)}{w_{0}}e^{iw_{0}l}+\frac{i\overline{g}_{02}\overline{\rho}(0)}{3w_{0}}e^{-iw_{0}l}+H_{1}e^{2iw_{0}l},
\end{equation}
\begin{equation}\label{4.23}
	W_{11}(l)=-\frac{ig_{11}\rho(0)}{w_{0}}e^{iw_{0}l}+\frac{i\overline{g}_{11}\overline{\rho}(0)}{w_{0}}e^{-iw_{0}l}+H_{2},
\end{equation}
where $H_{i}=(H_{i}^{(1)},H_{i}^{(2)},H_{i}^{(3)},H_{i}^{(4)},H_{i}^{(5)})^{T}\in R^{5}$, $i=1,2$ are constant vectors.
Next, we calculate the value of $H_{1}$ and $H_{2}$. From the definition of $A(0)$, when $l=0$ and (\ref{4.20}), we obtain
\begin{equation}\label{4.24}
	{\int_{-\tau_0}^{0}}d\eta(l)W_{20}(l)=2iw_{0}W_{20}(0)+g_{20}\rho(0)+\overline{g}_{02}\overline{\rho}(0)-f_{20}.
\end{equation}
Substituting (\ref{4.22}) into (\ref{4.24}), and note that $iw_{0}I-{\int_{-\tau_0}^{0}}e^{iw_{0}l}d\eta(l)\rho(0)=0$, we have
$$\left(2iw_{0}I-{\int_{-\tau_0}^{0}}e^{2iw_{0}l}d\eta(l)\right)H_{1}=2\left(\begin{array}{c}-(\beta_1\rho_4+\beta_2\rho_3)\\\rho(\beta_1\rho_4+\beta_2\rho_3)\\(1-\rho)(\beta_1\rho_4+\beta_2\rho_3)-a\rho_3\rho_5\\0\\\dfrac{c}{h+z_2}\rho_3\rho_5+\dfrac{cy_2 z_2}{(h+z_2)^3}\rho_5^2-\eta\rho_3\rho_5\end{array}\right),$$
which leads to
	$$
	H_1 = 2{D_1}^{-1}
		\begin{pmatrix}
			-(\beta_1\rho_4+\beta_2\rho_3)\\
			\rho(\beta_1\rho_4+\beta_2\rho_3)\\
			(1-\rho)(\beta_1\rho_4+\beta_2\rho_3)-a\rho_3\rho_5
			\\
			0\\
			\dfrac{c}{h+z_2}\rho_3\rho_5+\dfrac{cy_2 z_2}{(h+z_2)^3}\rho_5^2-\eta\rho_3\rho_5
		\end{pmatrix}.
	$$
The matrix $D_1$ has the following structure:
 $$D_1=\begin{pmatrix}
		2i\omega_0 + D_{11} & 0 & \beta_2 x_2 & \beta_1 x_2 & 0 \\
		-\rho(\beta_1 v_2 + \beta_2 y_2) & 2i\omega_0 + D_{22} & -\rho\beta_2 x_2 & -\rho\beta_1 x_2 & 0 \\
		-(1-\rho)(\beta_1 v_2 + \beta_2 y_2) & -\alpha & 2i\omega_0+D_{33} & -(1-\rho)\beta_1 x_2 & a y_2 \\
		0 & 0 & -k & 2i\omega_0 + D_{44} & 0 \\
		0 & 0 & \eta z_2 - \dfrac{c z_2 e^{-2i\omega_0\tau_0}}{h+z_2} & 0 & 2i\omega_0 +D_{55}
	\end{pmatrix}.$$
where $D_{11}=\mu_1 + \beta_1 v_2 + \beta_2 y_2$, $D_{22}=\alpha + \mu_2$, $D_{33}=- \bigl[(1-\rho)\beta_2 x_2 - \mu_3 - a z_2\bigr]$, $D_{44}=\mu_4$, $D_{55}= \mu_5+\eta y_2- \dfrac{c y_2 h e^{-2i\omega_0\tau_0}}{(h+z_2)^2}$.

Similarly, from the definition of $A(0)$, when $l = 0$ and (\ref{4.21}), we obtain
\begin{equation}\label{4.25}
	\int_{-\tau_0}^{0} d\eta(l) W_{11}(l) = g_{11}\rho(0) + \overline{g}_{11} \overline{\rho}(0) - f_{11}.
\end{equation}
Substituting (\ref{4.23}) into (\ref{4.25}), and note that $-iw_0I - \int_{-\tau_0}^{0} e^{-iw_0l}d\eta(l) \overline{\rho}(0) = 0$, we have
$$
\int_{-\tau_0}^{0} d\eta(l) H_2 = -\begin{pmatrix}
	-[\beta_1(\overline{\rho}_4+\rho_4)+\beta_1(\overline{\rho}_3+\rho_3)]\\
	\rho[\beta_1(\overline{\rho}_4+\rho_4)+\beta_1(\overline{\rho}_3+\rho_3)]\\
	(1-\rho)[\beta_1(\overline{\rho}_4+\rho_4)+\beta_1(\overline{\rho}_3+\rho_3)]-a(\rho_3\overline{\rho}_5+\overline{\rho}_3\rho_5)\\
	0\\
	\dfrac{c}{h+z_2}(\overline{\rho}_3\rho_5+\rho_3\overline{\rho}_5)+2\dfrac{cy_2 z_2}{(h+z_2)^3}\rho_5\overline{\rho}_5-\eta(\rho_3\overline{\rho}_5+\overline{\rho}_3\rho_5)
\end{pmatrix},$$ which leads to
$$
H_2=-
D_2^{-1}
\begin{pmatrix}
	-[\beta_1(\overline{\rho}_4+\rho_4)+\beta_1(\overline{\rho}_3+\rho_3)]\\
	\rho[\beta_1(\overline{\rho}_4+\rho_4)+\beta_1(\overline{\rho}_3+\rho_3)]\\
	(1-\rho)[\beta_1(\overline{\rho}_4+\rho_4)+\beta_1(\overline{\rho}_3+\rho_3)]-a(\rho_3\overline{\rho}_5+\overline{\rho}_3\rho_5)\\
	0\\
	\dfrac{c}{h+z_2}(\overline{\rho}_3\rho_5+\rho_3\overline{\rho}_5)+2\dfrac{cy_2 z_2}{(h+z_2)^3}\rho_5\overline{\rho}_5-\eta(\rho_3\overline{\rho}_5+\overline{\rho}_3\rho_5)
\end{pmatrix},
$$
where
$$
	D_2=\begin{pmatrix}
		D_{11} & 0 & \beta_2 x_2 & \beta_1 x_2 & 0 \\
	-\rho(\beta_1 v_2 + \beta_2 y_2) &  D_{22} & -\rho\beta_2 x_2 & -\rho\beta_1 x_2 & 0 \\
	-(1-\rho)(\beta_1 v_2 + \beta_2 y_2) & -\alpha & D_{33} & -(1-\rho)\beta_1 x_2 & a y_2 \\
	0 & 0 & -k & D_{44} & 0 \\
	0 & 0 & \eta z_2 - \dfrac{c z_2}{h+z_2} & 0 & \mu_5+\eta y_2 - \dfrac{c y_2 h}{(h+z_2)^2}
\end{pmatrix}.$$

Next, we derive the explicit expressions for $\Gamma_1$, $\Gamma_2$, and $T$, which are given by:

$$\begin{aligned}
	& C_{1}(0)=\frac{i}{2w_{0}}\left(g_{20}g_{11}-2|g_{11}|^{2}-\frac{|g_{02}|^{2}}{3}\right)+\frac{g_{21}}{2},\qquad\Gamma_1=-\frac{\mathrm{Re}(C_{1}(0))}{\mathrm{Re}(\xi^{\prime}(\tau_{0}))}, \\
	& \Gamma_2=2\mathrm{Re}(C_{1}(0)),\qquad T=-\frac{\mathrm{Im}\{C_{1}(0)\}+\Gamma_1\mathrm{Im}\left\{\xi^{\prime}(\tau_{0})\right\}}{w_{0}}.
\end{aligned}$$
	\begin{theorem} \label{lemma1}
For system (\ref{1.2}), there holds

\begin{enumerate}
	\item[(i)] If $\Gamma_1>0$ (respectively, $\Gamma_1<0$), the Hopf bifurcation is supercritical (subcritical), and a branch of periodic orbits emerges for $\tau>\tau_0$.
	
	\item[(ii)] If $\Gamma_2<0$ (respectively, $\Gamma_2>0$), the resulting periodic orbits are asymptotically stable (unstable).
	
	\item[(iii)] If $T<0$ (respectively, $T>0$), the period of the bifurcating orbits decreases (increases).
\end{enumerate}
\end{theorem}

	\section{Numerical simulation}
	
This study component employed computer simulations to assess findings and explore the impact of the CTL immune response on the system, utilizing the DDE23 Matlab Package for delay differential equations. The following are the basic parameters:
	$$\beta_1=0.001, \beta_2=0.001, \rho=0.4, m_1=4, m_2=4, \alpha=0.05, a=0.2,$$
	$$k=0.02, h=0.01, \mu_1=0.01,\mu_2=0.02,\mu_3=0.04,\mu_4=0.02,\mu_5=0.005,$$
	and the initial condition is chosen as $(x^0, p^0, y^0, v^0, z^0) = (75, 1, 1, 1, 0.5)$.
	
	First, we select $\lambda=7.5$, $c=0.05$, $\eta=0.003$, $\tau_1=\tau_2=1$, and $\tau_3=2$. This yields the equilibrium point $E_0 = (750, 0, 0, 0, 0)$ and the basic reproduction number $\mathcal{R}_0 = 0.2920 < 1$. According to Theorem 3.4, $E_0$ is globally asymptotically stable. As illustrated in Figure 3, the curve $x(t)$ representing uninfected cells converges to 750, while the other populations tend to zero, confirming the stability of $E_0$ and validating Theorem 3.4. In this scenario, CTL activation eliminates both infected cells and viral particles.
	
Subsequently, the values $\lambda = 75$, $\eta = 0.003$, and $c = 0.005$ are selected, yielding the equilibrium point $E_1 = (2568.3, 4.5164, 9.6011, 9.6011, 0)$, with $\mathcal{R}_0 = 2.9202 > 1$ and $\mathcal{R}_1 = 0.1623 < 1$. Based on Theorem 3.5, the equilibrium $E_1$ exhibits global asymptotic stability, which implies that the CTL immune response is insufficient to clear infected cells, leading to its eventual inactivation. Consequently, a stable equilibrium is reached where uninfected cells, cells in latent infection, cells in active infection, and viral particles coexist in a balanced manner (see Figure 4).
	
	To confirm the validity of Theorem 3.6, we set $\lambda=7.5$, $c=0.005$, $\eta=0.003$, $\tau_1=\tau_2=0.25$, and $\tau_3=0$. This yields $\mathcal{R}_1 = 5.8654 > 1$. According to Theorem 5.3, $E_2 = (708.6670, 0.7603, 0.2916, 0.2916, 0.4864)$ is globally asymptotically stable. Figure 5 illustrates the triggering of the immune response, enabling a stable balance among uninfected, latent infected, infected cells, viruses, and CTLs.
	
	As a final point, we investigate how $\tau_3$ affects the equilibrium \(E_2\) by examining the following two cases: (1) when $\tau_1$ and $\tau_2$ are both zero, while $\tau_3$ is non-zero; (2) when $\tau_1$, $\tau_2$, and $\tau_3$ are all non-zero. In the second scenario, with $\lambda=3$, $\eta=0.003$, and $c=0.01$, we obtain $\mathcal{R}_0 = 6.3776 > 1$, $\mathcal{R}_1 = 2.3033 > 1$, and $E_2 = (266.4163, 1.6792, 0.6303, 0.6303, 0.9047)$. For $\tau_3 = 70$, Figure 6 illustrates the stability of the immunity-activated equilibrium, denoted as $E_2$. Increasing $\tau_3$ to 100, Figure 7 confirms its stability. However, further increasing $\tau_3$ to 101, Figure 8 reveals that $E_2$ experiences a loss of stability, resulting in a Hopf bifurcation at $E_2$, giving rise to a periodic solution. For $\tau_3 = 120$, numerical simulations reveal intricate periodic oscillations inside the system (see Figure 9).
\begin{figure}[htbp] 
	\centering
	\begin{subfigure}{0.3\textwidth}
		\centering
		\includegraphics[width=4.2cm,height=3.15cm]{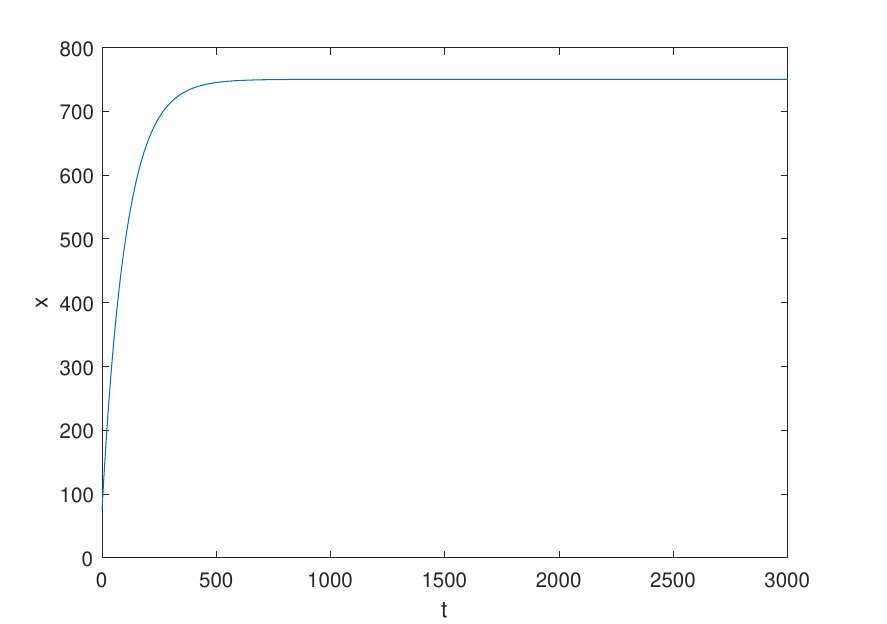}
		\caption{Temporal series of x(t)}
		\label{fig:image1}
	\end{subfigure}%
	\hfill
	\begin{subfigure}{0.3\textwidth}
		\centering
		\includegraphics[width=4.2cm,height=3.15cm]{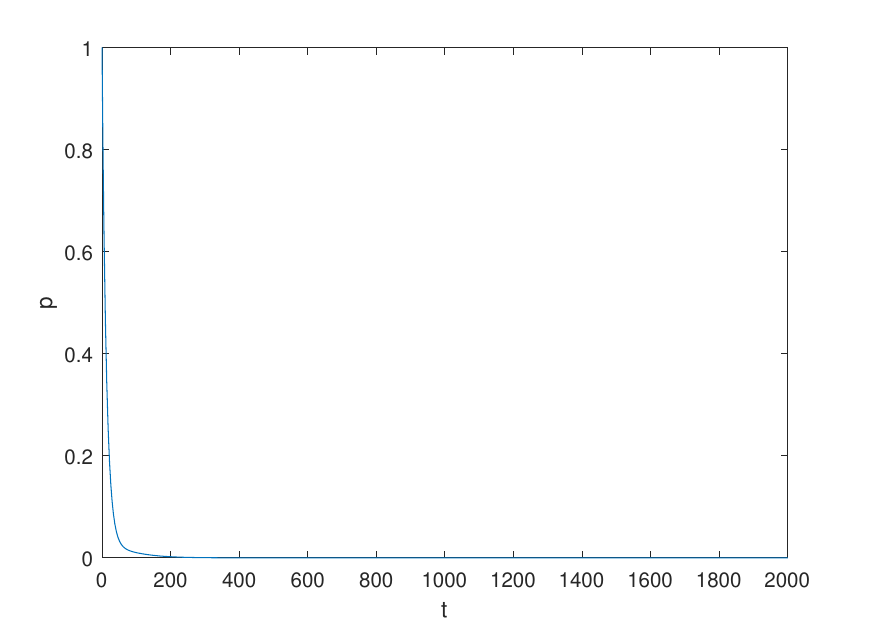}
		\caption{Temporal series of p(t)}
		\label{fig:image2}
	\end{subfigure}%
	\hfill
	\begin{subfigure}{0.3\textwidth}
		\centering
		\includegraphics[width=4.2cm,height=3.15cm]{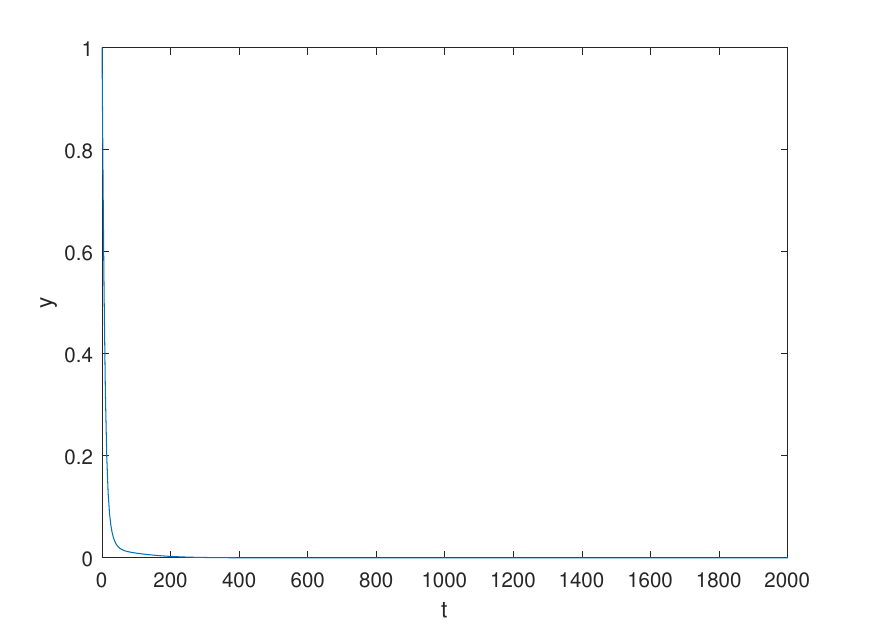}
		\caption{Temporal series of y(t)}
		\label{fig:image3}
	\end{subfigure}
	
	\begin{subfigure}{0.3\textwidth}
		\centering
		\includegraphics[width=4.2cm,height=3.15cm]{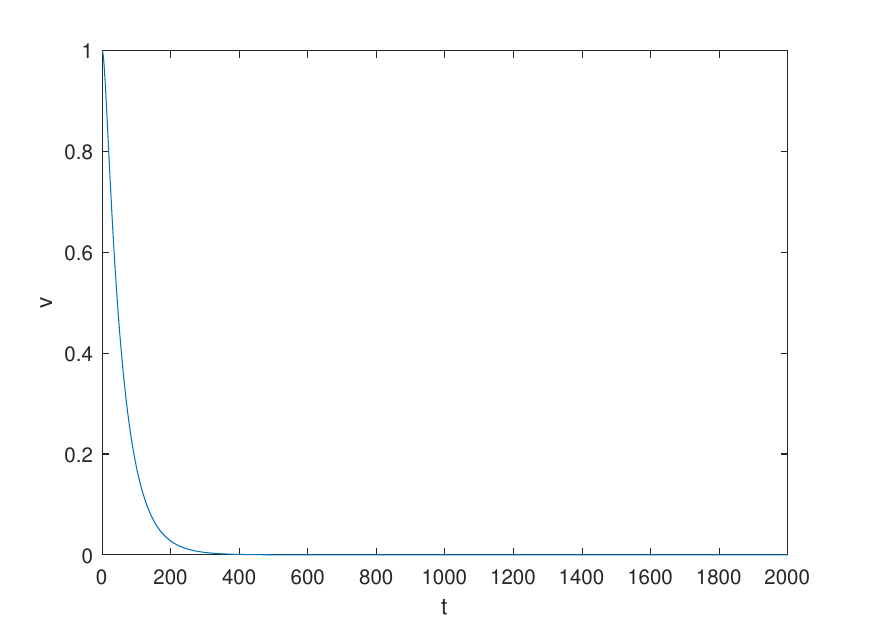}
		\caption{Temporal series of v(t)}
		\label{fig:image4}
	\end{subfigure}%
	\hfill
	\begin{subfigure}{0.3\textwidth}
		\centering
		\includegraphics[width=4.2cm,height=3.15cm]{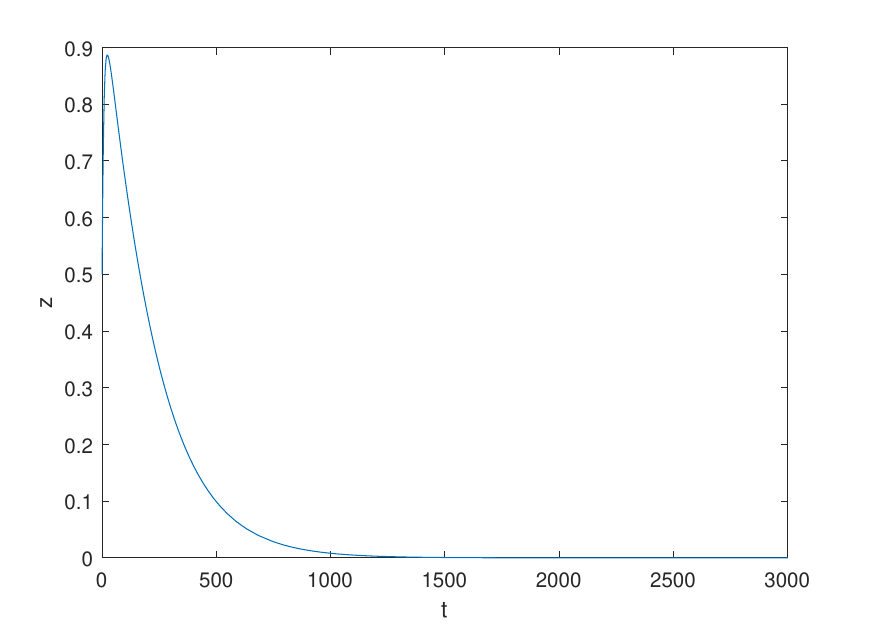}
		\caption{Temporal series of z(t)}
		\label{fig:image5}
	\end{subfigure}%
	\hfill
	\begin{subfigure}{0.3\textwidth}
		\centering
		\includegraphics[width=4.2cm,height=3.15cm]{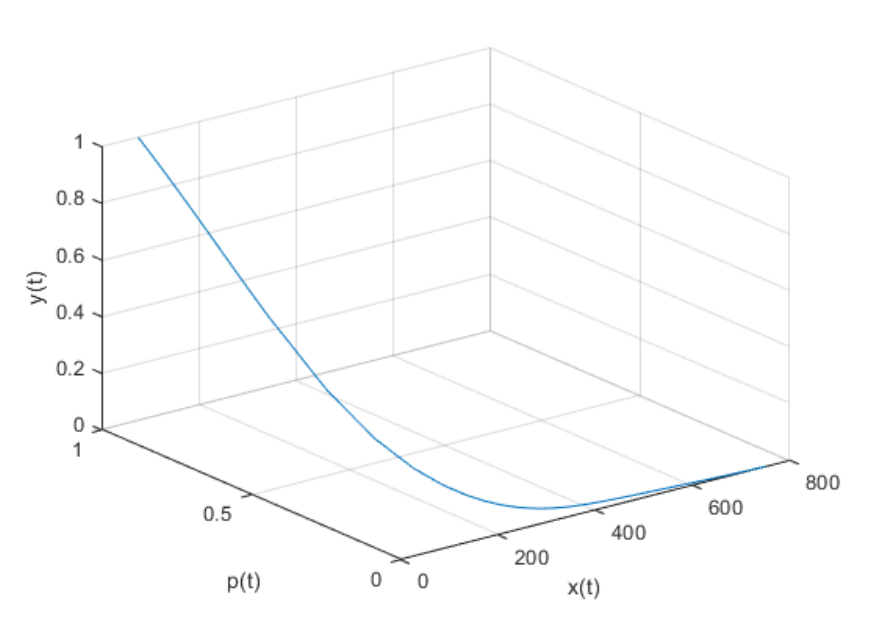}
		\caption{3D phase for x,p,y}
		\label{fig:image6}
	\end{subfigure}
	
	\caption{Simulation result of $E_0$ while $\tau_1=1$, $\tau_2=1$ and$\tau_3=2$}
	\label{fig}
\end{figure}
\begin{figure}[H] 
	\centering
	\begin{subfigure}{0.3\textwidth}
		\centering
		\includegraphics[width=4.2cm,height=3.15cm]{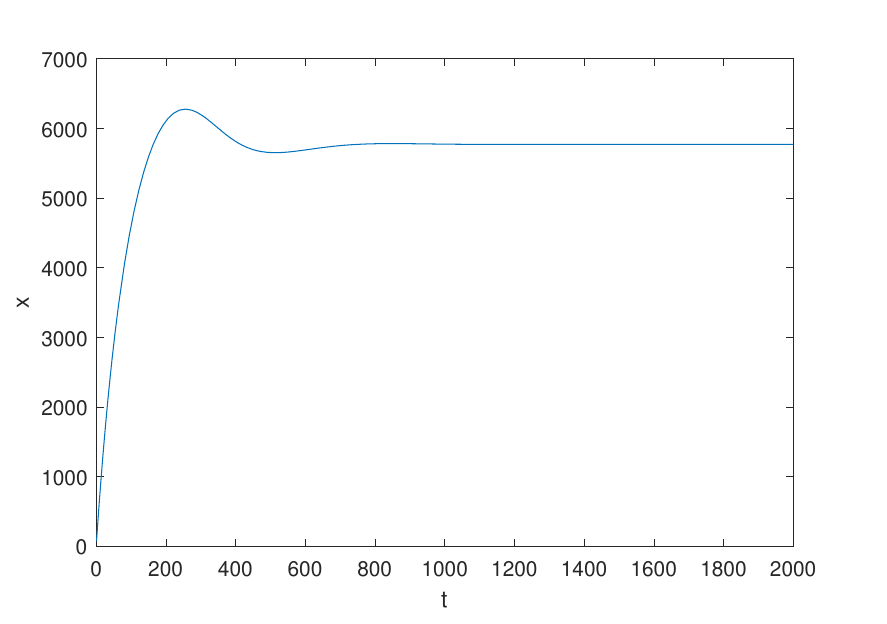}
		\caption{Temporal series of x(t)}
		\label{fig:image1}
	\end{subfigure}%
	\hfill
	\begin{subfigure}{0.3\textwidth}
		\centering
		\includegraphics[width=4.2cm,height=3.15cm]{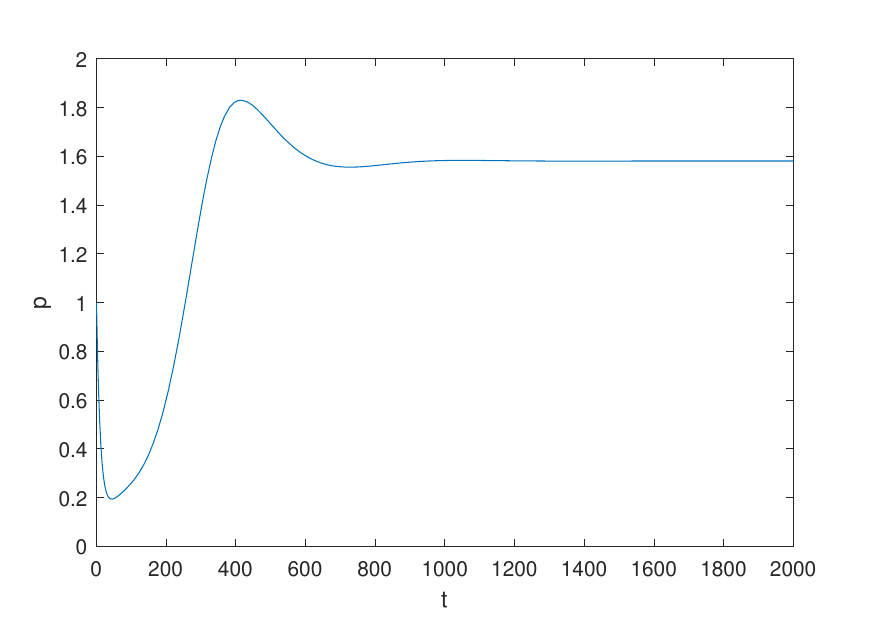}
		\caption{Temporal series of p(t)}
		\label{fig:image2}
	\end{subfigure}%
	\hfill
	\begin{subfigure}{0.3\textwidth}
		\centering
		\includegraphics[width=4.2cm,height=3.15cm]{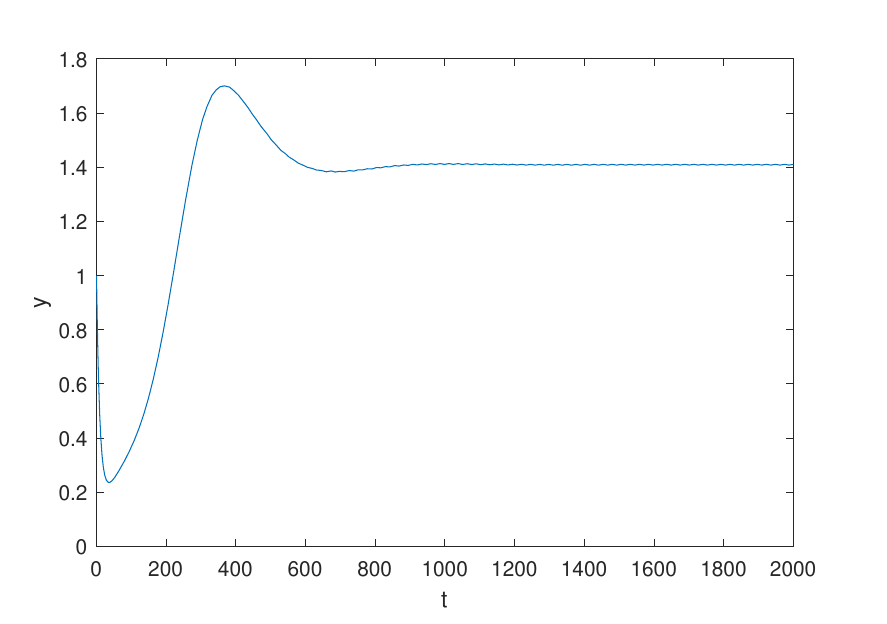}
		\caption{Temporal series of y(t)}
		\label{fig:image3}
	\end{subfigure}
	
	\begin{subfigure}{0.3\textwidth}
		\centering
		\includegraphics[width=4.2cm,height=3.15cm]{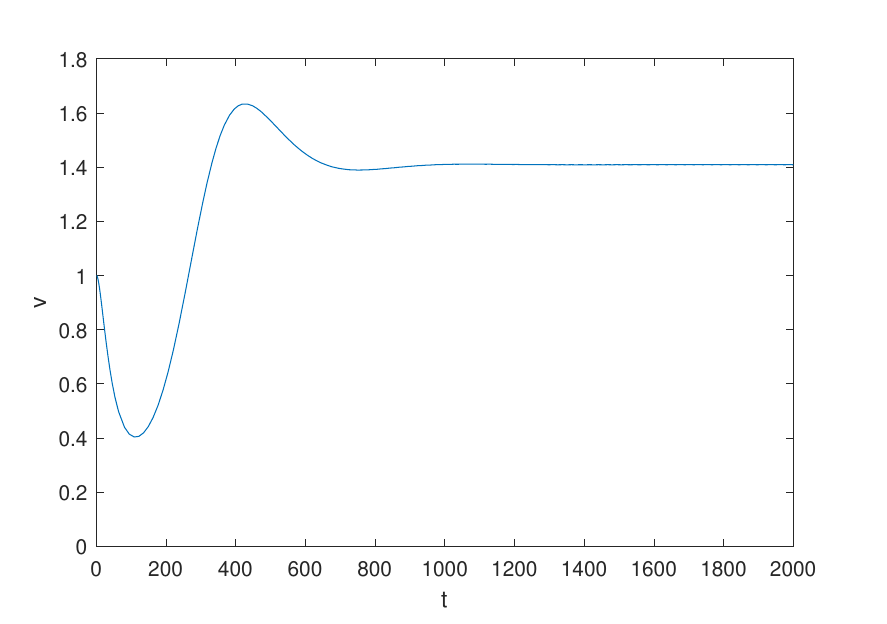}
		\caption{Temporal series of v(t)}
		\label{fig:image4}
	\end{subfigure}%
	\hfill
	\begin{subfigure}{0.3\textwidth}
		\centering
		\includegraphics[width=4.2cm,height=3.15cm]{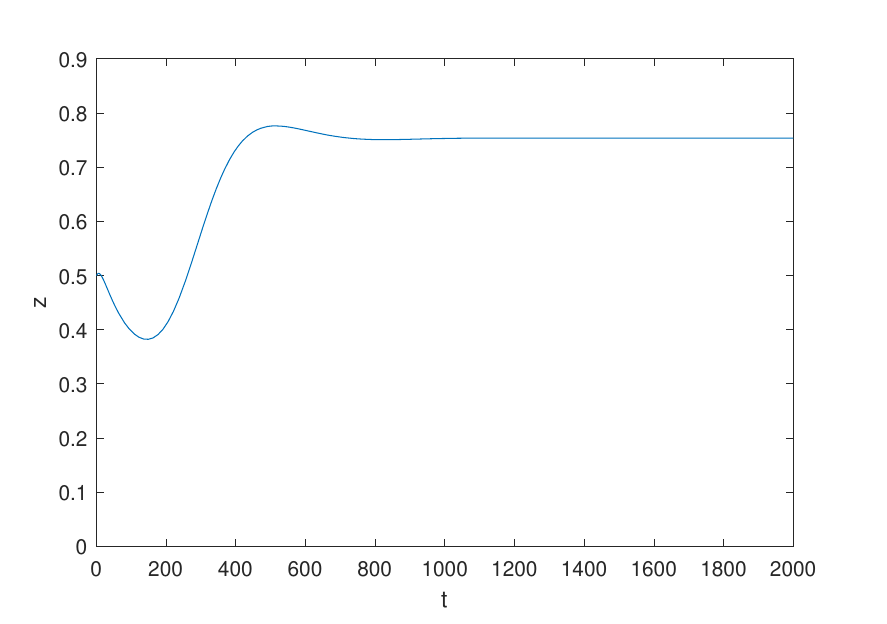}
		\caption{Temporal series of z(t)}
		\label{fig:image5}
	\end{subfigure}%
	\hfill
	\begin{subfigure}{0.3\textwidth}
		\centering
		\includegraphics[width=4.2cm,height=3.15cm]{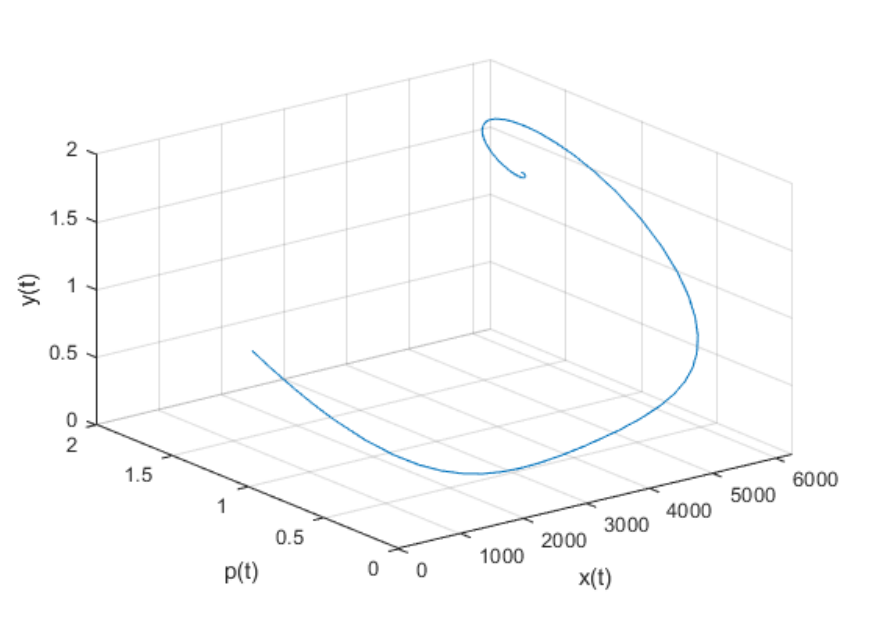}
		\caption{3D phase for x,p,y}
		\label{fig:image6}
	\end{subfigure}
	
	\caption{Simulation result of $E_1$ while $\tau_1=1$, $\tau_2=1$ and$\tau_3=2$}
	\label{fig}
\end{figure}
\begin{figure}[htbp] 
	\centering
	\begin{subfigure}{0.3\textwidth}
		\centering
		\includegraphics[width=4.2cm,height=3.15cm]{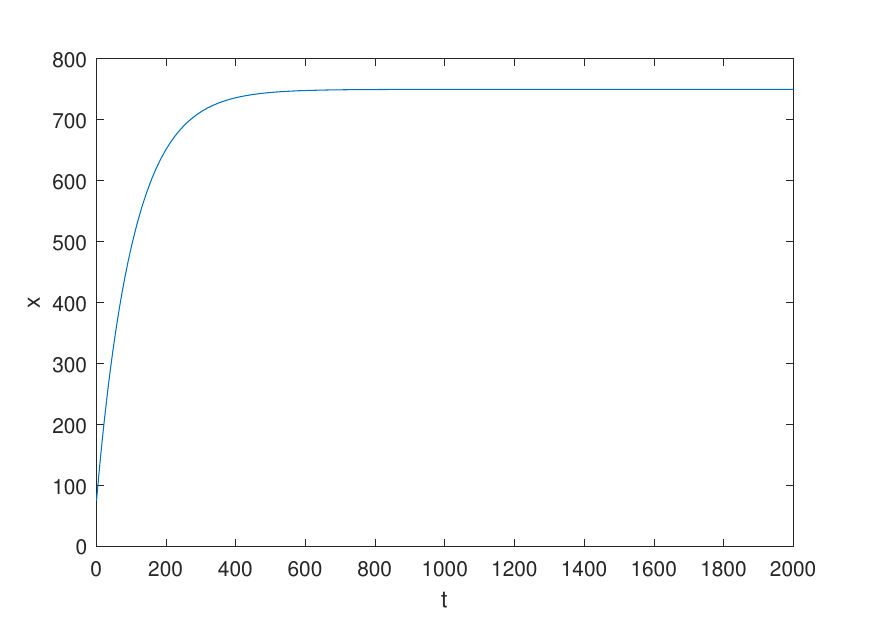}
		\caption{Temporal series of x(t)}
		\label{fig:image1}
	\end{subfigure}%
	\hfill
	\begin{subfigure}{0.3\textwidth}
		\centering
		\includegraphics[width=4.2cm,height=3.15cm]{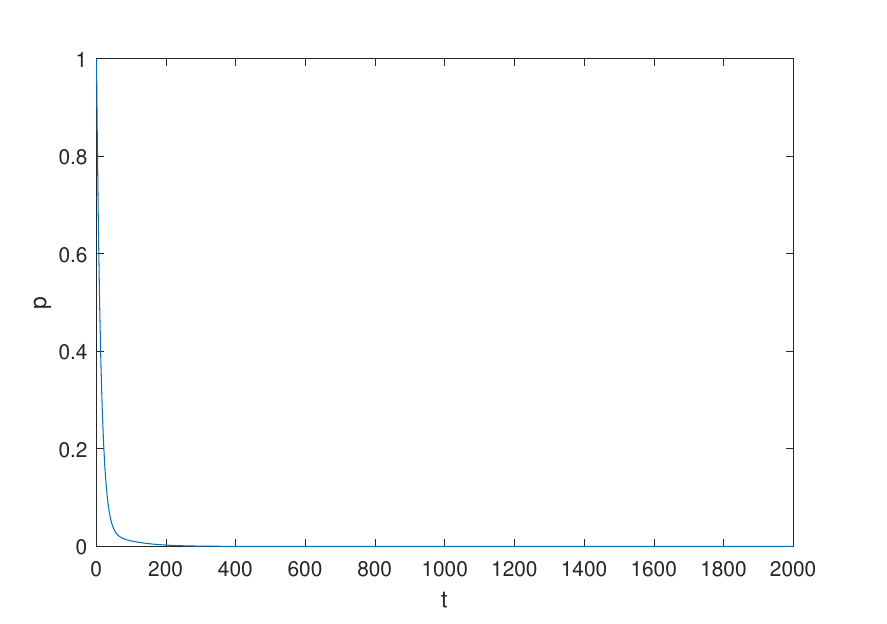}
		\caption{Temporal series of p(t)}
		\label{fig:image2}
	\end{subfigure}%
	\hfill
	\begin{subfigure}{0.3\textwidth}
		\centering
		\includegraphics[width=4.2cm,height=3.15cm]{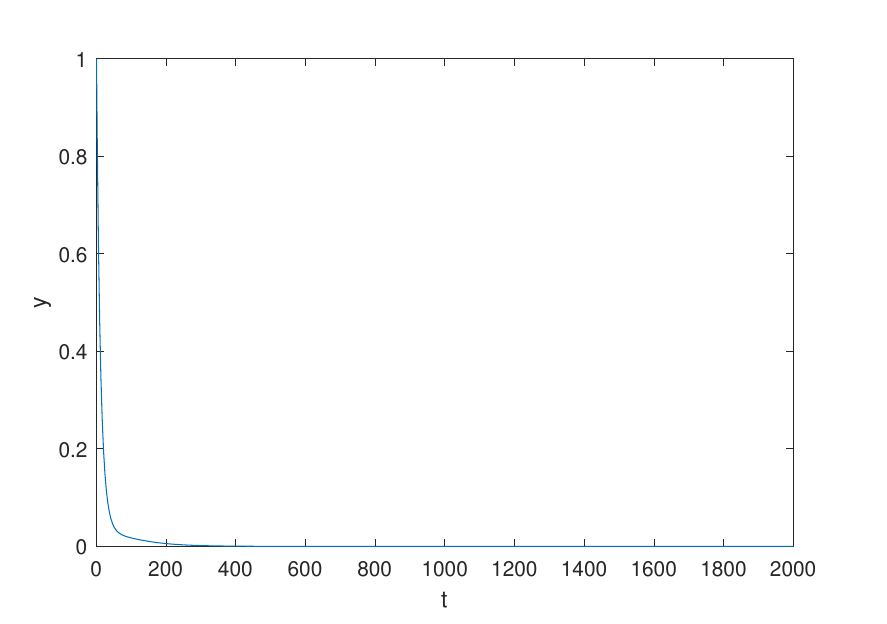}
		\caption{Temporal series of y(t)}
		\label{fig:image3}
	\end{subfigure}
	
	\begin{subfigure}{0.3\textwidth}
		\centering
		\includegraphics[width=4.2cm,height=3.15cm]{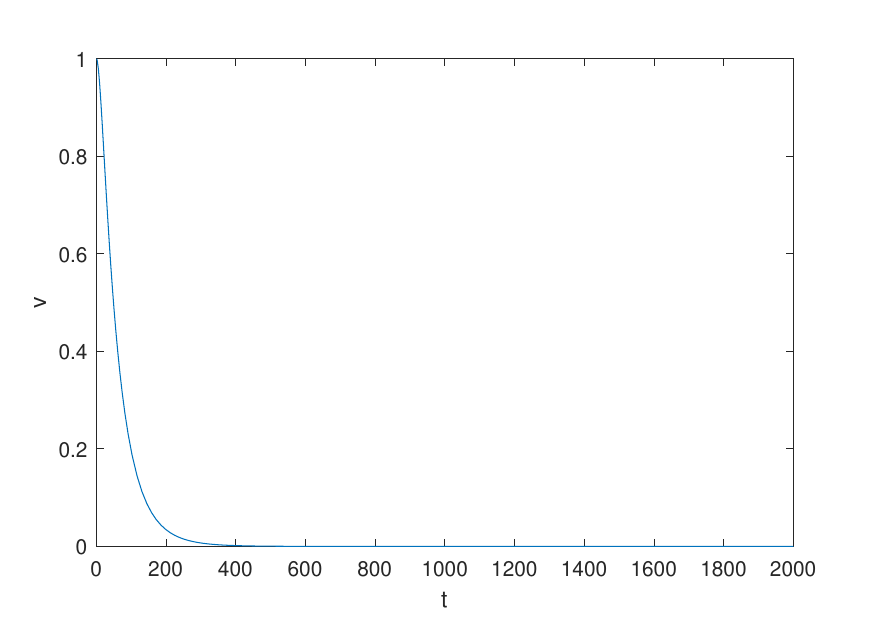}
		\caption{Temporal series of v(t)}
		\label{fig:image4}
	\end{subfigure}%
	\hfill
	\begin{subfigure}{0.3\textwidth}
		\centering
		\includegraphics[width=4.2cm,height=3.15cm]{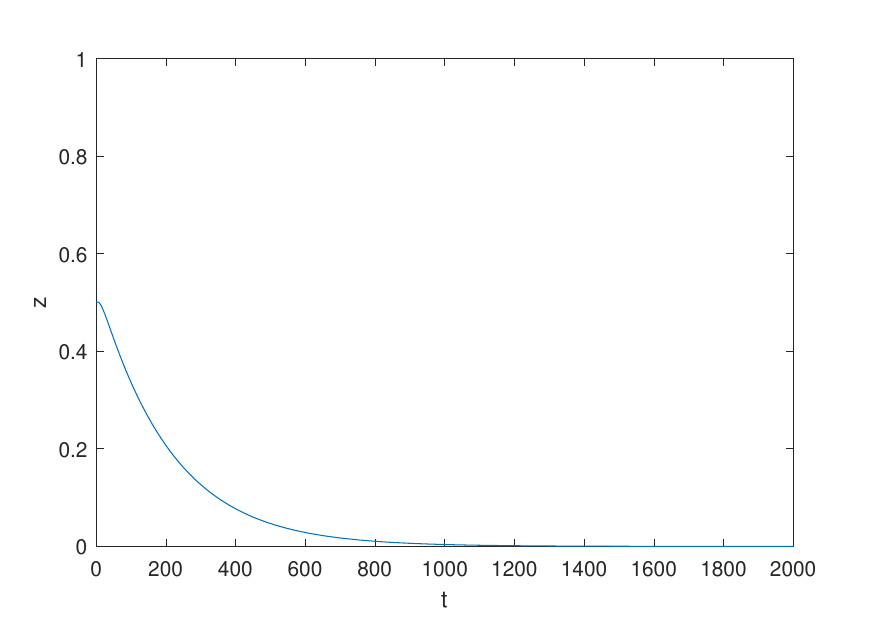}
		\caption{Temporal series of z(t)}
		\label{fig:image5}
	\end{subfigure}%
	\hfill
	\begin{subfigure}{0.3\textwidth}
		\centering
		\includegraphics[width=4.2cm,height=3.15cm]{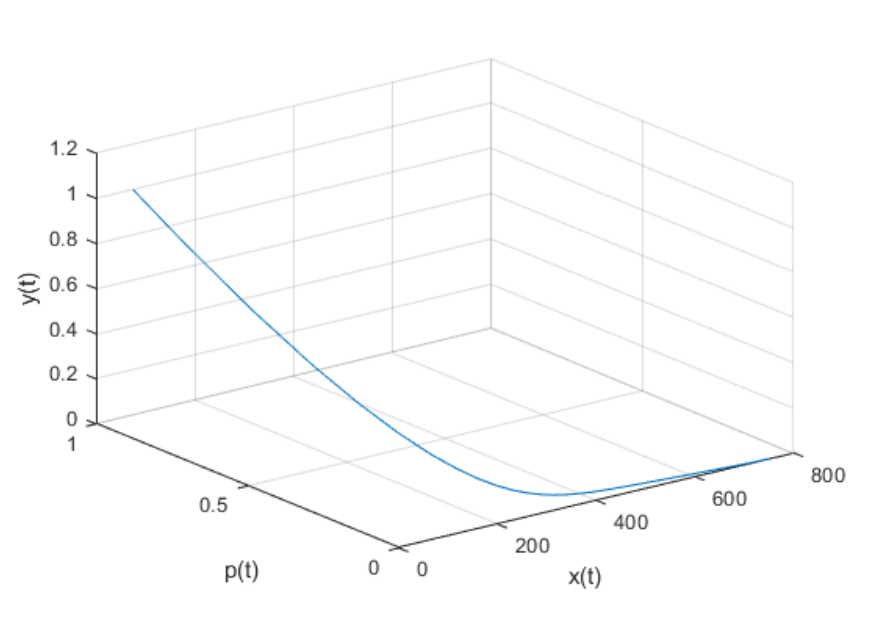}
		\caption{3D phase for x,p,y}
		\label{fig:image6}
	\end{subfigure}
	
	\caption{Simulation result of $E_2$ while $\tau_1=0.25$, $\tau_2=0.25$ and$\tau_3=0$}
	\label{fig}
\end{figure}
\begin{figure}[htbp] 
	\centering
	\begin{subfigure}{0.3\textwidth}
		\centering
		\includegraphics[width=4.2cm,height=3.15cm]{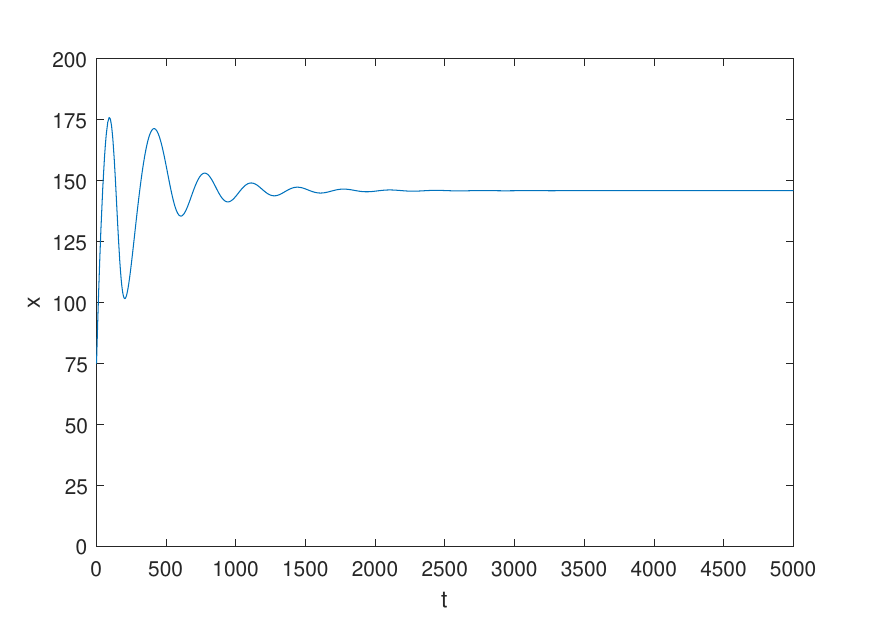}
		\caption{Temporal series of x(t)}
		\label{fig:image1}
	\end{subfigure}%
	\hfill
	\begin{subfigure}{0.3\textwidth}
		\centering
		\includegraphics[width=4.2cm,height=3.15cm]{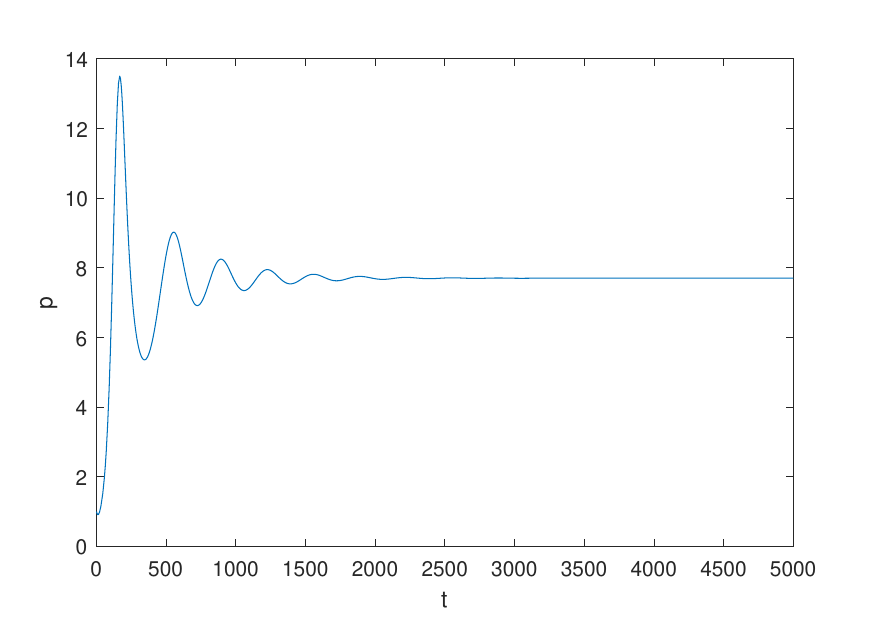}
		\caption{Temporal series of p(t)}
		\label{fig:image2}
	\end{subfigure}%
	\hfill
	\begin{subfigure}{0.3\textwidth}
		\centering
		\includegraphics[width=4.2cm,height=3.15cm]{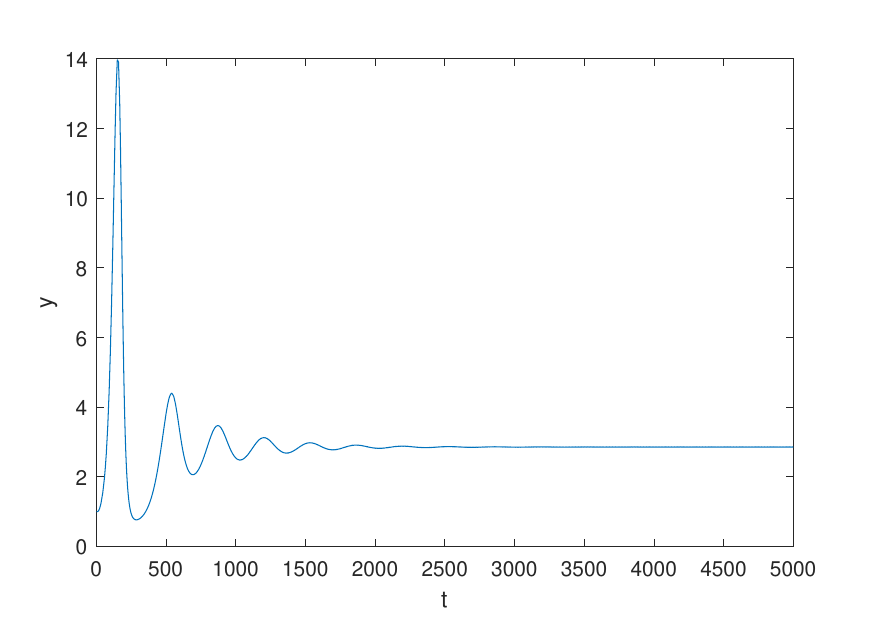}
		\caption{Temporal series of y(t)}
		\label{fig:image3}
	\end{subfigure}
	
	\begin{subfigure}{0.3\textwidth}
		\centering
		\includegraphics[width=4.2cm,height=3.15cm]{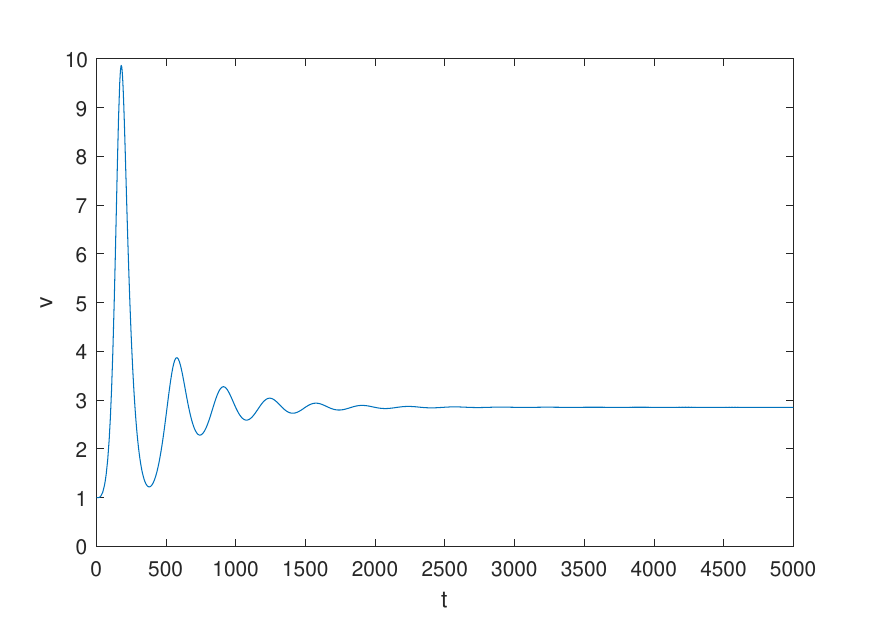}
		\caption{Temporal series of v(t)}
		\label{fig:image4}
	\end{subfigure}%
	\hfill
	\begin{subfigure}{0.3\textwidth}
		\centering
		\includegraphics[width=4.2cm,height=3.15cm]{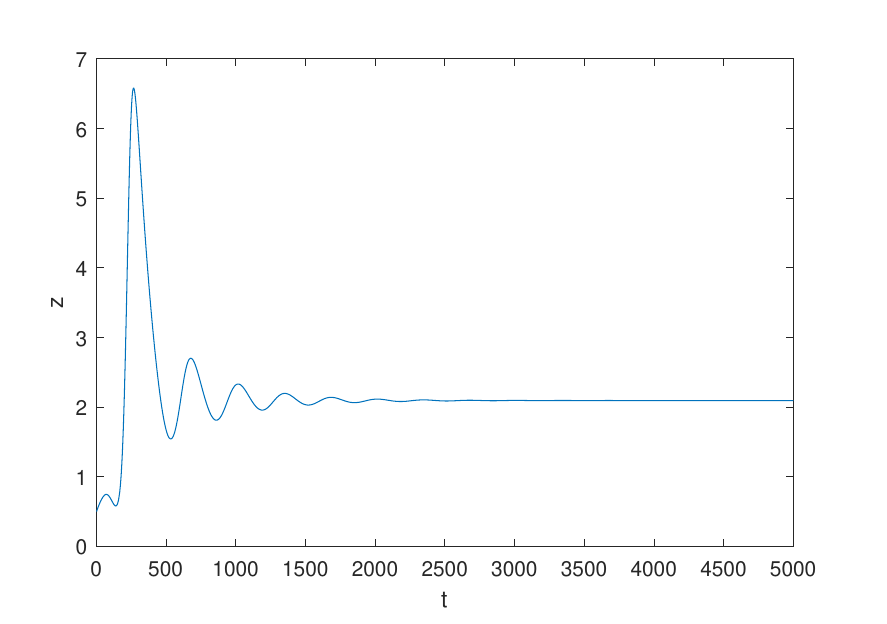}
		\caption{Temporal series of z(t)}
		\label{fig:image5}
	\end{subfigure}%
	\hfill
	\begin{subfigure}{0.3\textwidth}
		\centering
		\includegraphics[width=4.2cm,height=3.15cm]{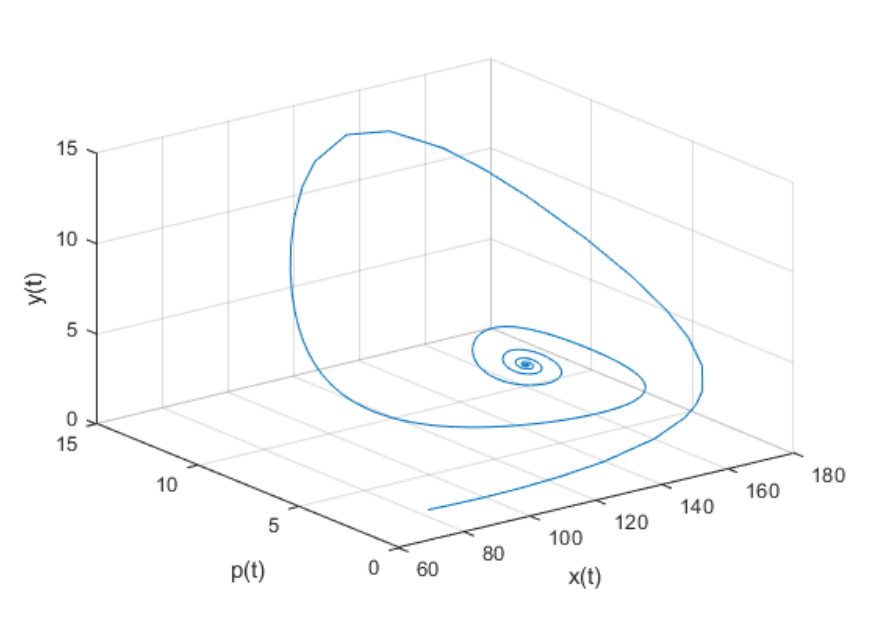}
		\caption{3D phase for x,p,y}
		\label{fig:image6}
	\end{subfigure}
	
	\caption{Simulation result of $E_2$ while $\tau_1=\tau_2=0$ and $\tau_3=70$}
	\label{fig}
\end{figure}

\begin{figure}[htbp] 
	\centering
	\begin{subfigure}{0.3\textwidth}
		\centering
		\includegraphics[width=4.2cm,height=3.15cm]{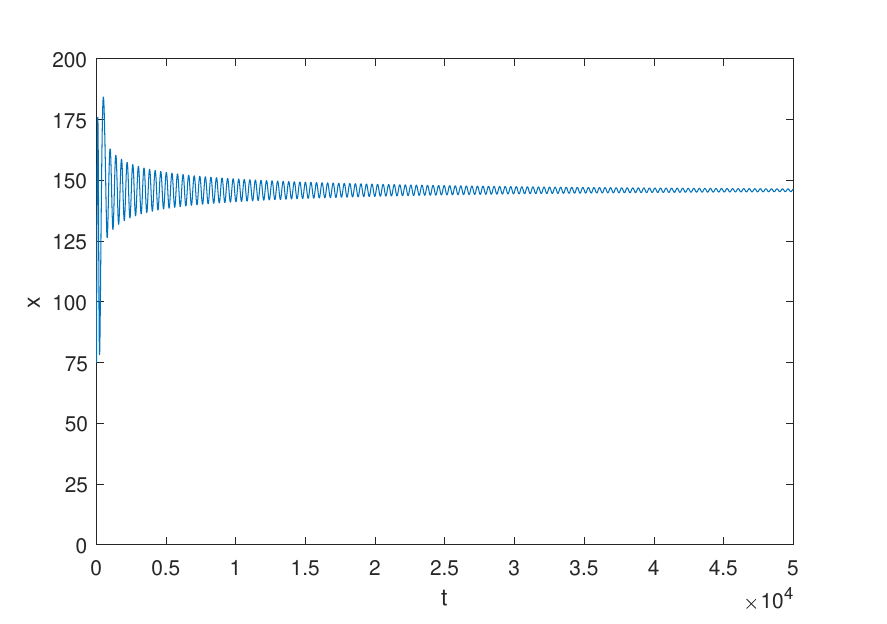}
		\caption{Temporal series of x(t)}
		\label{fig:image1}
	\end{subfigure}%
	\hfill
	\begin{subfigure}{0.3\textwidth}
		\centering
		\includegraphics[width=4.2cm,height=3.15cm]{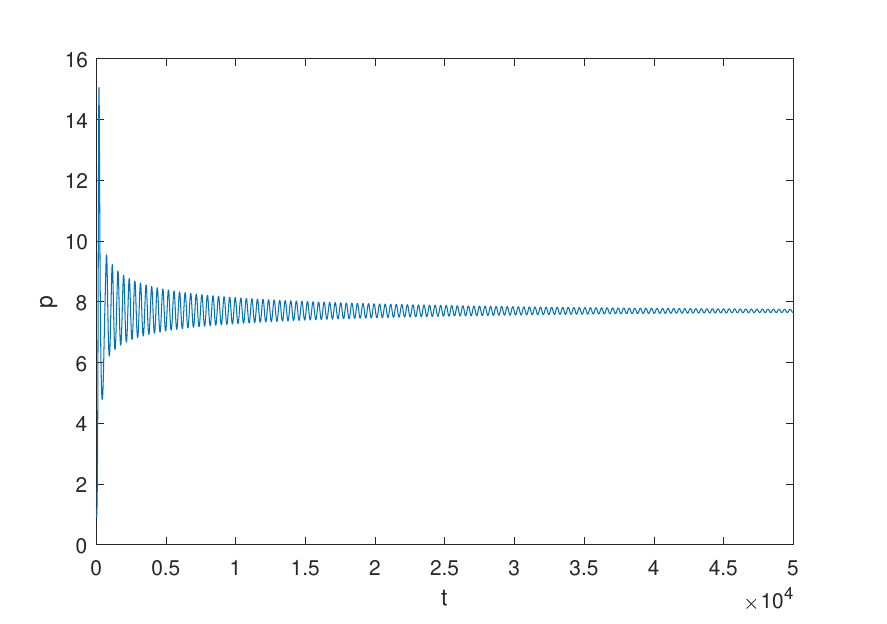}
		\caption{Temporal series of p(t)}
		\label{fig:image2}
	\end{subfigure}%
	\hfill
	\begin{subfigure}{0.3\textwidth}
		\centering
		\includegraphics[width=4.2cm,height=3.15cm]{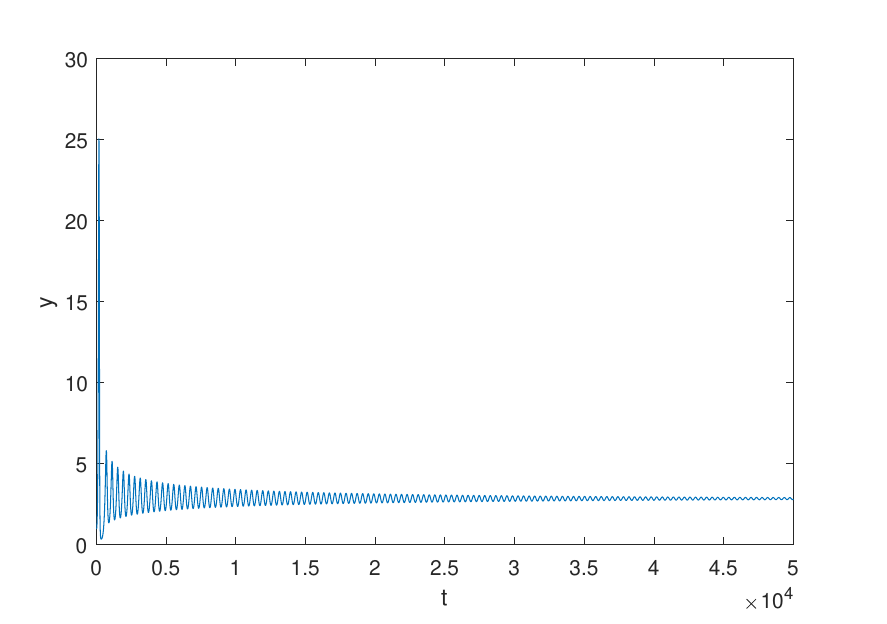}
		\caption{Temporal series of y(t)}
		\label{fig:image3}
	\end{subfigure}
	
	\begin{subfigure}{0.3\textwidth}
		\centering
		\includegraphics[width=4.2cm,height=3.15cm]{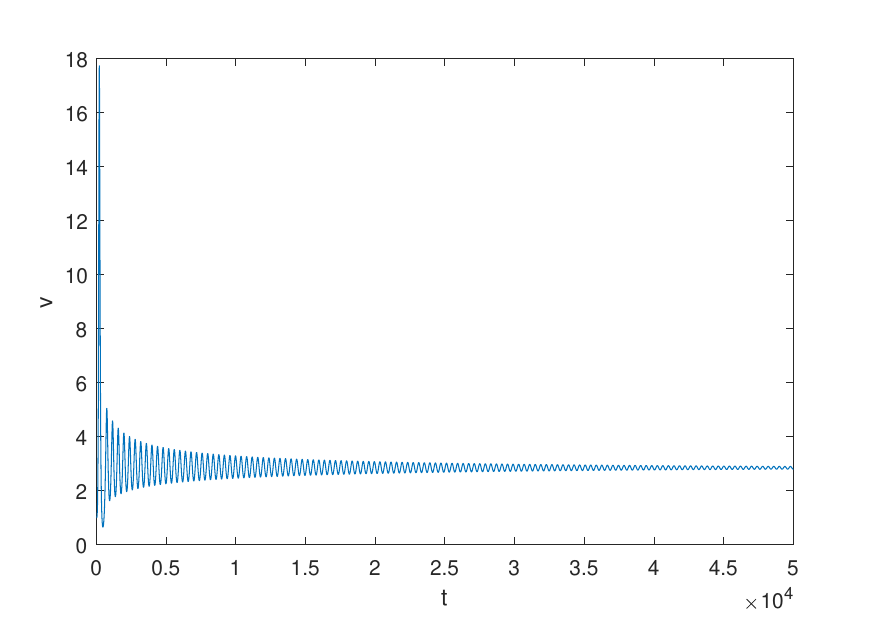}
		\caption{Temporal series of v(t)}
		\label{fig:image4}
	\end{subfigure}%
	\hfill
	\begin{subfigure}{0.3\textwidth}
		\centering
		\includegraphics[width=4.2cm,height=3.15cm]{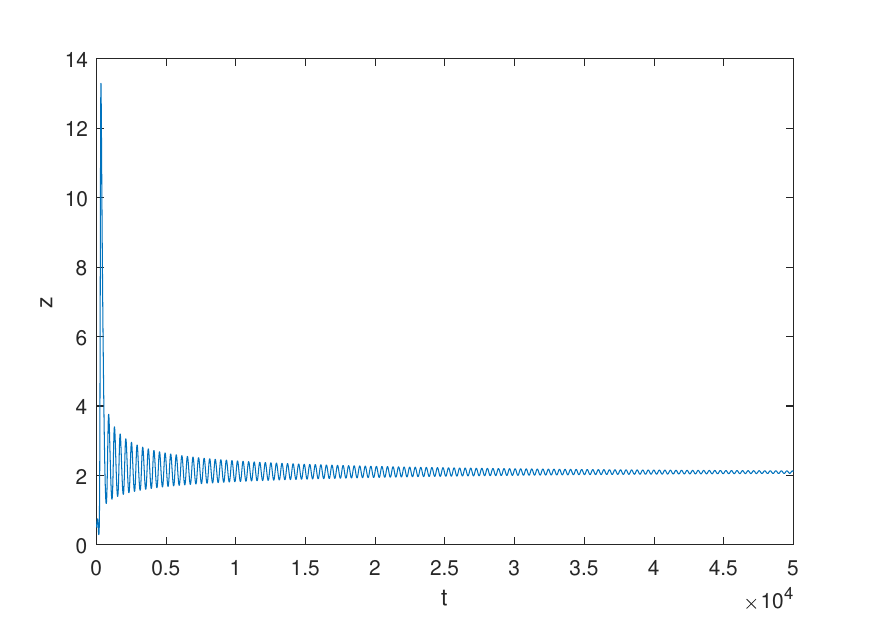}
		\caption{Temporal series of z(t)}
		\label{fig:image5}
	\end{subfigure}%
	\hfill
	\begin{subfigure}{0.3\textwidth}
		\centering
		\includegraphics[width=4.2cm,height=3.15cm]{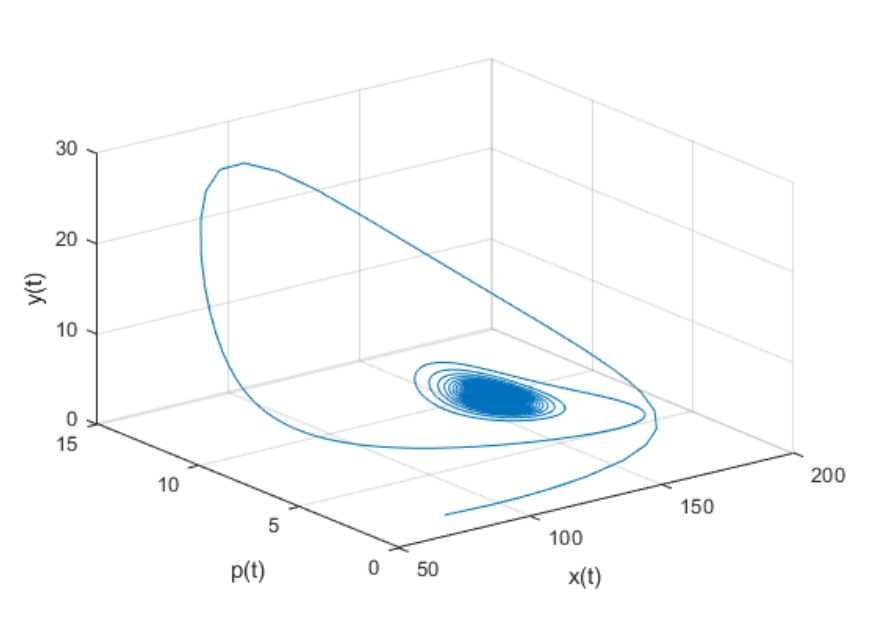}
		\caption{3D phase for x,p,y}
		\label{fig:image6}
	\end{subfigure}
	
	\caption{Simulation result of $E_2$ while $\tau_1=\tau_2=0$ and $\tau_3=100$}
	\label{fig}
\end{figure}

\begin{figure}[htbp] 
	\centering
	\begin{subfigure}{0.3\textwidth}
		\centering
		\includegraphics[width=4.2cm,height=3.15cm]{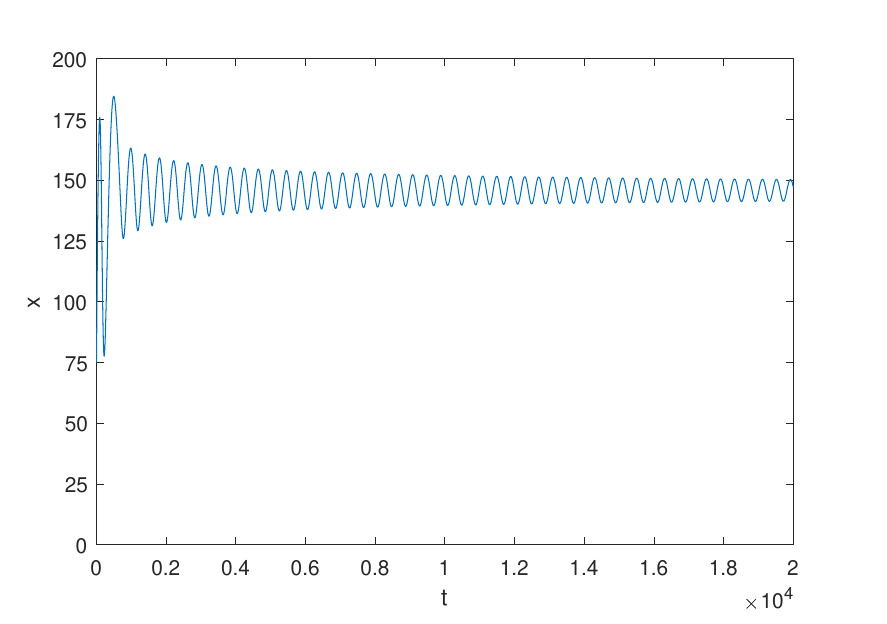}
		\caption{Temporal series of x(t)}
		\label{fig:image1}
	\end{subfigure}%
	\hfill
	\begin{subfigure}{0.3\textwidth}
		\centering
		\includegraphics[width=4.2cm,height=3.15cm]{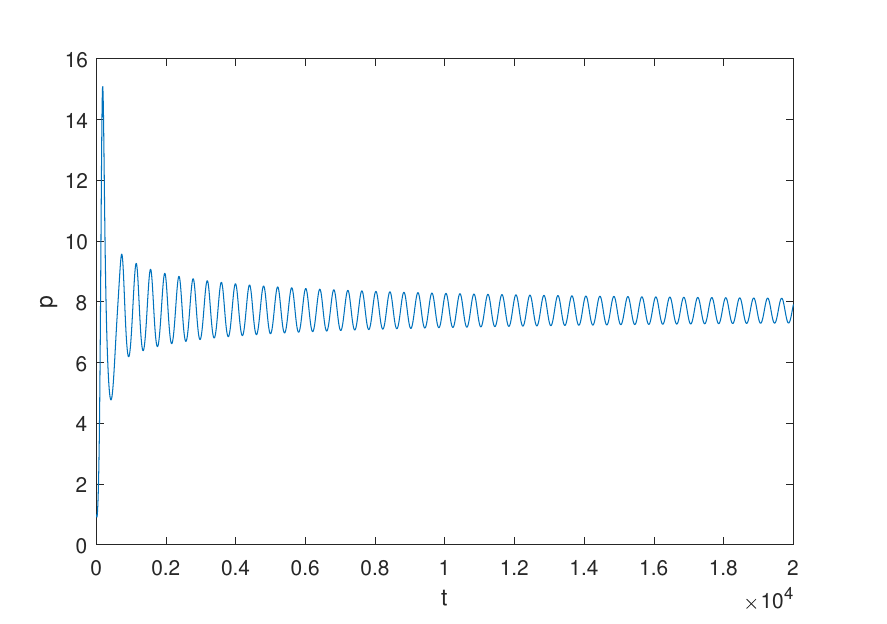}
		\caption{Temporal series of p(t)}
		\label{fig:image2}
	\end{subfigure}%
	\hfill
	\begin{subfigure}{0.3\textwidth}
		\centering
		\includegraphics[width=4.2cm,height=3.15cm]{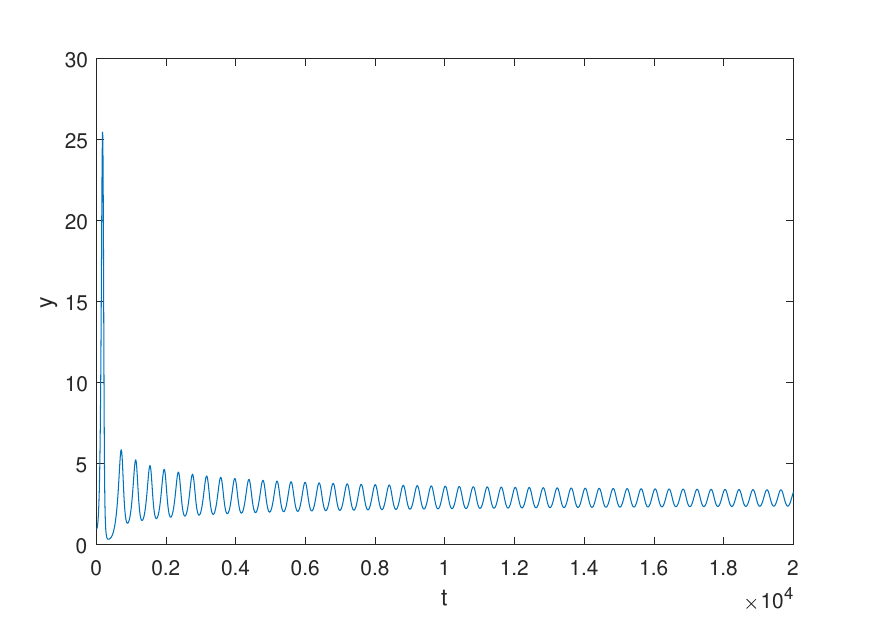}
		\caption{Temporal series of y(t)}
		\label{fig:image3}
	\end{subfigure}
	
	\begin{subfigure}{0.3\textwidth}
		\centering
		\includegraphics[width=4.2cm,height=3.15cm]{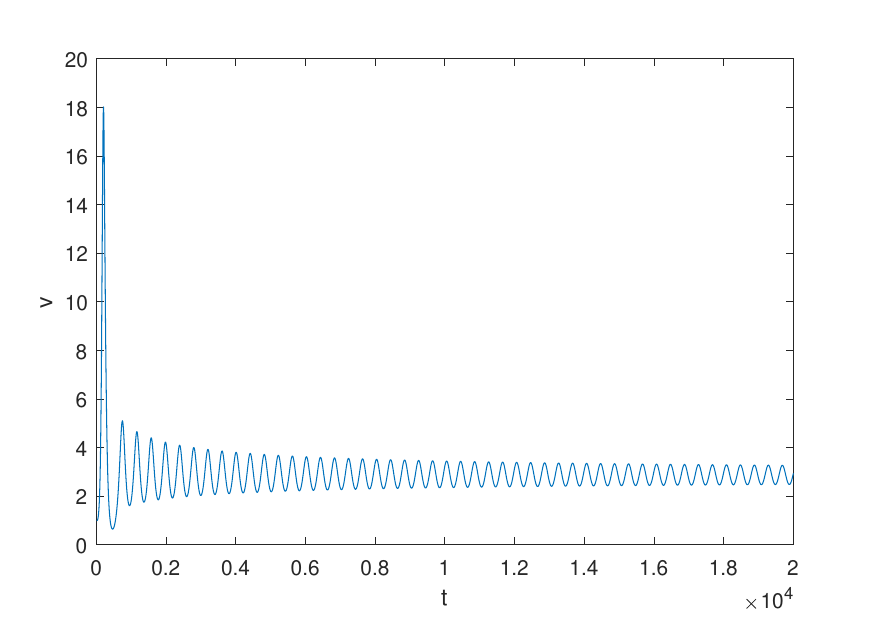}
		\caption{Temporal series of v(t)}
		\label{fig:image4}
	\end{subfigure}%
	\hfill
	\begin{subfigure}{0.3\textwidth}
		\centering
		\includegraphics[width=4.2cm,height=3.15cm]{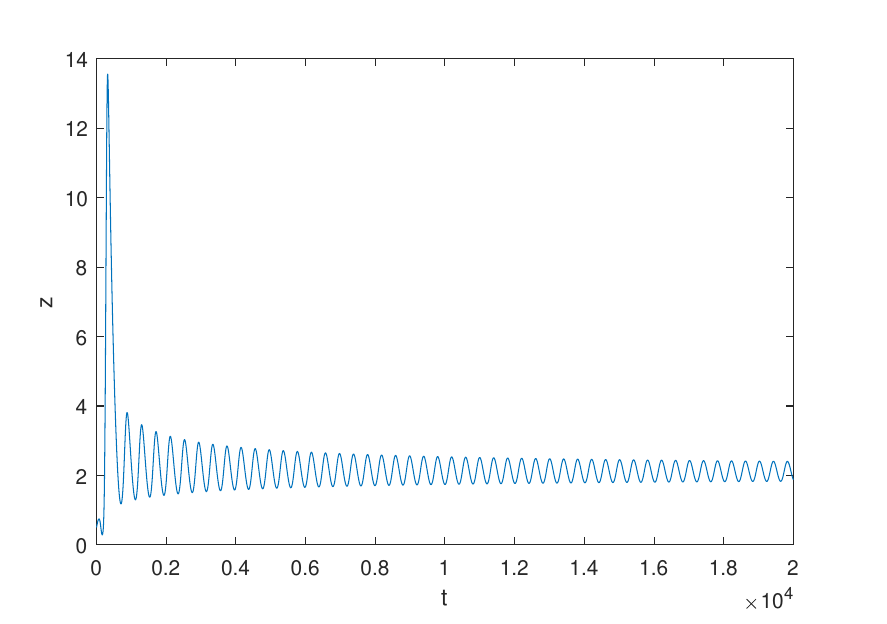}
		\caption{Temporal series of z(t)}
		\label{fig:image5}
	\end{subfigure}%
	\hfill
	\begin{subfigure}{0.3\textwidth}
		\centering
		\includegraphics[width=4.2cm,height=3.15cm]{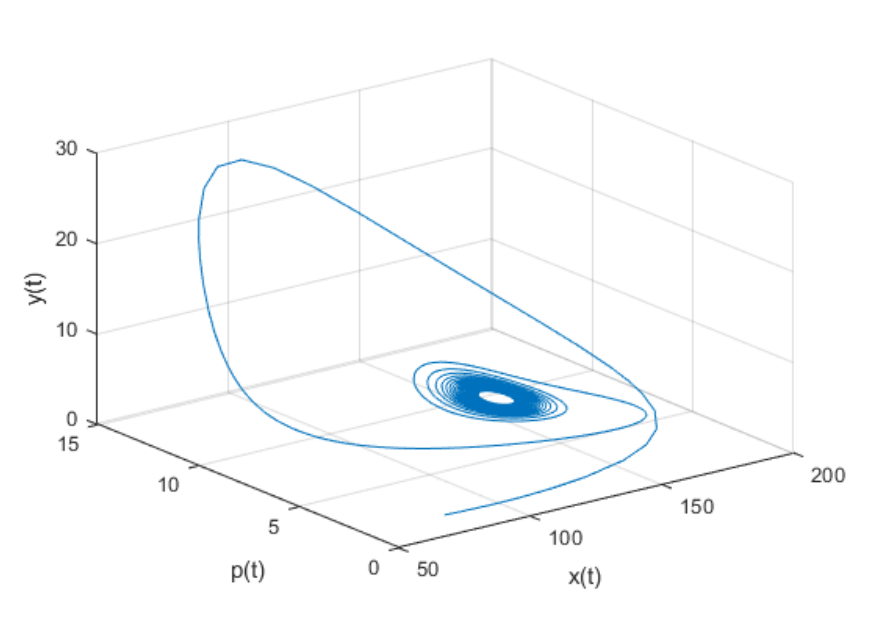}
		\caption{3D phase for x,p,y}
		\label{fig:image6}
	\end{subfigure}
	
	\caption{Simulation result of $E_2$ while $\tau_1=\tau_2=0$ and $\tau_3=101$}
	\label{fig}
\end{figure}
\begin{figure}[htbp] 
	\centering
	\begin{subfigure}{0.3\textwidth}
		\centering
		\includegraphics[width=4.2cm,height=3.15cm]{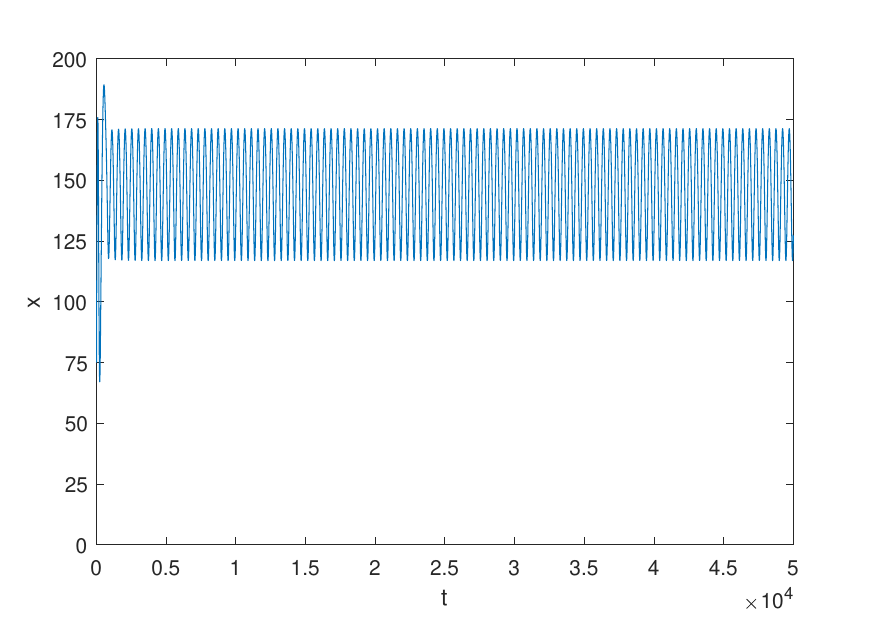}
		\caption{Temporal series of x(t)}
		\label{fig:image1}
	\end{subfigure}%
	\hfill
	\begin{subfigure}{0.3\textwidth}
		\centering
		\includegraphics[width=4.2cm,height=3.15cm]{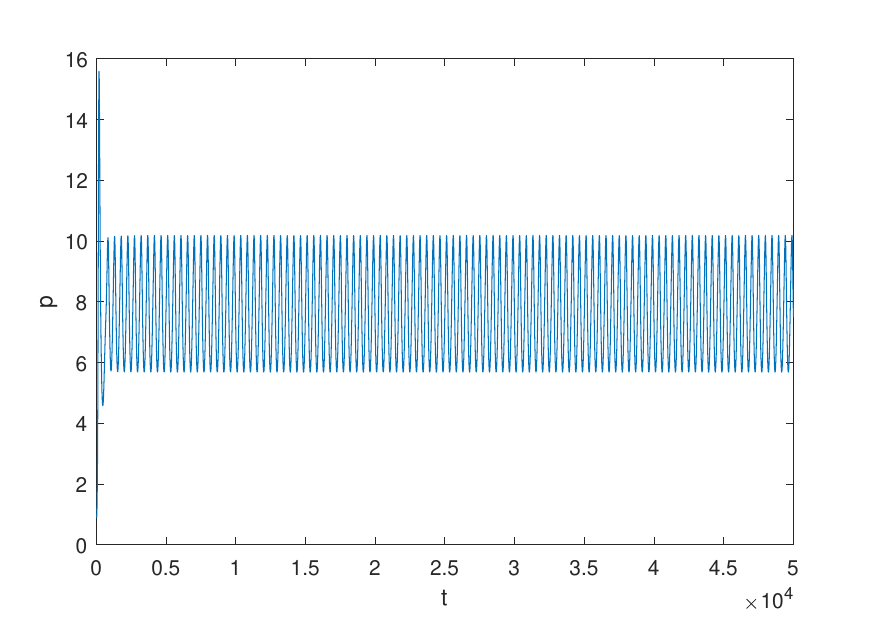}
		\caption{Temporal series of p(t)}
		\label{fig:image2}
	\end{subfigure}%
	\hfill
	\begin{subfigure}{0.3\textwidth}
		\centering
		\includegraphics[width=4.2cm,height=3.15cm]{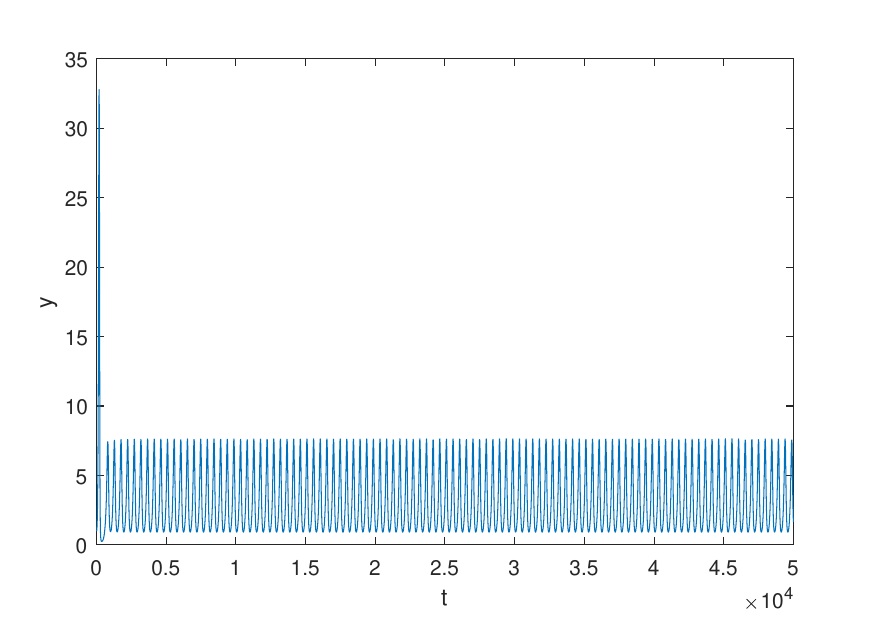}
		\caption{Temporal series of y(t)}
		\label{fig:image3}
	\end{subfigure}
	
	\begin{subfigure}{0.3\textwidth}
		\centering
		\includegraphics[width=4.2cm,height=3.15cm]{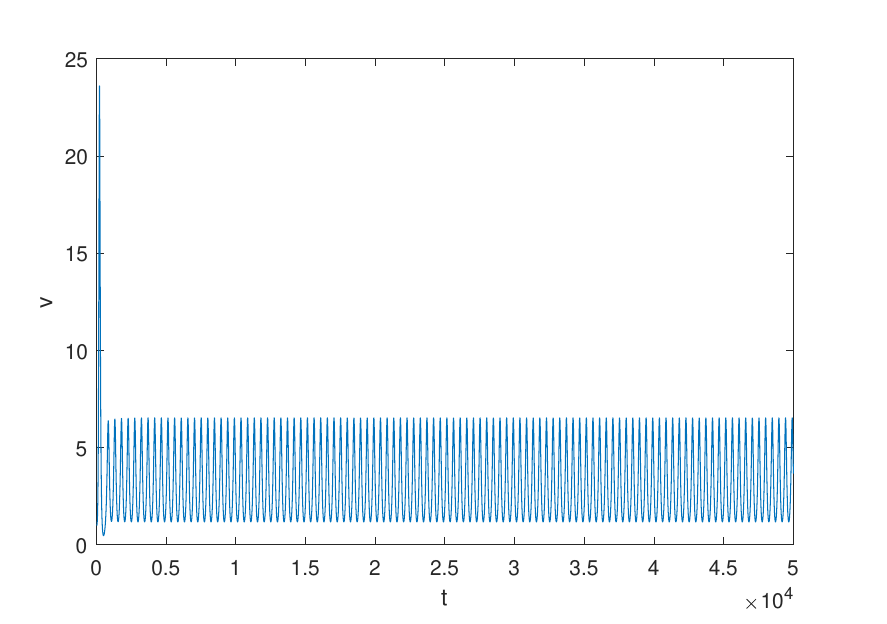}
		\caption{Temporal series of v(t)}
		\label{fig:image4}
	\end{subfigure}%
	\hfill
	\begin{subfigure}{0.3\textwidth}
		\centering
		\includegraphics[width=4.2cm,height=3.15cm]{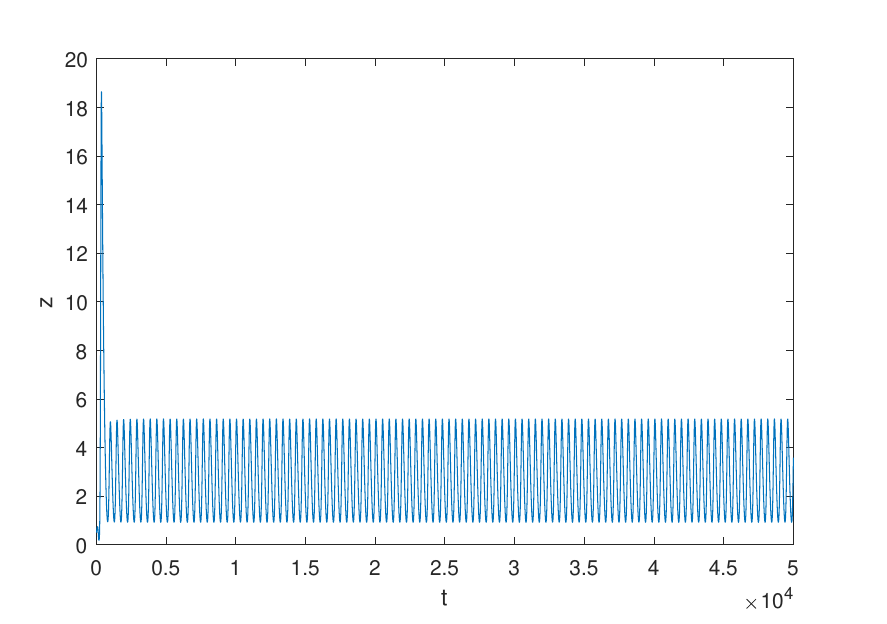}
		\caption{Temporal series of z(t)}
		\label{fig:image5}
	\end{subfigure}%
	\hfill
	\begin{subfigure}{0.3\textwidth}
		\centering
		\includegraphics[width=4.2cm,height=3.15cm]{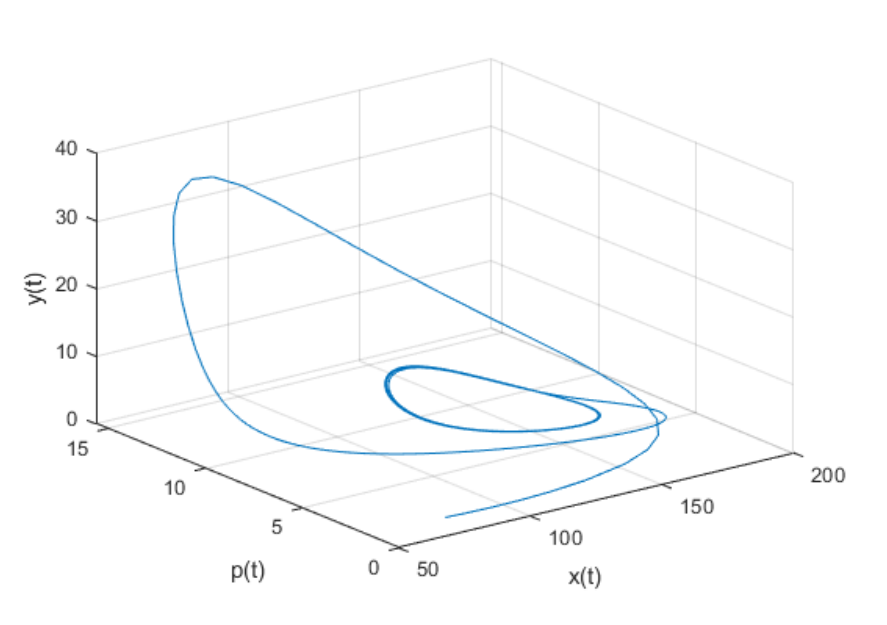}
		\caption{3D phase for x,p,y}
		\label{fig:image6}
	\end{subfigure}
	
	\caption{Simulation result of $E_2$ while $\tau_1=\tau_2=0$ and $\tau_3=120$}
	\label{fig}
\end{figure}
\begin{figure}[htbp] 
	\centering
	\begin{subfigure}{0.3\textwidth}
		\centering
		\includegraphics[width=4.2cm,height=3.15cm]{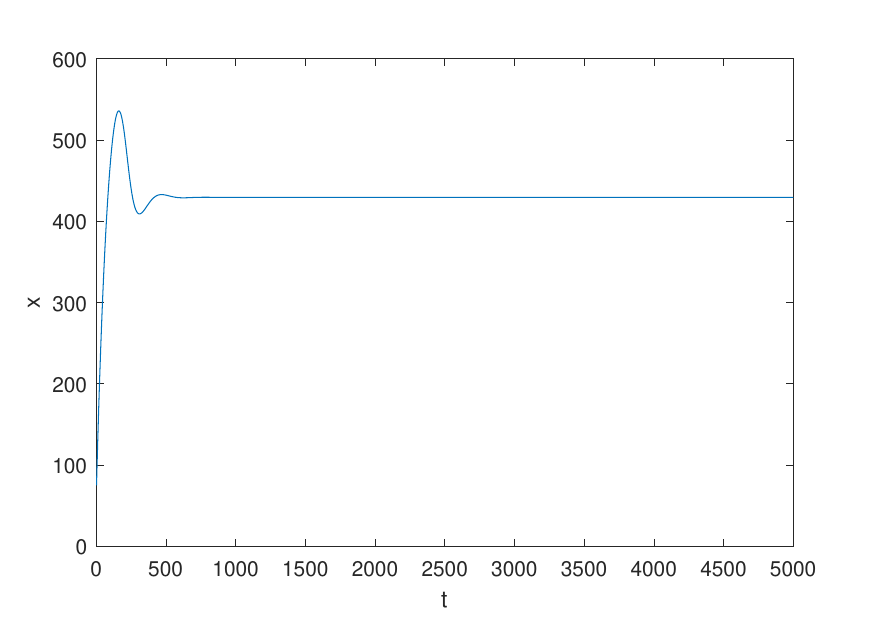}
		\caption{Temporal series of x(t)}
		\label{fig:image1}
	\end{subfigure}%
	\hfill
	\begin{subfigure}{0.3\textwidth}
		\centering
		\includegraphics[width=4.2cm,height=3.15cm]{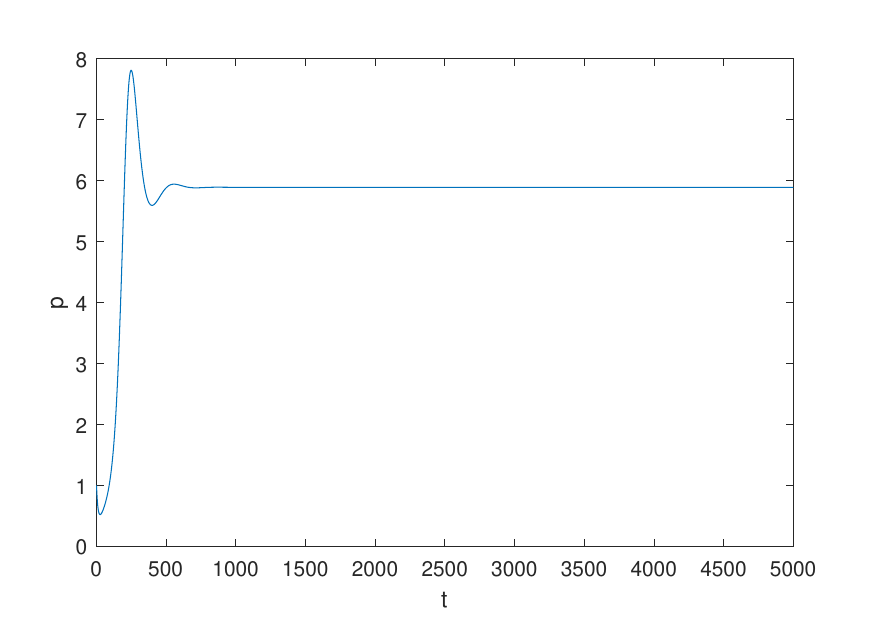}
		\caption{Temporal series of p(t)}
		\label{fig:image2}
	\end{subfigure}%
	\hfill
	\begin{subfigure}{0.3\textwidth}
		\centering
		\includegraphics[width=4.2cm,height=3.15cm]{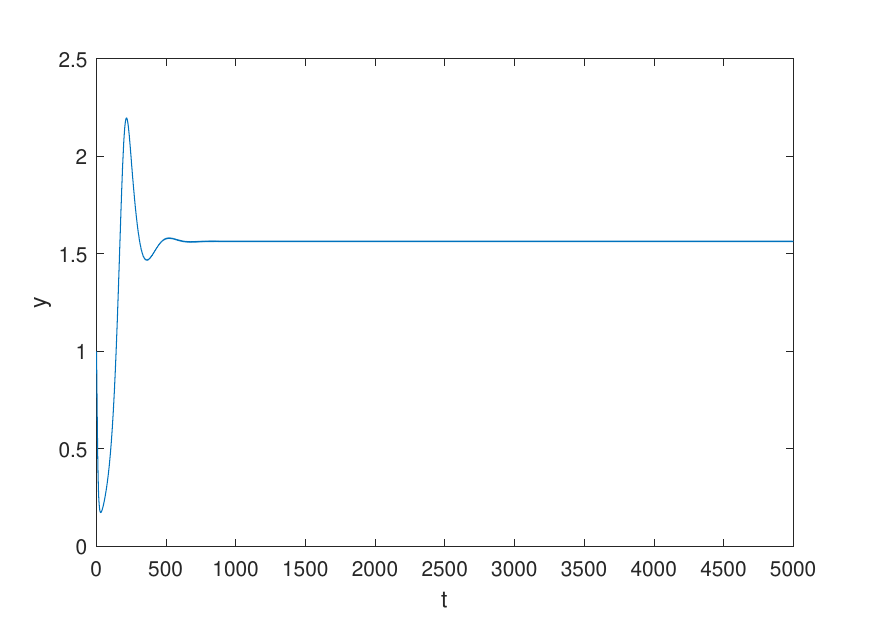}
		\caption{Temporal series of y(t)}
		\label{fig:image3}
	\end{subfigure}
	
	\begin{subfigure}{0.3\textwidth}
		\centering
		\includegraphics[width=4.2cm,height=3.15cm]{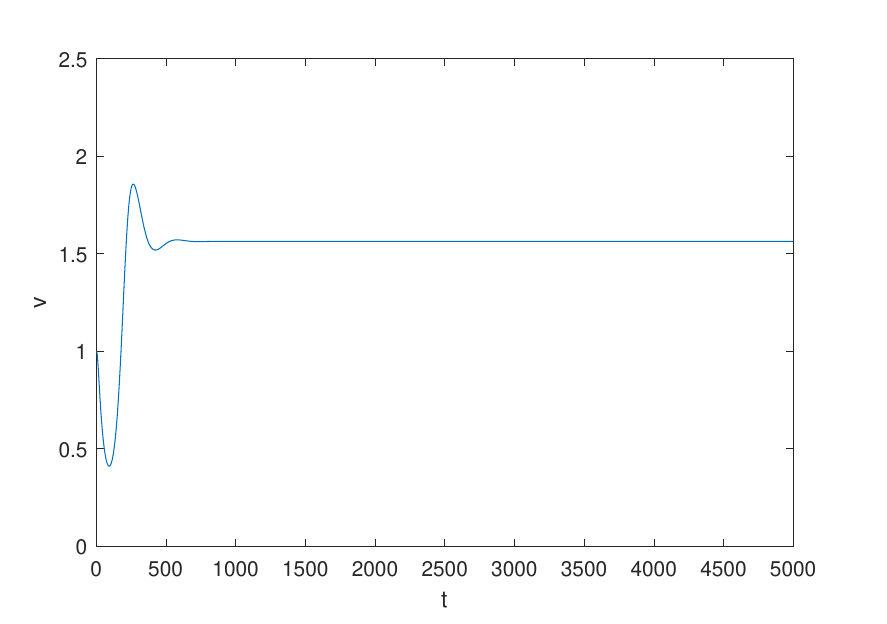}
		\caption{Temporal series of v(t)}
		\label{fig:image4}
	\end{subfigure}%
	\hfill
	\begin{subfigure}{0.3\textwidth}
		\centering
		\includegraphics[width=4.2cm,height=3.15cm]{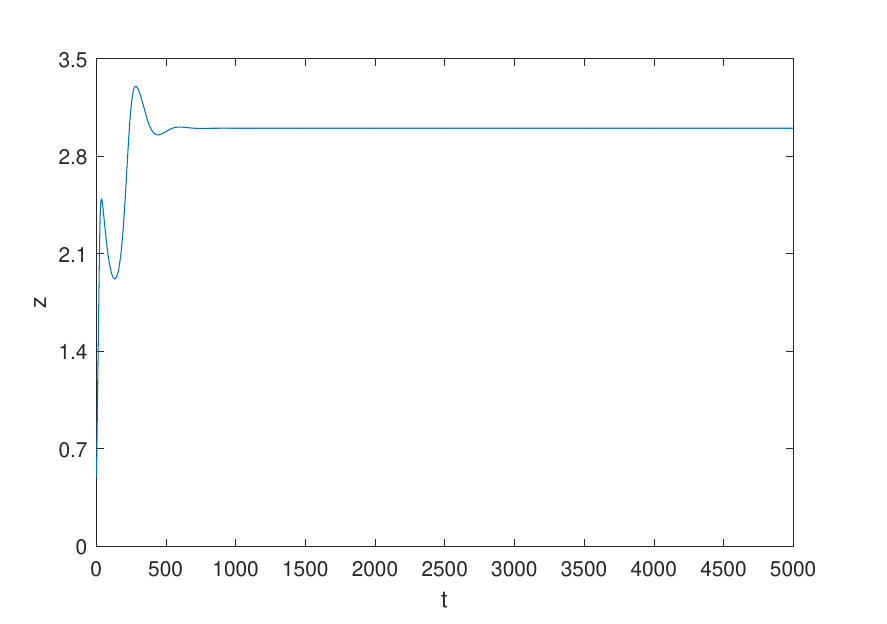}
		\caption{Temporal series of z(t)}
		\label{fig:image5}
	\end{subfigure}%
	\hfill
	\begin{subfigure}{0.3\textwidth}
		\centering
		\includegraphics[width=4.2cm,height=3.15cm]{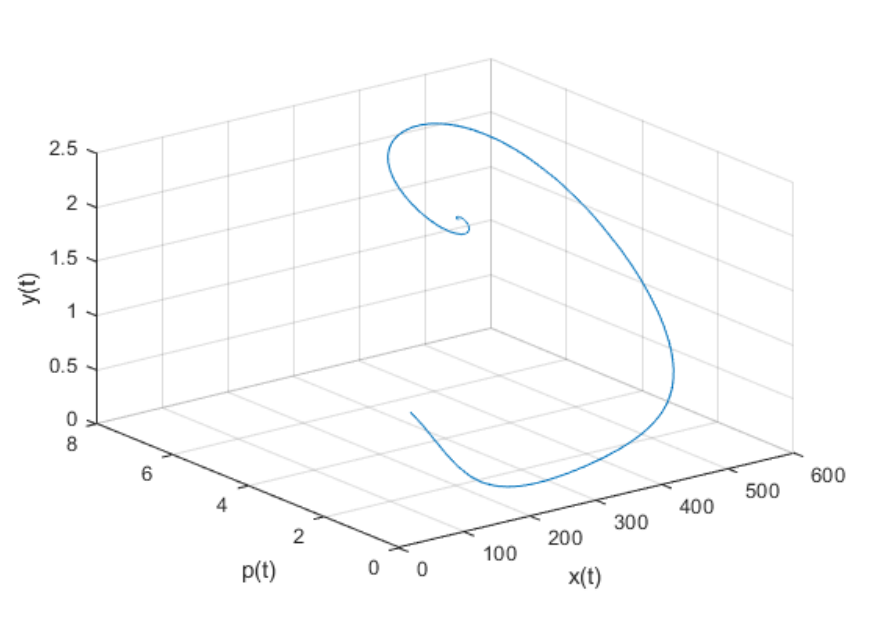}
		\caption{3D phase for x,p,y}
		\label{fig:image6}
	\end{subfigure}
	
	\caption{Simulation result of $E_2$ while $\tau_1=\tau_2=0.25$ and$\tau_3=20$}
	\label{fig}
\end{figure}
\begin{figure}[htbp] 
	\centering
	\begin{subfigure}{0.3\textwidth}
		\centering
		\includegraphics[width=4.2cm,height=3.15cm]{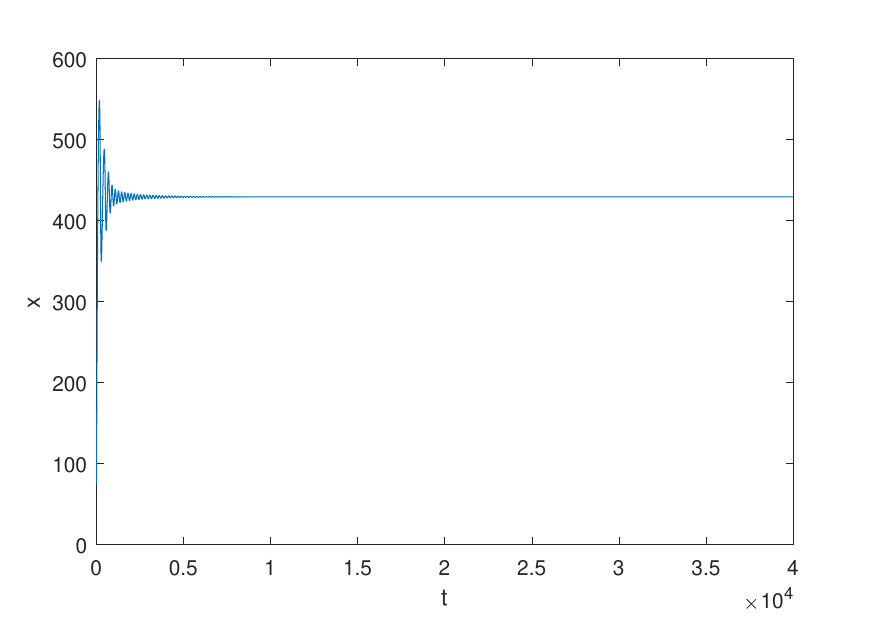}
		\caption{Temporal series of x(t)}
		\label{fig:image1}
	\end{subfigure}%
	\hfill
	\begin{subfigure}{0.3\textwidth}
		\centering
		\includegraphics[width=4.2cm,height=3.15cm]{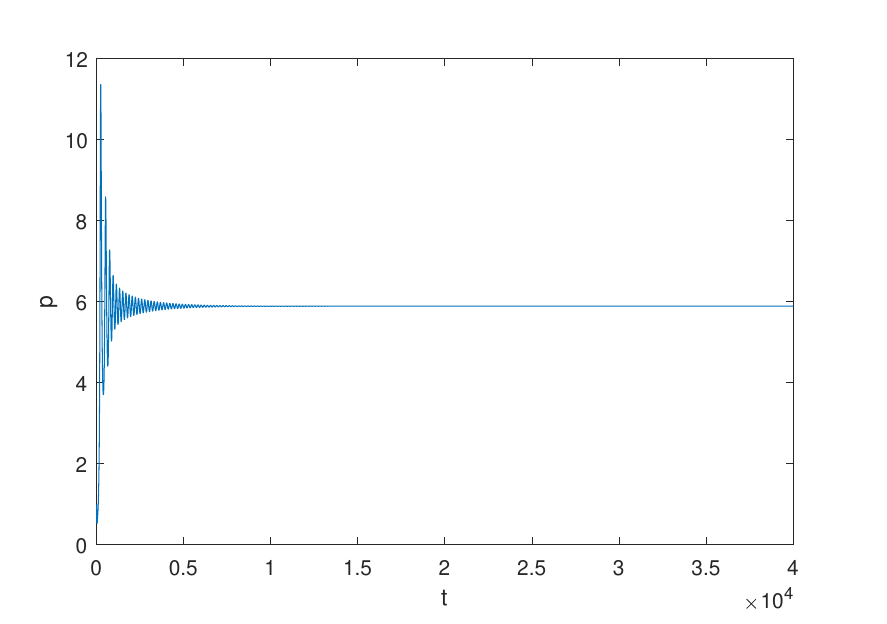}
		\caption{Temporal series of p(t)}
		\label{fig:image2}
	\end{subfigure}%
	\hfill
	\begin{subfigure}{0.3\textwidth}
		\centering
		\includegraphics[width=4.2cm,height=3.15cm]{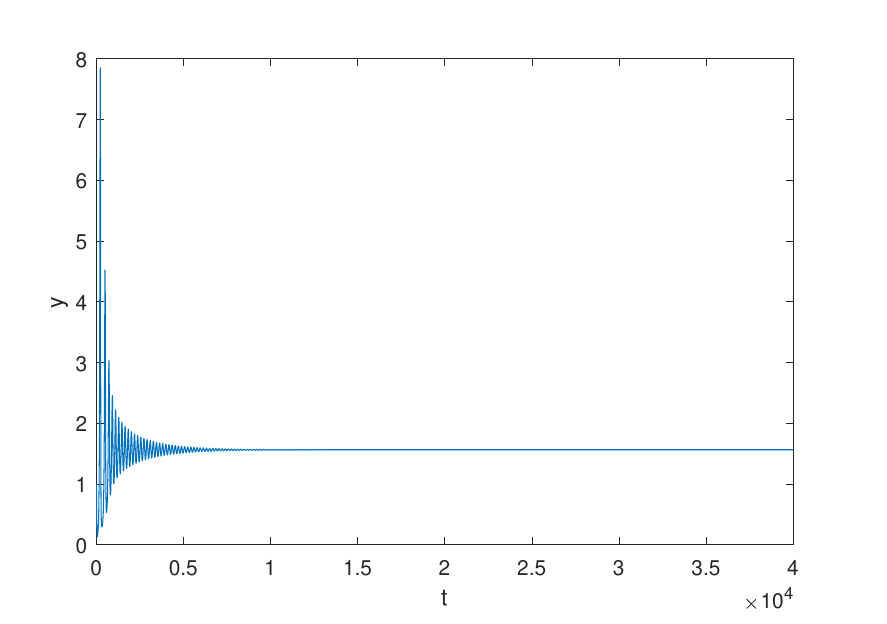}
		\caption{Temporal series of y(t)}
		\label{fig:image3}
	\end{subfigure}
	
	\begin{subfigure}{0.3\textwidth}
		\centering
		\includegraphics[width=4.2cm,height=3.15cm]{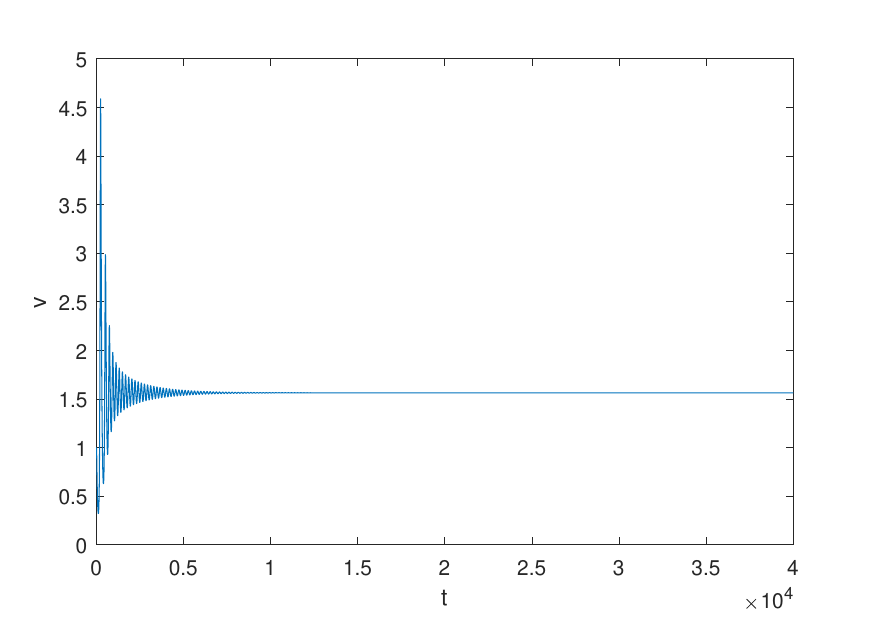}
		\caption{Temporal series of v(t)}
		\label{fig:image4}
	\end{subfigure}%
	\hfill
	\begin{subfigure}{0.3\textwidth}
		\centering
		\includegraphics[width=4.2cm,height=3.15cm]{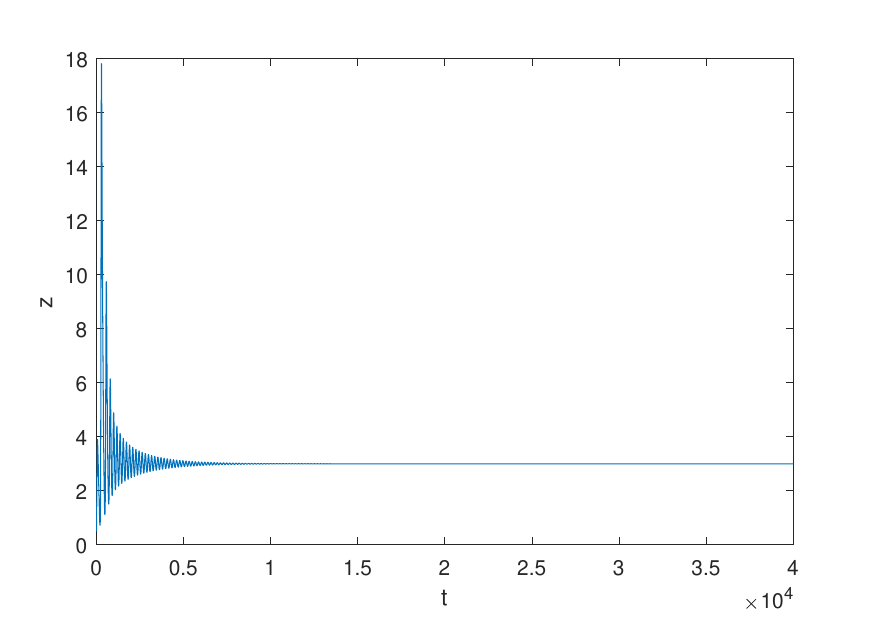}
		\caption{Temporal series of z(t)}
		\label{fig:image5}
	\end{subfigure}%
	\hfill
	\begin{subfigure}{0.3\textwidth}
		\centering
		\includegraphics[width=4.2cm,height=3.15cm]{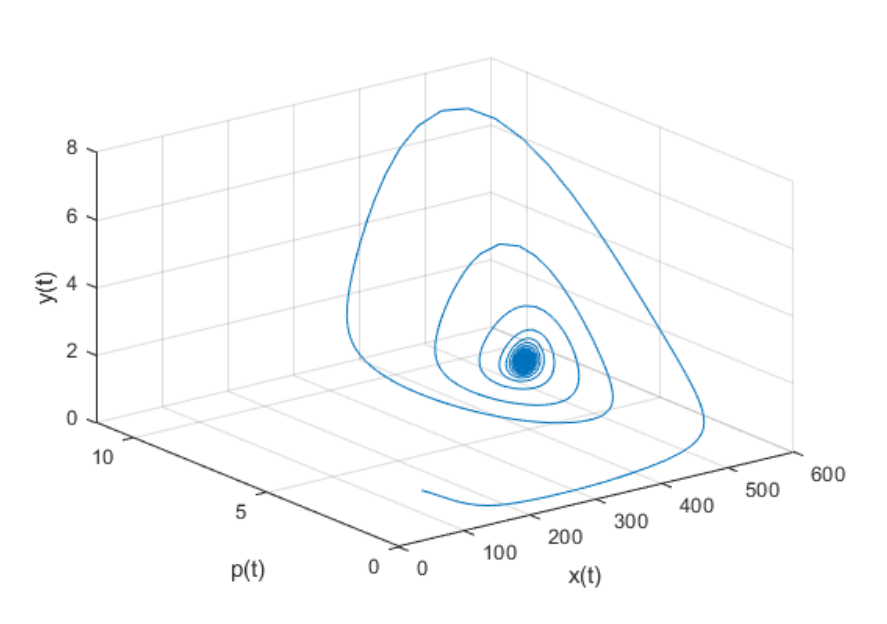}
		\caption{3D phase for x,p,y}
		\label{fig:image6}
	\end{subfigure}
	
	\caption{Simulation result of $E_2$ while $\tau_1=\tau_2=0.25$ and $\tau_3=42$}
	\label{fig}
\end{figure}
\begin{figure}[htbp] 
	\centering
	\begin{subfigure}{0.3\textwidth}
		\centering
		\includegraphics[width=4.2cm,height=3.15cm]{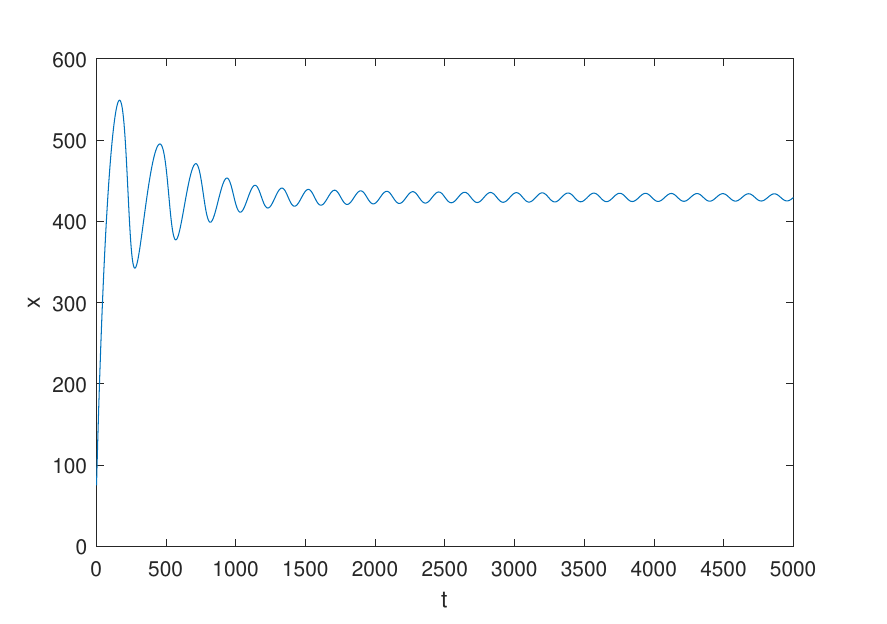}
		\caption{Temporal series of x(t)}
		\label{fig:image1}
	\end{subfigure}%
	\hfill
	\begin{subfigure}{0.3\textwidth}
		\centering
		\includegraphics[width=4.2cm,height=3.15cm]{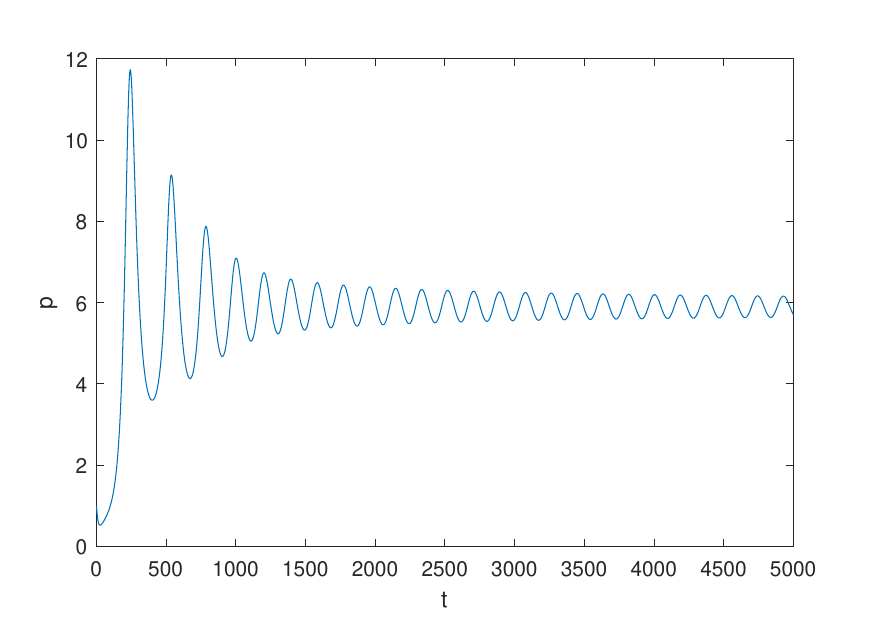}
		\caption{Temporal series of p(t)}
		\label{fig:image2}
	\end{subfigure}%
	\hfill
	\begin{subfigure}{0.3\textwidth}
		\centering
		\includegraphics[width=4.2cm,height=3.15cm]{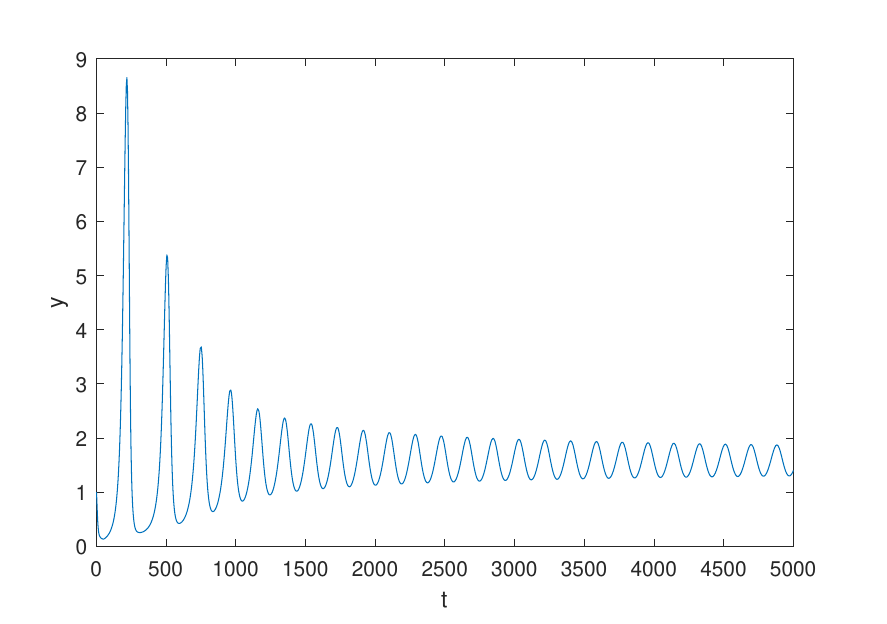}
		\caption{Temporal series of y(t)}
		\label{fig:image3}
	\end{subfigure}
	
	\begin{subfigure}{0.3\textwidth}
		\centering
		\includegraphics[width=4.2cm,height=3.15cm]{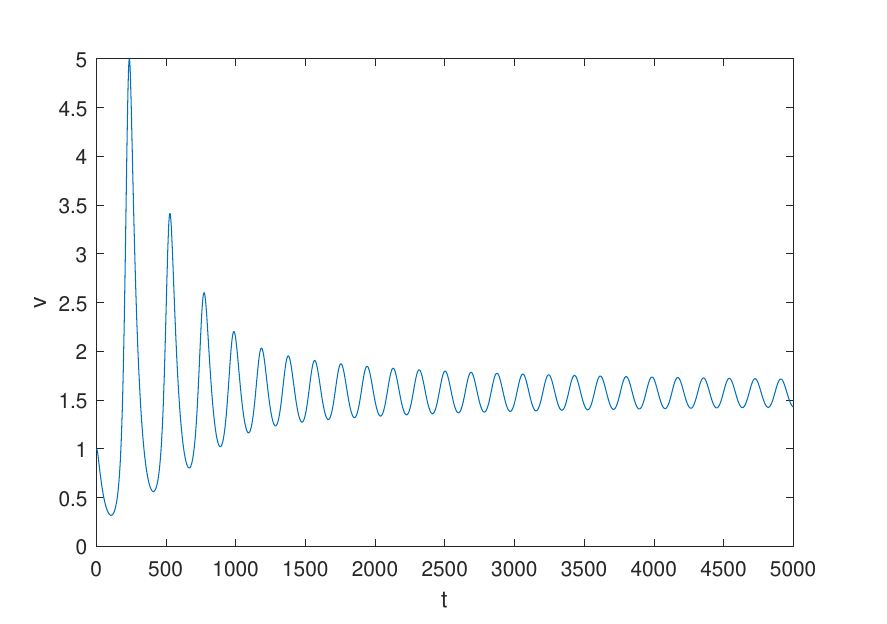}
		\caption{Temporal series of v(t)}
		\label{fig:image4}
	\end{subfigure}%
	\hfill
	\begin{subfigure}{0.3\textwidth}
		\centering
		\includegraphics[width=4.2cm,height=3.15cm]{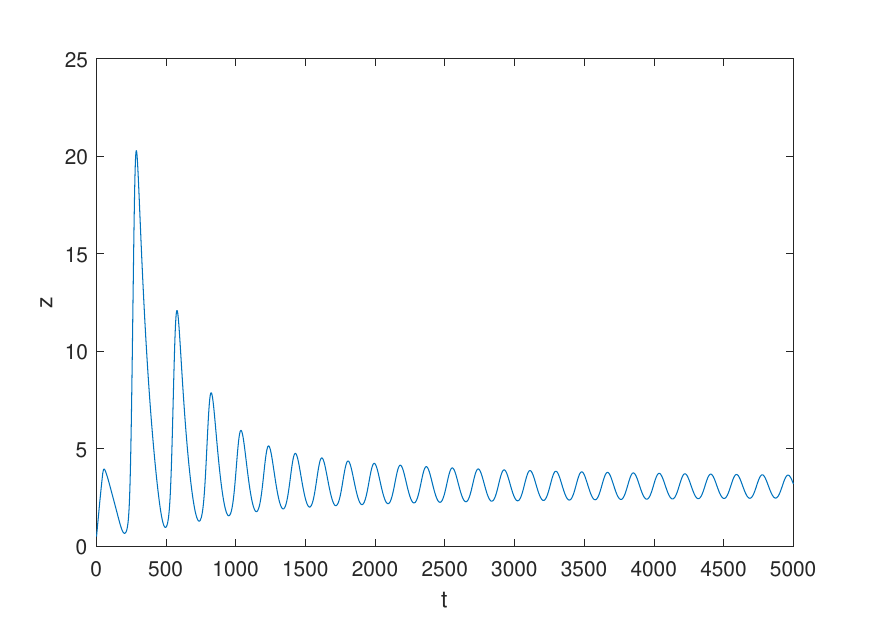}
		\caption{Temporal series of z(t)}
		\label{fig:image5}
	\end{subfigure}%
	\hfill
	\begin{subfigure}{0.3\textwidth}
		\centering
		\includegraphics[width=4.2cm,height=3.15cm]{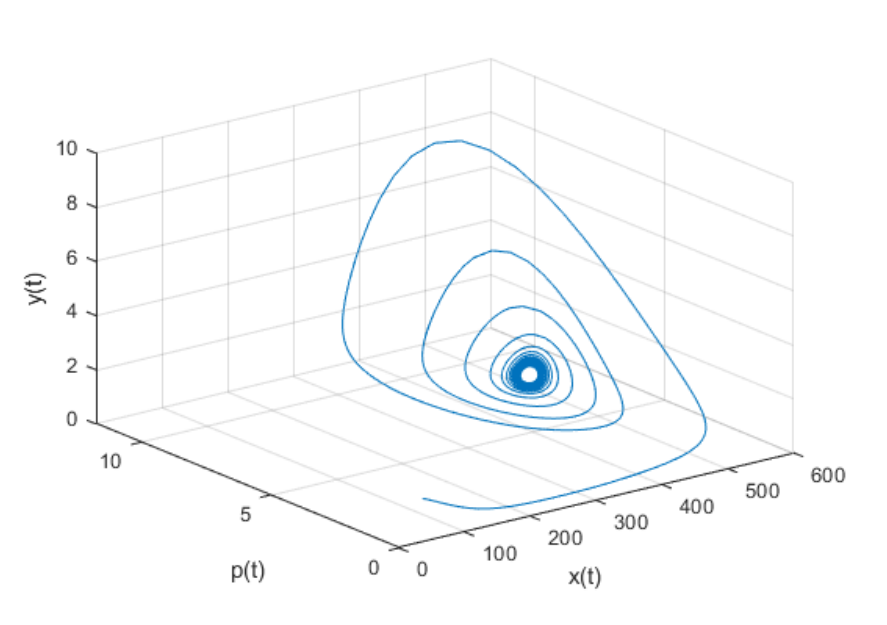}
		\caption{3D phase for x,p,y}
		\label{fig:image6}
	\end{subfigure}
	
	\caption{Simulation result of $E_2$ while $\tau_1=\tau_2=0.25$ and $\tau_3=43$}
	\label{fig}
\end{figure}
\begin{figure}[htbp] 
	\centering
	\begin{subfigure}{0.3\textwidth}
		\centering
		\includegraphics[width=4.2cm,height=3.15cm]{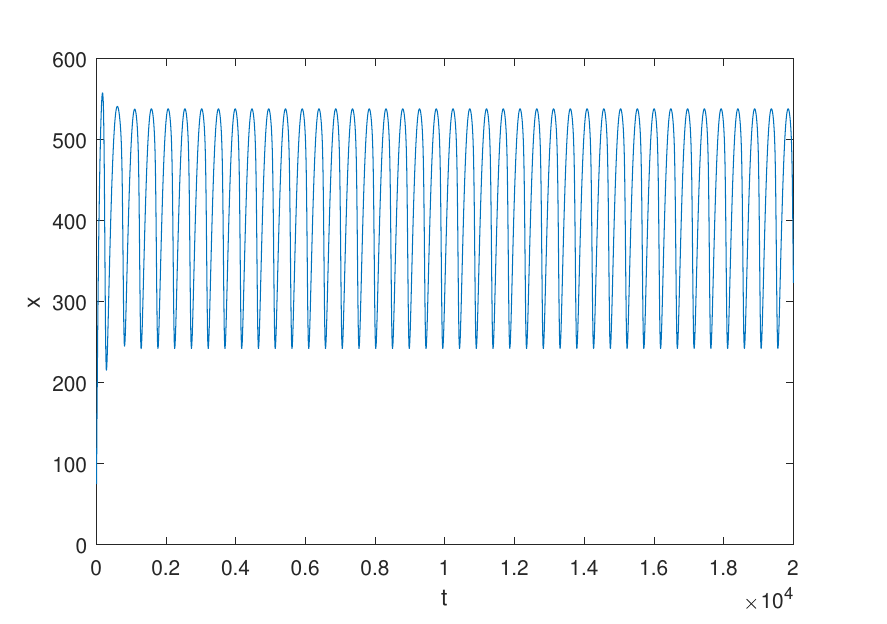}
		\caption{Temporal series of x(t)}
		\label{fig:image1}
	\end{subfigure}%
	\hfill
	\begin{subfigure}{0.3\textwidth}
		\centering
		\includegraphics[width=4.2cm,height=3.15cm]{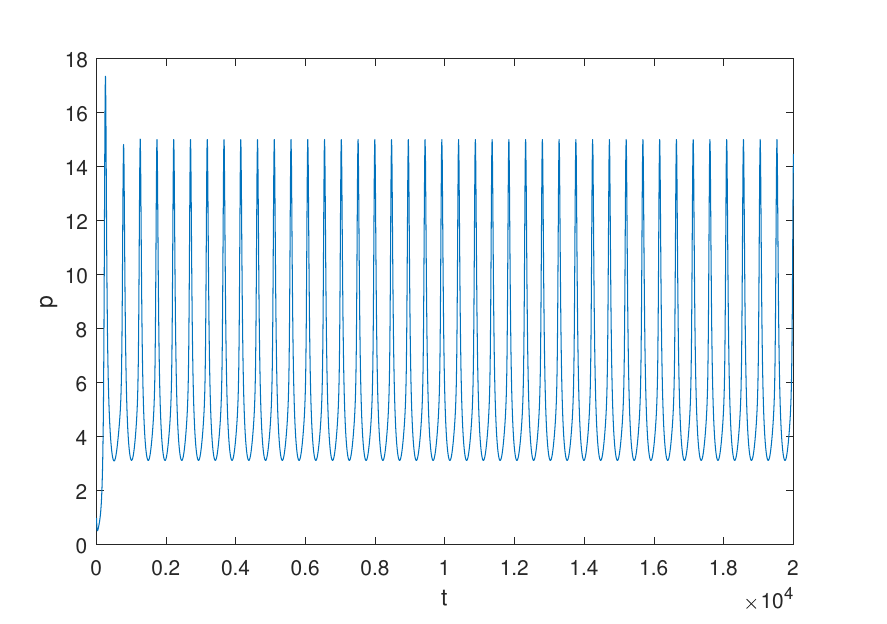}
		\caption{Temporal series of p(t)}
		\label{fig:image2}
	\end{subfigure}%
	\hfill
	\begin{subfigure}{0.3\textwidth}
		\centering
		\includegraphics[width=4.2cm,height=3.15cm]{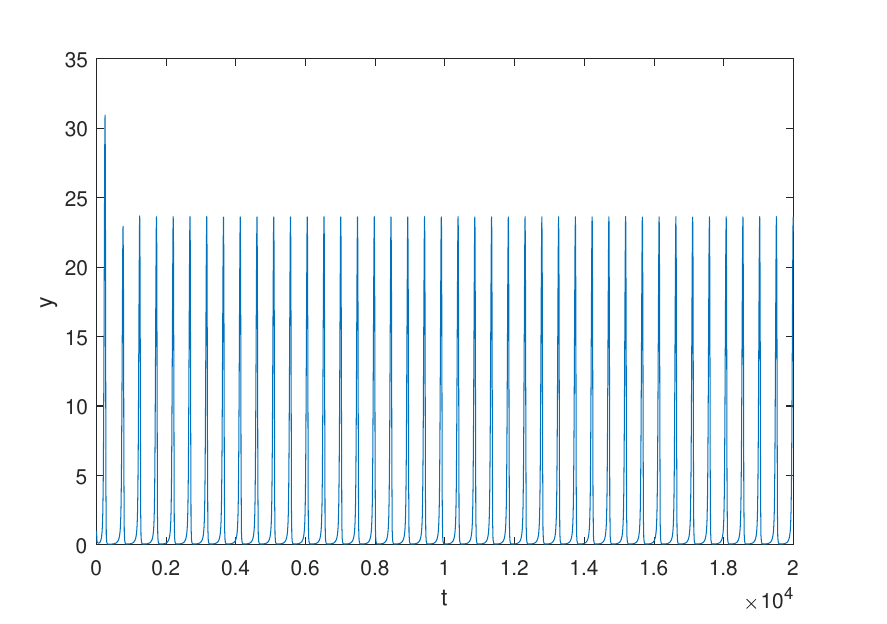}
		\caption{Temporal series of y(t)}
		\label{fig:image3}
	\end{subfigure}
	
	\begin{subfigure}{0.3\textwidth}
		\centering
		\includegraphics[width=4.2cm,height=3.15cm]{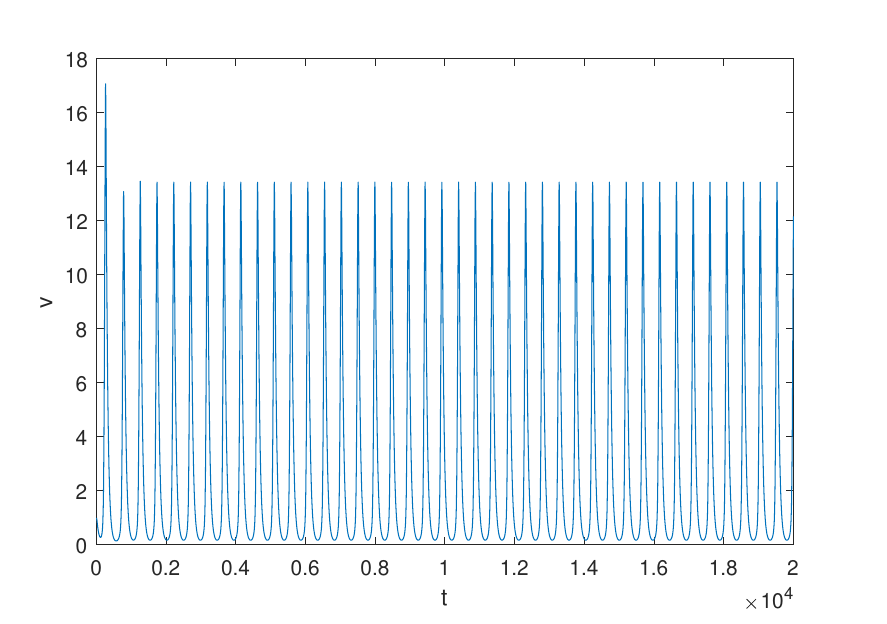}
		\caption{Temporal series of v(t)}
		\label{fig:image4}
	\end{subfigure}%
	\hfill
	\begin{subfigure}{0.3\textwidth}
		\centering
		\includegraphics[width=4.2cm,height=3.15cm]{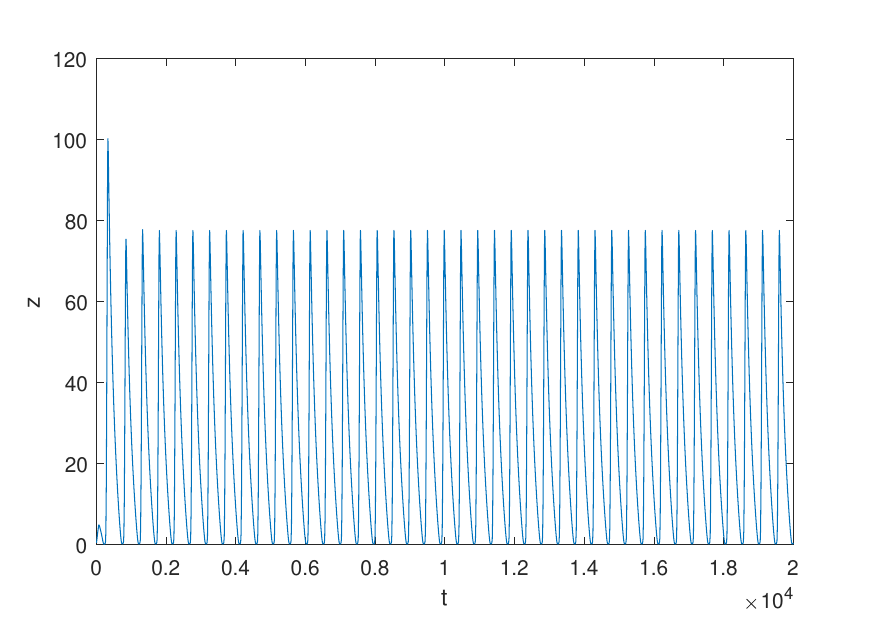}
		\caption{Temporal series of z(t)}
		\label{fig:image5}
	\end{subfigure}%
	\hfill
	\begin{subfigure}{0.3\textwidth}
		\centering
		\includegraphics[width=4.2cm,height=3.15cm]{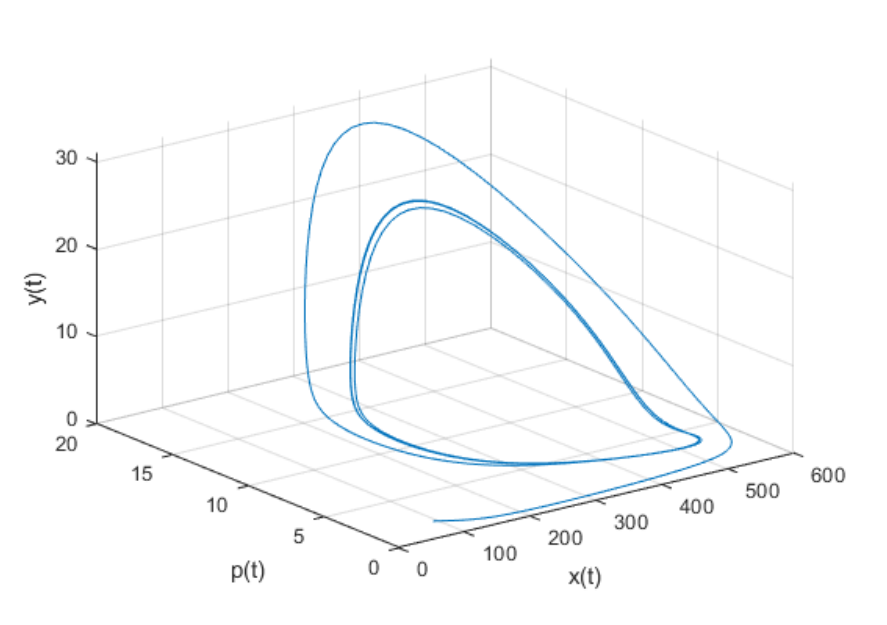}
		\caption{3D phase for x,p,y}
		\label{fig:image6}
	\end{subfigure}
	
	\caption{Simulation result of $E_2$ while $\tau_1=\tau_2=0.25$ and $\tau_3=60$}
	\label{fig}
\end{figure}
\newpage
Additionally, given $\mathcal{R}_0 = 5.8654 > 1$ and $\mathcal{R}_1 = 178.7590 > 1$, the system (\ref{1.2}) exhibits an immunity-activated equilibrium at $E_2 = (742.7828, 0.1328, 0.0486, 0.0486, 0.7423)$. We select $\tau_3 = 20$, $\tau_3 = 42$, $\tau_3 = 43$, and $\tau_3 = 60$, and the numerical simulation outcomes corresponding to the analysis are depicted in Figures 10, 11, 12, and 13. These results indicate that, when $\tau_1$, $\tau_2$, and $\tau_3$ are all nonzero, the Hopf bifurcation value $\tau_0\in (42, 43).$

\section{Conclusion}
In this paper, we extended an HIV model incorporating a latent period by introducing dual transmission routes, a saturated incidence rate, and three time-delay parameters (namely, the time $\tau_1$ required for the virus to infect uninfected cells, the time $\tau_2$ needed for viral replication, and the time $\tau_3$ for the occurrence of CTL immune response). We confirmed that the solutions of the model are nonnegative and bounded. Additionally, we analyzed three equilibrium states: the infection-free equilibrium $E_0$ determined by the basic reproduction number $\mathcal{R}_0$, the CTL-free infection equilibrium $E_1$ jointly determined by $\mathcal{R}_0$ and $\mathcal{R}_1$, and the endemic infection equilibrium $E_2$ with CTL response (when $\tau_3 = 0$) determined by $\mathcal{R}_1$.

Our research revealed that when $\mathcal{R}_0 < 1$, the equilibrium $E_0$ is globally asymptotically stable for all positive values of $\tau_1$, $\tau_2$, and $\tau_3$. When $\mathcal{R}_0 > 1$ and $\mathcal{R}_1 < 1$, the equilibrium $E_1$ is globally asymptotically stable for all positive $\tau_1$, $\tau_2$, and $\tau_3$. When $\mathcal{R}_1 > 1$ and $\tau_3 = 0$, the equilibrium $E_2$ is globally asymptotically stable for all positive $\tau_1$ and $\tau_2$. This indicates that $\tau_1$ and $\tau_2$ do not affect the stability of $E_0$ and $E_1$, while $\tau_3$ is crucial for the stability of $E_2$. There exists a critical threshold $\tau_0$. When $\tau_3 < \tau_0$, $E_2$ is stable; when $\tau_3 > \tau_0$, the system becomes unstable, leading to periodic oscillations and Hopf bifurcation.

From the perspectives of viral dynamics and immune response, these findings are of great significance. On the one hand, we found that $\mathcal{R}_0$ decreases monotonically with the increase of $\tau_1$ and $\tau_2$. This suggests that prolonging the latent period or the replication time helps to control the infection, highlighting the importance of considering time delays in viral dynamics modeling. Conversely, if $\tau_1$ and $\tau_2$ are ignored, that is, assuming that the virus enters the latent period and completes replication instantaneously, $\mathcal{R}_0$ will be overestimated. This, in turn, may lead to a misjudgment of the virus's transmission ability, affecting the formulation of prevention and control strategies and potentially resulting in resource waste or ineffective prevention and control.

On the other hand, the Hopf bifurcation phenomenon induced by $\tau_3$ reflects the dynamic imbalance between viral replication and immune clearance. When the delay in CTL response exceeds the critical threshold $\tau_0$, the immune system is unable to initiate effective clearance before the virus completes a critical replication cycle, leading to system instability. This finding not only clarifies the kinetic mechanism of viral persistence but also identifies the key time window for immune control, providing an important theoretical basis for optimizing treatment strategies.

This study elucidates the dynamic equilibrium between latency regulation, viral replication delay, and CTL immune response delay, providing a unified kinetic framework to explain clinically observed phenomena including viral latency maintenance, intermittent viral blips and immune control failure.  Notably, it clarifies the critical threshold effect of immune response delay in determining infection outcomes. However, our model adopts a deterministic framework, which has certain limitations in describing the stochastic activation process of HIV-latent infected cells. In fact, the activation of latent reservoir cells exhibits significant stochastic characteristics, and this stochastic activation process is highly likely to have a significant impact on the dynamic characteristics of viral rebound. Especially in cases of low viral load, stochastic effects may become the dominant factor affecting the system dynamics. Therefore, we plan to consider incorporating stochastic differential equations with white noise or continuous-time Markov chain models in subsequent studies to more accurately characterize the stochastic nature of latent infection, further improve our understanding of HIV infection dynamics, and provide more comprehensive and accurate scientific support for HIV prevention and control.

\vskip 20 pt
\noindent{\bf Acknowledgement}
\vskip 10 pt
The first author is partially supported by the National Key Research and Development Program of China 2020YFA0713100 and by the National Natural Science Foundation of China (Grant No. 11721101).
\nocite{*}
\bibliography{k}

\end{document}